\documentclass[a4paper]{amsart}
\usepackage{graphicx}
\usepackage{etex}
\usepackage{amsthm,stmaryrd}
\usepackage{amssymb}
\usepackage{array}
\usepackage{tabu}
\usepackage{arydshln}
\usepackage{epsfig}
\usepackage[usenames,dvipsnames]{color}
\usepackage{verbatim}
\usepackage{hyperref}	
\usepackage{float}
\usepackage{mathrsfs}
\usepackage[all]{xy}
\usepackage{xcolor}
\usepackage{enumerate}
\usepackage{tikz}   
\usepackage{changepage} 
\usepackage[T1]{fontenc}
\usetikzlibrary{arrows,calc,positioning,decorations.pathreplacing}

\usepackage{pst-node}
\usepackage{tikz-cd}

\newcommand{\Aut}{\operatorname{Aut}}

\newcommand{\Id}{\operatorname{Id}}

\newcommand{\Z}{\mathbb{Z}}
\newcommand{\Q}{\mathbb{Q}}

\newcommand{\C}{\mathbb{C}}
\newcommand{\N}{\mathbb{N}}

\newcommand{\Li}{\mathbb{L}}

\newcommand{\FF}{\mathbb{F}}
\newcommand{\F}{\mathscr{F}}

\newcommand{\U}{\mathscr{U}}

\newcommand{\unit}{\ensuremath{\mathbb{I}}}

\newcommand{\tcoev}{\stackrel{\longleftarrow}{\operatorname{coev}}}

\newcommand{\tev}{\stackrel{\longleftarrow}{\operatorname{ev}}}
\newcommand{\ev}{\stackrel{\longrightarrow}{\operatorname{ev}}}

\newcommand{\qs}{q}

\newcommand{\coev}{\stackrel{\longrightarrow}{\operatorname{coev}}}

\newtheorem{definition}{Definition}[subsection]
\newtheorem{theorem}[definition]{Theorem}
\newtheorem*{theorem*}{Theorem}

\newtheorem{proposition}[definition]{Proposition}
\newtheorem{lemma}[definition]{Lemma}
\newtheorem{remark}[definition]{Remark}
\newtheorem{notation}[definition]{Notation}

\newtheorem{corollary}[definition]{Corollary}
\newcounter{exo} \newcounter{numexercice}
\renewcommand{\theexo}{\arabic{exo}}

\newcounter{IntroCounter}
\stepcounter{IntroCounter}

\begin{document}
\title[Coloured Jones and Alexander invariants as intersections of cycles]{Coloured Jones and Alexander polynomials as topological intersections of cycles in configuration spaces} 
  \author{Cristina Ana-Maria Anghel}
\address{Mathematical Institute, University of Oxford, Oxford, United Kingdom} \email{palmeranghel@maths.ox.ac.uk} 
  \thanks{ }
  \date{\today}

\begin{abstract}
Coloured Jones and Alexander polynomials are sequences of quantum invariants recovering the Jones and Alexander polynomials at the first terms. We show that they can be seen conceptually in the same manner, using topological tools, as intersection pairings in covering spaces between explicit homology classes given by Lagrangian submanifolds. The main result proves that the $N^{th}$ coloured Jones polynomial and $N^{th}$ coloured Alexander polynomial come as different specialisations of an intersection pairing of the same homology classes over two variables, with extra framing corrections in each case. The first corollary explains Bigelow's picture for the Jones polynomial with noodles and forks from the quantum point of view. Secondly, we conclude that the $N^{th}$ coloured Alexander polynomial is a graded intersection pairing in a $ \mathbb Z \oplus \mathbb Z_N$-covering of the configuration space in the punctured disc. 

%In this paper we will present a homological model for Coloured Jones Polynomials. For each colour $N \in \N$, we will describe the invariant $J_N(L,q)$ as a graded intersection pairing of certain homology classes in a covering of the configuration space on the punctured disk. This construction is based on the Lawrence representation and a result due to Kohno that relates quantum representations and homological representations of the braid groups.  
\end{abstract}

\maketitle
\setcounter{tocdepth}{1}
 \tableofcontents
 {
 
\section{Introduction} 
The theory of quantum invariants for knots started with the discovery of the Jones polynomial. After that, Reshetikhin and Turaev developed an algebraic tool which starts with a quantum group and leads to a link invariant. Using this algebraic method, the representation theory of $U_q(sl(2))$ leads to a family of link invariants $\{J_N(L,q) \in \mathbb Z[q^{\pm 1}]\}_{N \in \mathbb N}$ called coloured Jones polynomials. The first term of this sequence is the original Jones polynomial. On the other hand, the quantum group at roots of unity $U_{\xi}(sl(2))$, leads to a sequence of invariants, called coloured Alexander polynomials, having the original Alexander invariant as the first term.  On the topological side, R. Lawrence (\cite{Law},\cite{Law1}) introduced a sequence of homological braid group representations based on coverings of configurations spaces and using these, Bigelow and Lawrence gave a homological model for the original Jones polynomial, using the skein nature of the invariant for the proof. Later on, Kohno and Ito (\cite{Koh}, \cite{Koh2} \cite{Ito},\cite{Ito2}) presented an identification between highest weight quantum representations of the braid group and the homological Lawrence representations. 

We are interested in questions concerning topological models for these quantum invariants, using homological braid group actions on the homology of configuration spaces. In \cite{Cr1} we presented a topological model for all coloured Jones polynomials, showing that they are graded intersection pairings between two homology classes in a covering of the configuration space in the punctured disc. This result used the formulas from \cite{Ito}. However, even if the definition of these homology classes was explicit, it involved functions that are difficult to deal with from the computational point of view. 

Concerning the representation theory of quantum groups at roots of unity, in \cite{Ito2}, Ito suggested an identification of highest weight representations at roots of unity with a quotient of the Lawrence representation. Then he concluded a homological model for the coloured Alexander invariants as a sum of traces of these truncated Lawrence representations.  Based on Ito's identification at roots of unity, we showed in \cite{Cr2} a topological model for the coloured Alexander invariants as graded intersection pairings between two homology classes in the truncated Lawrence representation, using a quotient of the homology of the covering of the configuration space in the punctured disc. Out of these two topological models, we reached three precise questions, as follows.
\begin{itemize}
\item{\bf Question 1} Find topological models with explicit homology classes.\\ 
\item{\bf Question 2} What is the meaning of the truncation which occurs at the homological level in the previous models for $U_q(sl(2))-$quantum invariants at root of unity?\\ 
\item{\bf Question 3} What is the explanation from the quantum point of view of Bigelow's model for the Jones polynomial?\\
\end{itemize}
In this paper we answer these problems. Our main result shows that the $N^{th}$ coloured Jones and coloured Alexander polynomials come from the intersection pairing of the {\em same concrete homology classes over two variables} conveniently specialised, with an extra framing coefficient.

%we can see the both coloured Jones polynomials and coloured Alexander polynomials in a topological way, specialised from the Lawrence representation over two variables.

 Some consequences of the main theorem are the following.
\begin{itemize}
\item This result explains for the case of the Jones polynomial, why Bigelow's noodles and forks appear naturally from the quantum world.  \\
\item Moreover, it explains why Bigelow's model still gives the Jones polynomial when one removes one noodle.\\
\item We show a topological model for the $N^{th}$ coloured Alexander polynomial as an intersection pairing in a $\Z \oplus \Z_N-$covering of the configuration space, {\em without any further truncation or specialisation}.
\end{itemize}

\subsection{Description of the topological tools} 
Let $C_{n,m}$ be the unordered configuration space of $m$ points in the $n$-punctured disc. We will use the following tools for our construction:
\begin{enumerate}
\item sequence of Lawrence representations $H_{n,m}$ which are $\Z[x^{\pm 1},d^{\pm 1}]$-modules and carry a $B_n-$action
(defined from the Borel-Moore homology of a $\Z \oplus \Z $-covering of $C_{n,m}$- definition \ref{T1})
\item sequence of dual Lawrence representations $H^{\partial}_{n,m}$ ( notation \ref{T2} )\\ 
(defined using the homology relative to the boundary of the same covering) 
\item certain topological intersection pairings $< , >$ between the Lawrence representations and their dual representations 
$$< , >:H_{n,m} \otimes H^{\partial}_{n,m}\rightarrow \Z[x^{\pm1}, d^{\pm1}] \ \ \ \ \ \ \ \ \ \ \ \ \ \ \ \  ( \text{ definition }\ref{T3} ).$$
\end{enumerate}
 \
 In \cite{Martel}, Martel presented a version of Kohno's identification for the generic quantum group $U_q(sl(2))$ with more explicit bases in the Lawrence representation. In the following we will use this identification. First of all, we start on the algebraic side with quantum representations on weight spaces from tensor powers of the generic Verma module over two variables. We remark that when we specialise to one variable with $q$ generic or $q=\xi_N$ a root of unity, we can use weight spaces in a tensor power of an $N$-dimensional subspace inside the Verma module. For the generic version this is not surprising, however, for the root of unity case this differs from the usual construction of the coloured Alexander invariants. After we study the precise form of the coefficients of the $R$-matrix after specialisations, we show that we can see both coloured Jones invariants and coloured Alexander invariants from a specific weight space inside the tensor power of the Verma module over two variables. Then, we use identifications with the homological Lawrence representation and we construct certain homology classes. In the last part, we show that the graded intersection pairing between these classes leads by two different specialisations to the two sequences of quantum invariants, namely coloured Jones polynomials and coloured Alexander polynomials.
\begin{notation}
We will use the following specialisations:\\
1) Generic case $(q \text{ generic} ,  \lambda=N-1 \in \N)$
$$\psi_{q,\lambda}: \Z[x^{\pm},d^{\pm}]\rightarrow \Z[q^{\pm}]$$
2) Root of unity case $(q=\xi_N=e^{\frac{2\pi i}{2N}} , \lambda\in \C)$
$$\psi_{\xi_N,\lambda}: \Z[x^{\pm},d^{\pm}]\rightarrow \C.$$
given by the formula:
\vspace{-5mm}
\begin{equation*}
\begin{cases}
\psi_{q,\lambda}(x)=q^{2\lambda}\\
\psi_{q,\lambda}(d)=q^{-2}.
\end{cases}
\end{equation*}

\end{notation}
%The power of the root of unity situation, reflects on the geometric picture mainly on the specialisation of the local system, using the root of unity. 
We present the detailed construction of the homology classes in definition \ref{D:4}. They are certain lifts of immersed Lagrangians in the configuration space, as in the picture below.
\vspace{0mm}
$$ {  \bf \Huge \color{red} H_{2n-1,(n-1)(N-1)}}  \ \ \ \ \ \ \ \ \ \ \ \ \ \ \ \ \ \ \ \ \ \ \ \ \ \ \ \ \ \ \ \ \ \ \ \ \ \ \ {\bf \Huge \color{green} H^{\partial}_{2n-1,(n-1)(N-1)}}$$
$$ {\Huge \color{red} \mathscr E_n^{N}= \hspace{-3mm}\sum_{{ i_1,...,i_{n-1}=0}}^{N-1} \hspace{-3mm} d^{\sum_{k=1}^{n-1} i_k} \cdot \tilde{ \mathscr U}_{0,i_1,...,i_{n-1},N-1-i_{n-1},...,N-1-i_{1}}} \ \ \ \ \ \ \ \ \ \ \ \ \ \ \ \ \ \ \ \ \ \ \ \ \ \  {\color{green} \mathscr G_{n}^N} \ \ \ \ \ \ \ \ \ \ \ \ \ \ \ \ \ \ \ \ \ \ \ \ \ \ \ \ \ \ \ \ \ \ \  $$

\vspace{-3mm}
$\hspace{60 mm} \downarrow \text{lifts}$
\vspace{-2mm}
\begin{figure}[H]
\centering
\includegraphics[scale=0.44]{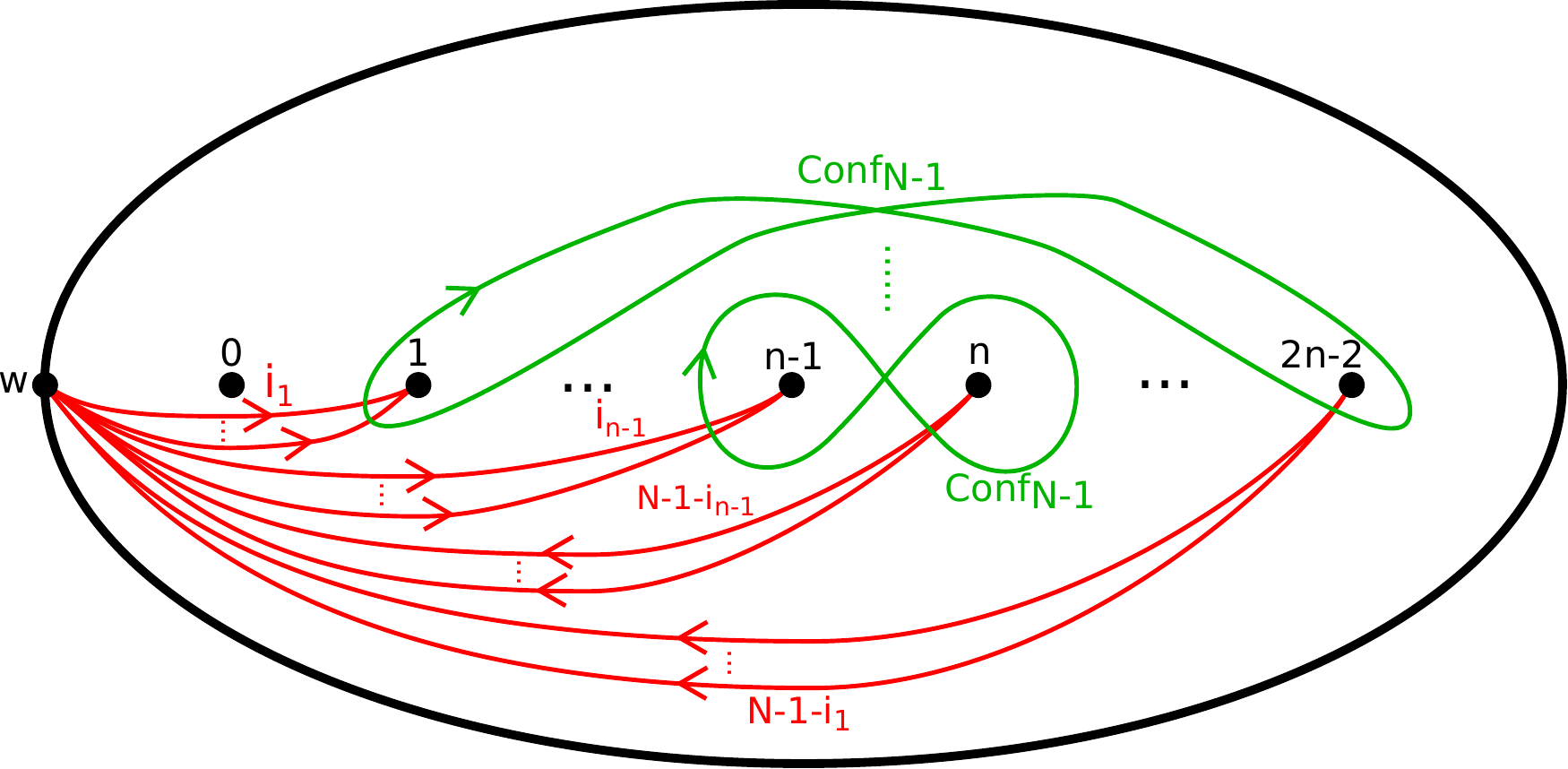}
\vspace{-4mm}
$${\color{red} \bar{U}_{0,i_1,...,i_{n-1},N-1-i_{n-1},...,N-1-i_{1}}} \ \ \ \ \ \ \ \ \ \ \ \ \ \ \ \ \ \ \ \ \ \ \ \ \ \ \ \ \ \ \ \ \ \ \ \ \ \ \ \  \ \ \ \ \ \ \ \ \ \ \color{green} {G_n^N} \ \ \ \ \ \ \ \ \ \ \ \ \ \ \ \ \ \ \ \ \ \ \ 
$$
\vspace{-4mm}
\caption{Homology classes (definition \ref{D:4})}
\label{fig:classes}
\end{figure}
\begin{theorem}(Coloured Jones and Alexander invariants from generic intersection pairings)\label{THEOREM}

Let us consider the homology class $\mathscr E_n^N \in H_{2n-1,(n-1)(N-1)}$ which is the linear combinations of the submanifolds $\tilde{\mathscr U}$ presented above. Let $\mathscr G_{n}^N \in H^{\partial}_{2n-1,(n-1)(N-1)}$ be a certain lift of the product of figure-eight configuration spaces from figure \ref{fig:classes}. 

Then, if an oriented knot $L$ is a closure of a braid with $n$ strands $\beta_n \in B_n$, we have the following models:
\begin{equation}\label{THEOREM1}
\ J_N(L,q) \ = \ q^{-(N-1)w(\beta_n)} \cdot \ q^{(N-1)(n-1)} \ \ <(\beta_{n} \cup {\mathbb I}_{n-1} ) \ { \color{red} \mathscr E_n^N}, {\color{green} \mathscr G_n^N}> |_{\psi_{q,N-1}}
\end{equation} 
\begin{equation}\label{THEOREM2}
\Phi_{N}(L,\lambda)={\xi_N}^{(N-1)\lambda w(\beta_n)} \cdot {\xi_N}^{\lambda (1-N)(n-1)} \ <(\beta_{n} \cup {\mathbb I}_{n-1} ) \ { \color{red} \mathscr E_n^{N}}, {\color{green} \mathscr G_n^N}> |_{\psi_{\xi_N,\lambda}}. 
\end{equation}
Here $w(\beta_n)$ is the writhe of the braid and ${\mathbb I}_{n-1}$ is the trivial braid with $n-1$ strands.
\end{theorem}
This model answers the first question and we see that the classes that give the coloured Alexander and Jones polynomials are given by explicit linear combinations lifts of Lagrangian submanifolds in the configuration space.  
%\subsection{Strategy of the proof}

\subsection{Unified model coming from two variables}

\

\noindent 
The model presented above, together with the fact that the homology class $\mathscr E_n^N$ is defined over two variables, enable us to see both coloured Alexander and Jones polynomials conceptually in the same way, despite their initial different descriptions from the representation theory point of view. 
\begin{corollary}(Unified model giving coloured Jones and Alexander invariants)\\ \label{C:unified}
Let us fix a colour $N \in \N$ and $L$ an oriented knot given by the closure of $\beta_n \in B_n$. Let us consider the following polynomial in two variables:
$$\mathscr I_N(\beta_n):=<(\beta_{n} \cup {\mathbb I}_{n-1} ){ \color{red} \mathscr E_n^N}, {\color{green} \mathscr G_n^N}> \in \Z[x^{\pm 1},d^{\pm 1}].$$
Then, we recover the two quantum invariants from $\mathscr I_N$, through the corresponding specialisations as follows:
\begin{equation}
J_N(L,q)= q^{-(N-1)w(\beta_n)} \cdot \ q^{(N-1)(n-1)} \ \ \mathscr I_N(\beta_n)|_{\psi_{q,N-1}}.
\end{equation} 
\begin{equation}
\Phi_{N}(L,\lambda)={\xi_N}^{(N-1)\lambda w(\beta_n)} \cdot {\xi_N}^{\lambda (1-N)(n-1)}  \mathscr I_N(\beta_n) |_{\psi_{\xi_N,\lambda}}. \ \ 
\end{equation}

\end{corollary}

\begin{figure}[H]
\centering
\includegraphics[scale=0.27]{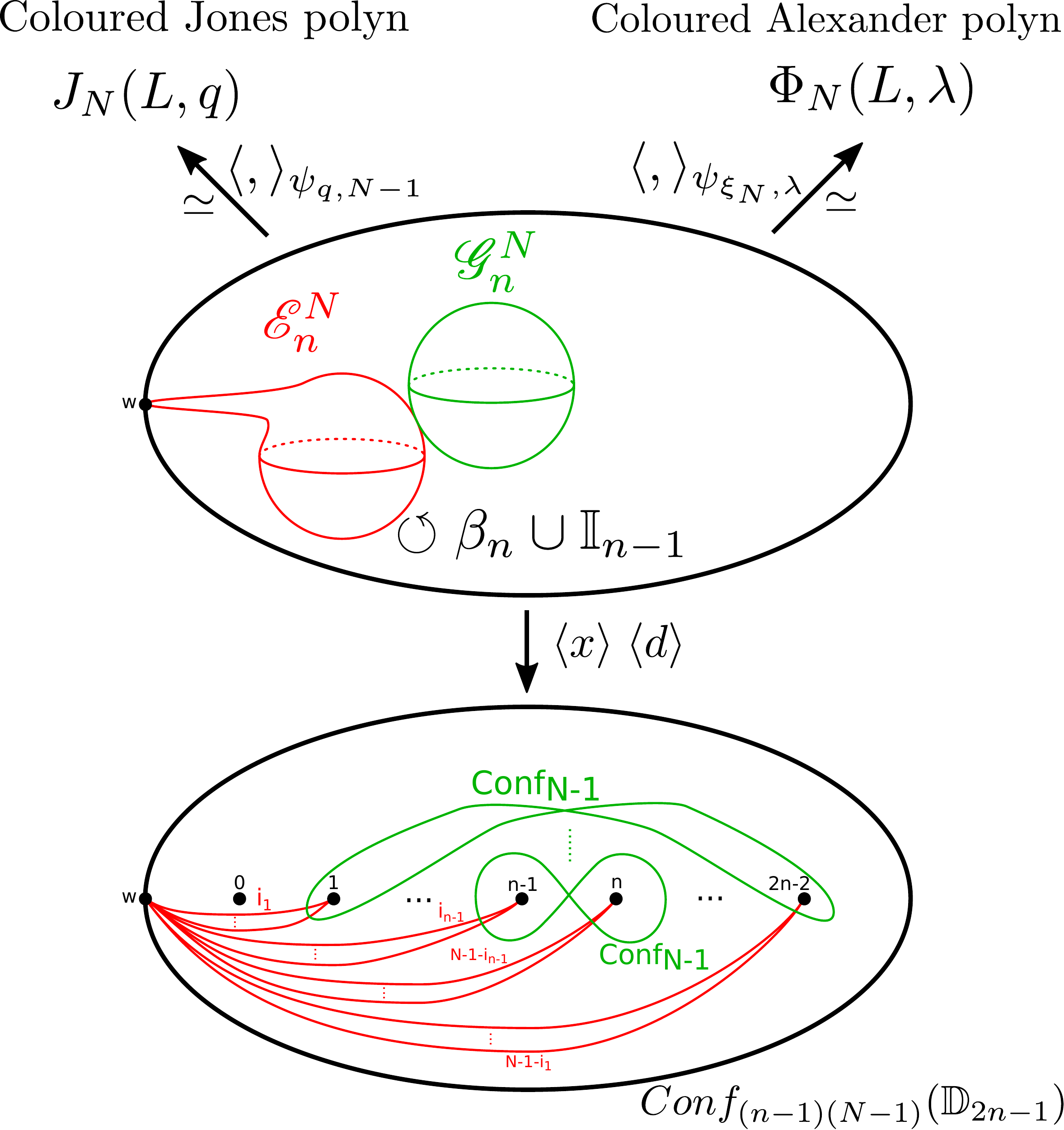}
\caption{}
\end{figure}
\vspace{-5mm}
This shows that both $N^{th}$ coloured Jones and coloured Alexander invariants come from a pairing $\mathscr I_N(\beta_n)$ which takes values in two variables. Then in order to obtain one or another, we have to correct with the corresponding framing coefficient and the last step is to specialise via $\psi_{q,N-1}$ or $\psi_{\xi_N,\lambda}$ respectively. 
\subsection{Relation to Bigelow's model for the original Jones polynomial}

\

\noindent
In \cite{Big}, Bigelow shows an intersection model for the Jones polynomial, using forks and noodles. More precisely, he consider the corresponding pairing between forks and noodles and shows that it is a knot invariant which satisfies the skein relation which characterises the Jones polynomial. 

However, the way in which this model corresponds to the definition of the Jones polynomial as a quantum invariant (from representation theory) remained still mysterious. In \cite{Cr1}, we explained that the configuration space that appears in his model comes naturally from the representation theory. However,  the question of what predicted the forks and noodles still remained open.

In Section \ref{SCJ:6}, we show why the figure eights are such a natural choice and how the pairing between them and forks fits with the quantum side.

 Moreover, it is known in the literature that one can obtain the Jones polynomial from Bigelow's model, if one removes one figure-eight. This paper gives an explanation for this fact, which corresponds to the property that on the quantum side, one can obtain the (coloured) Jones polynomials both from quantum traces as well as partial quantum traces. In the current paper we use the definition with partial traces, meaning that we cut the knot before applying the quantum definition. 
\begin{corollary}(Recovering Bigelow's model for the Jones polynomial)\\
%The case $N=2$ from corresponding to the original Jones polynomial.  
The model for the Jones polynomial from Theorem \ref{THEOREM} with classes $\mathscr F_n^2$ and $\mathscr G_n^2$ corresponds to Bigelow's model with noodles and forks (taking into account that he works with plat closures and we have normal closures).
 
More precisely the class $\mathscr F_n^2$ is obtained from the forks where we push the middle point of each segment towards the boundary point $w$ and the class $\mathscr G_n^2$ is given exactly by the noodles from Bigelow's picture. 
\end{corollary}

\subsection{Truncation encoded in the specialisation of coefficients}

\

In this part, we discuss about the relations between the previous topological models for coloured Alexander polynomials and this new model. We start with the second question that we presented in the begining, concerning the importance of the truncation of the Lawrence representation which appeared in \cite{Ito2} and \cite{Cr1}. 

These previous models for the coloured Alexander polynomials used certain quotients at the homological level, called truncations, and also specialisations of coefficients. In particular, in \cite{Ito2}, it is emphasised that there were no direct topological models for coloured Alexander polynomials with the colour $N>2$ and that the truncation of the Lawrence representation was one of the obstructions. In our model, we still choose the braid representative, but everything else including the braid group action and the intersection pairing are topological. Moreover, Theorem \ref{THEOREM} shows that the truncation part at the root of unity is reflected on the homological side directly by the specialisation $\psi_{\xi_N,\lambda}$. In other words, this specialisation is powerful enough to contain the truncation which occurred in the previous models.
\subsection{Coloured Alexander invariants from $\Z\oplus \Z_N$ coverings} 

\

In the last part, in section \ref{S:8}, we consider homological representations of the braid group on a $\Z \oplus \Z_N$-covering of the configuration space in the punctured disc. More precisely, we define $H^N_{n,m}$ to be a version of the Lawrence representation defined using the local system $\Z \oplus \Z_N$ (by projecting the second component of the $\Z\oplus \Z$ local system modulo $N$). Then, we have a corresponding intersection pairing, which we denote by $< , >_N$ (proposition \ref{P:5}):
\begin{equation}
< , >_N:H^N_{n,m} \otimes H^{\partial,N}_{n,m} \rightarrow \Z[s^{\pm 1}, \xi_N^{\pm 2}]
\end{equation}

\begin{theorem}(Coloured Alexander invariants from $\Z\oplus\Z_N$ covering spaces)\label{C:Alex}  
Let us consider the homology classes
$$\bar{\mathscr E}_n^{\xi_N} \in H^N_{2n-1,(n-1)(N-1)} \ \text{ and } \  \bar{\mathscr G}_{n}^N \in H^{\partial,N}_{2n-1,(n-1)(N-1)}$$ given by the same geometric submanifolds in the base space as the ones from Theorem \ref{THEOREM}, lifted in the covering $C^N_{2n-1,(n-1)(N-1)}$. Then the $N^{th}$ coloured Alexander invariant is given by the following intersection (if we identify $s$ with $\xi_N^{\lambda}$):
$$\Phi_{N}(L,\lambda)={\xi_N}^{(N-1)\lambda w(\beta_n)} \cdot {\xi_N}^{\lambda (1-N)(n-1)}<(\beta_{n} \cup {\mathbb I}_{n-1} ){ \color{red} \bar{\mathscr E}_n^{\xi_N}}, {\color{green} \bar{\mathscr G}_n^N}>_N.$$
\end{theorem}

\subsection{Question-towards categorification} The main result from this article\\
 presents the explicit form of the Lagrangian manifolds that are naturally predicted from the quantum side and lead to the two sequences of quantum  invariants.  This paper comes as a pair with a sequel paper, where we study the geometry of these classes and we present a model whose corresponding Lagrangians are more suitable for categorifications. More specifically, in that paper, the submanifolds in the base space are already embedded rather than immersed and we encode geometrically the coefficients that occur in the first homology class.
\subsection{Question-towards loop expansions}
In \cite{GM}, Gukov and Manolescu introduced a two variable power series constructed from knot complements and conjectured that this is a knot invariant recovering the loop expansion of the coloured Jones polynomials. It would be interesting to investigate asymptotic limits of the unified model over two variables from \ref{C:unified} and relations with the loop expansion of coloured Jones polynomials. Pursuing this line, we are interested in studying connections between this model and the Gukov-Manolescu power series.

\subsection*{Structure of the paper} This paper has 7 main sections. In Section \ref{2}, we present the version of the quantum group as well as the representation theory that we work with. Then, in Section \ref{3} we discuss homological tools, which are given by different versions of the Lawrence representation. In Section \ref{4}, we construct an algebraic set-up over two variables which leads to the coloured Jones and Alexander polynomials, through different specialisations. In Section \ref{5} we define the main objects, namely the two homology classes given by Lagrangian submanifolds. Section \ref{6} is devoted to the proof of the intersection formula for the coloured Jones polynomials, presented in equation \eqref{THEOREM1}. Further on, in Section \ref{7} we prove the topological model for the coloured Alexander polynomials from equation \eqref{THEOREM2} and conclude Theorem \ref{THEOREM}. 
In the last part, Section \ref{S:8}, we define new homological tools and prove that the $N^{th}$ coloured Alexander invariant is an intersection pairing in a $\Z \oplus \Z_N$ covering of a configuration space, as in Theorem \ref{C:Alex}. 
\subsection*{Acknowledgements} 
 I would like to thank Professor Christian Blanchet for many beautiful discussions and for asking me two of the main questions. Also, I would like to thank Christine Lescop, Jacob Rasmussen, Alexis Virelizier and Emmanuel Wagner for useful conversations.  
 
 This paper was prepared at the University of Oxford, and I acknowledge the support of the European Research Council (ERC) under the European Union's Horizon 2020 research and innovation programme (grant agreement No 674978).
 
 \clearpage
\section{Notations}
Throughout the paper, we will use various specialisations of coefficients, in order to interplay between the algebraic side given by the representation theory and the homological side represented by homology of covering spaces. We will use the following definition.
\begin{notation}(Specialisation)\label{N:spec}\\ 
Let us start with a free module $M$ over a ring $R$ and let $\mathscr B$ be a basis of the module $M$.
If $R'$ is another ring, let us assume that we fix a specialisation of the coefficients, meaning a morphism:
$$\psi: R \rightarrow R'.$$
Then, the specialisation of the module $M$ by the function $\psi$ is the following module: $$M|_{\psi}:=M \otimes_{R} R'$$ which has the corresponding basis over $R'$: $$\mathscr B_{M|_{\psi}}:=\mathscr B \otimes_{R}1 \in M|_{\psi}. $$
\end{notation}
In our case, we will use the following notations.
\begin{notation}(Specialisations of coefficients)
\begin{equation*}
\begin{cases}
\gamma: \Z[x^{\pm},d^{\pm}] \rightarrow \Z[q^{\pm 1}, s^{\pm 1}]\\
\gamma(x)= s^{2}; \ \ \gamma(d)=q^{-2}.
\end{cases}
\end{equation*}
\begin{equation*}
\begin{aligned}
&\begin{cases}
\psi_{q,N-1}: \Z[x^{\pm},d^{\pm}]\rightarrow \Z[q^{\pm}]\\
\psi_{q,N-1}(x)= q^{2(N-1)}; \ \ \psi_{q,N-1}(d)=q^{-2}.
\end{cases}
\begin{cases}
\psi_{\xi_N,\lambda}: \Z[x^{\pm},d^{\pm}]\rightarrow \C\\
\alpha_{\lambda}(x)= {\xi_N}^{2 \lambda}; \ \ \alpha_{\lambda}(d)={\xi_N}^{-2}.
\end{cases}\\
&\begin{cases}
\eta_{q,N-1}: \Z[q^{\pm 1},s^{\pm 1}]\rightarrow \Z[q^{\pm 1}]\\
\psi_{q,N-1}(s)= q^{(N-1)}.
\end{cases}
\hspace{17mm}\begin{cases}
\eta_{\xi_N,\lambda}: \Z[q^{\pm 1},s^{\pm 1}]\rightarrow \C\\
\eta_{\xi_N,\lambda}(s)= \xi_N^{\lambda}; \ \ \psi_{q,N-1}(q)=\xi_N.
\end{cases}
\end{aligned}
\end{equation*}
\end{notation}

\definecolor{bittersweet}{rgb}{1.0, 0.44, 0.37}
\definecolor{brickred}{rgb}{0.8, 0.25, 0.33}
\begin{figure}[H]
\begin{center}
\begin{tikzpicture}
[x=1.1mm,y=1.1mm]

% Nodes of the diagram
\node (b1)  [color=blue] at (-30,30)    {${\mathbb Z[x^{\pm1},d^{\pm1}]}$};
\node (b2) [color=orange] at (0,30)   {${\mathbb Z[s^{\pm1},q^{\pm1}]}$};
\node (b3) [color=red]   at (30,30)   {${\tilde{\Li}_N=\mathbb Z[s^{\pm1},q^{\pm1}]\left(I_N \right)^{-1}}$};
%\node (b4') [color=brickred]  at (48,30)   {$\hookrightarrow$};
%\node (b4) [color=brickred]  at (55,30)   {$\Q(q,s)$};
\node (b5) [color=green]  at (0,2)   {$\boldsymbol{\mathbb C}$};
\node (b5') [color=green]  at (0,-5)   {$\Z[q^{\pm 1}]$};

\node (b6) [color=blue]  at (-22,27)   {$\boldsymbol {x}$};
\node (b6') [color=blue]  at (-22,24)   {$\boldsymbol {d}$};
\node (b7) [color=orange]  at (-7,27)   {$\boldsymbol {s^{2}}$};
\node (b7') [color=orange]  at (-7,24)   {$\boldsymbol {q^{-2}}$};
\node (b8) [color=orange]  at (-6,20)   {$\boldsymbol {q}$};
\node (b8') [color=orange]  at (-2,20)   {$\boldsymbol {s}$};
\node (b9) [color=green]  at (-6,10)   {$\boldsymbol {\xi}$};
\node (b9') [color=green]  at (-2,10)   {$\boldsymbol {\xi^{\lambda}}$};

% Arrows
%\draw[->]             (cbottom) to node[left,yshift=-2mm,font=\small]{${1) \tCoev}^{\otimes n}_{V_N}$} (b1);
%\draw[<->,color=green, dashed]   (b1'''')      to node[right,font=\small, yshift=3mm]{}                           (b3'');
%\draw[<->,color=red, dashed]   (b1''')      to node[right,font=\small, yshift=3mm]{}                           (b3');
%\draw[<->,color=orange, dashed]   (b1'')      to node[right,font=\small, yshift=3mm]{}                           (b2');

\draw[->,color=blue, dashed]   (b6)      to node[right,font=\small, yshift=3mm]{}                           (b7);
\draw[->,color=blue,dashed]   (b6')      to node[right,font=\small, yshift=3mm]{}                           (b7');
\draw[->,color=blue]   (b1)      to node[right,font=\small, yshift=3mm]{${\gamma}$}                           (b2);

\draw[->,color=red]             (b2)      to node[left,font=\small, yshift=3 mm]{$\iota_N$}   (b3);
%\draw[->,color=teal]             (b3)      to node[left,font=\small]{$i$}   (b4);
\draw[->, color=cyan]             (b1)     to[in=160,out=20] node[right,yshift=-3mm,xshift=-5mm,font=\small]{$ {\tilde{\gamma}}_N$}   (b3);
\draw[->,color=red]             (b3)      to node[right,font=\small, xshift=1mm]{${\tilde{\eta}_{\xi,\lambda}}$}   (b5);
\draw[->,color=red,]             (b3)      to [in=30,out=-90] node[right,font=\small, xshift=1mm]{${\tilde{\eta}_{q,N-1}}$}   (b5');

%\draw[->,color=orange]   (b1)      to node[right,font=\small]{$\gamma$}                        (b5);
\draw[->,color=orange,thick]   (b2)      to node[right,font=\small]{${\eta_{\xi,\lambda}}$}                        (b5);
\draw[->,color=green]   (b1)      to node[left,font=\small]{${\psi_{\xi,\lambda}}$}                        (b5);
\draw[->,color=green]   (b1)     to[in=160,out=-90] node[left,font=\small]{$\boldsymbol{\psi_{q,N-1}}$}                        (b5');
\draw[->,color=orange, dashed]   (b8)      to node[right,font=\small, yshift=3mm]{}                           (b9);
\draw[->,color=orange,dashed]   (b8')      to node[right,font=\small, yshift=3mm]{}                           (b9');
\end{tikzpicture}
\end{center}
\caption{Specialisations}
\label{fig:7}  
\end{figure}
\clearpage
\section{Representation theory of $U_q(sl(2))$}\label{2}
}

%\subsection{$U_q(sl(2))$ and its representations}
Let $\qs, s$ parameters and consider the ring 
$$\Li_s:=\Z[\qs^{\pm1},s^{\pm1}].$$
\begin{definition}(Quantum group over two variables with divided powers \cite{JK} )

Consider the quantum enveloping algebra $U_q(sl(2))$, to be the algebra over $\Li_s$ generated by the elements $\{ E,F^{(n)}, K^{\pm1} | \ n \in \mathbb N^{*} \}$ with the following relations:
\begin{equation*}
\begin{cases}
KK^{-1}=K^{-1}K=1; \ \ \ KE=\qs^2EK; & \ \ \ KF^{(n)}=\qs^{-2n} F^{(n)}K;\\
F^{(n)}F^{(m)}= {n+m \brack n}_{\qs} F^{(n+m)}&\\
[E,F^{(n+1)}]=F^{(n)}(q^{-n}K-q^{n}K^{-1}).
\end{cases}
\end{equation*}
\end{definition}
Then, one has that $U_q(sl(2))$ is a Hopf algebra with the following comultiplication, counit and antipode:
\begin{equation*}
\begin{cases}
\Delta(E)= E\otimes K+ 1 \otimes E  & S(E)=-EK^{-1}\\
\Delta (F^{(n)}) = \sum_{j=0}^n q^{-j(n-j)}K^{j-n}F^{(j)}\otimes F^{(n-j)}& S(F^{(n)})=(-1)^{n}q^{n(n-1)}K^{n}F^{(n)}\\
\Delta(K)=K\otimes K & S(K)=K^{-1}\\
\Delta(K^{-1})=K^{-1}\otimes K^{-1} &  S(K^{-1})=K.
\end{cases}
\end{equation*}

 We will use the following notations:
$$ \{ x \} :=\qs^x-\qs^{-x} \ \ \ \ [x]_{\qs}:= \frac{\qs^x-\qs^{-x}}{\qs-\qs^{-1}}$$
$$[n]_\qs!=[1]_\qs[2]_\qs...[n]_\qs$$
$${n \brack j}_\qs=\frac{[n]_\qs!}{[n-j]_\qs! [j]_\qs!}.$$

\begin{definition} (The Verma module)

Consider $\hat{V}$ be the $\Li_s$-module generated by an infinite family of vectors $\{v_0, v_1,...\}$. The following relations define an $U_q(sl(2))$ action on $\hat{V}$ (see \cite{JK}):
\begin{equation}\label{action}  
\begin{cases}
Kv_i=s\qs^{-2i}v_i,\\
Ev_i=v_{i-1},\\
F^{(n)}v_i = {n+i \brack i}_\qs \prod _{k=0}^{n-1} (s\qs^{-k-i}-s^{-1} \qs^{k+i}) v_{i+n}.
\end{cases}
\end{equation}
\end{definition}

%\subsection{Specialisations}

\
In the sequel, we will use certain specialisations of the previous quantum group.
\begin{definition}
We consider two type of specialisations of the coefficients, where we specialise the highest weight using $q^{\lambda}$:\\
{ Case a) $(q \text{ generic}, \ \lambda=N-1 \in \N)$:}
\begin{equation}
\begin{cases}
 \eta_{q,\lambda}:\Z[q^{\pm 1},s^{\pm 1}]\rightarrow\Z[q^{\pm 1}]\\
 \eta_{q,\lambda}(s)= q^{\lambda }.
\end{cases}
\end{equation}
{ Case b) $(q=\xi_N=e^{\frac{2 \pi i}{2N}}, \ \lambda \in \C \text{ generic})$:}
\begin{equation}
\begin{cases}
 \eta_{\xi_N,\lambda}:\Z[q^{\pm 1},s^{\pm 1}]\rightarrow\C \\
\eta_{\xi_N,\lambda}(s)= {\xi}^{\lambda }_N.
\end{cases}
\end{equation}

\end{definition}
Using these specialisations, we will consider the corresponding specialised quantum groups and their representation theory. We obtain the following:
{\footnotesize
\begin{center}
\label{tabel}
\begin{tabular} { |c|c|c|c| } 
 \hline
 Ring & Quantum Group & Representations &  Specialisations \\ 
\hline                                    
$ \Li_s=\Z[q^{\pm},s^{\pm}] $ & $ U_q(sl(2))$  & $\hat{V}$  & $ q,s \text{  param }  $ \\
\hline
 $ \Li=\Z[q^{\pm}] $ & $\U=U_q(sl(2))\otimes_{\eta_{q,N-1}}\Z[q^{\pm 1}]$  &  $\hat{V}_{q,N-1}=\hat{V}\otimes_{\eta_{q,N-1}}\Z[q^{\pm 1}]$    & $ a) (q \text{ param,                                                               } $ $ \lambda=N-1)$   \\
 \hdashline
& &$ V_N\subseteq \hat{V}_{q,N-1}$  &$ \eta_{q,N-1} \ $\\
\hline
 $ \C $ & $\U_{\xi_N}=U_q(sl(2))\otimes_{\eta_{\xi_N,\lambda}}\C$  & $\hat{V}_{\xi_N,\lambda}=\hat{V}\otimes_{\eta_{\xi_N,\lambda}}\C$  & $ b) (q=\xi_N $,  $ \lambda \in \C)$   \\
 \hdashline
& & $U_{\lambda}\subseteq \hat{V}_{\xi_N,\lambda}$  & $\eta_{\xi_N,\lambda} \ $\\
\hline
\end{tabular}
\end{center}}
\begin{lemma}\label{inclusion}(\cite{JK},\cite{Cr2})
For the first case $a)$, which has a natural highest weight $\lambda$, $\hat{V}_{q,N-1}$ has an $N$-dimensional $\U$-submodule generated by the first $N$ vectors $ \{ v_{0},...,v_{N-1} \} $.
We denote this by: 
$$V_{N}:=<v_{0},...,v_{N-1}> \subseteq \hat{V}_{q,N-1}.$$
\end{lemma} 
\begin{notation}
For the root of unity case, we consider the $\C$-vector space generated by the first $N$ vectors in the specialisation of the generic Verma module as follows:
$$U_{\lambda}:=<v_{0},...,v_{N-1}> \subseteq \hat{V}_{\xi_N,\lambda}.$$
 
\end{notation}
\subsection{Braid group actions}\label{SS:1}
\begin{definition}(\cite{JK},\cite{Ito})\label{P:RT}(Braid group action on the Verma module)

There exists an $R$-matrix for the generic quantum group, which belongs to the completion $ U_\qs(sl(2))\hat{\otimes} U_\qs(sl(2))$ (\cite{JK} page 7), which leads to a brad group action. In order to define it, we start with following expressions:
\begin{equation}
\begin{aligned}
R=\sum_{n=0}^{\infty} q^{\frac{n(n-1)}{2}}E^{n}\otimes F^{(n)}\\
C(v_i \otimes v_j)=s^{-(i+j)}q^{2ij}v_j \otimes v_i.
\end{aligned}
\end{equation} 

By twisting the two components on which we act with the $R$-operator, we consider the following element: 
\begin{equation}
\mathscr R=C \circ R.
\end{equation}
\end{definition}

\begin{remark}(Action on the generic basis \cite{JK}) \label{R:R}\\
The $\mathscr{R}$-action on the standard basis of $\hat{V}\otimes \hat{V}$ is given by the following formula:
\begin{equation}
 \mathscr{R}(v_i\otimes v_j)= s^{-(i+j)} \sum_{n=0}^{i} F_{i,j,n}(q) \prod_{k=0}^{n-1}(sq^{-k-j}-s^{-1}q^{k+j}) \ v_{j+n}\otimes v_{i-n}.
 \end{equation}
Here, the polynomial $F_{i,j,n}\in \Z[q^{\pm}]$ has the following form:
$$F_{i,j,n}(q)=q^{2(i-n)(j+n)}q^{\frac{n(n-1)}{2}} {n+j \brack j}_q.$$
\end{remark}

\begin{proposition}(Generic braid group action \cite{JK})

This element, leads to representations of the braid group on the generic Verma module $\hat{V}$ of $U_q(sl(2))$ as follows 
\begin{equation}
\begin{aligned}
\hat{\varphi}_n: B_n \rightarrow Aut_{U_q(sl(2))}\left(\hat{V}^{\otimes n}\right) \ \ \ \ \ \ \ \ \ \ \ \ \ \ \ \ \ \ \ \ \ \ \\
\ \ \ \ \ \ \ \ \ \ \ \ \ \sigma _i^{\pm 1}\rightarrow Id_V^{\otimes (i-1)}\otimes \mathscr R^{\pm 1}  \otimes Id_V^{\otimes (n-i-1)}.
\end{aligned}
\end{equation}
\end{proposition}
\begin{definition}\hspace{-2mm} (Specialised action corresponding to natural parameter $\lambda=N-1$)\label{P:6}

Using the specialisation $ \eta_{q,N-1}$ and the property that $V_N$ is a $\U$-submodule of $\hat{V}_{q,N-1}$, the action $\hat{\varphi}_n$ induces the following specialised actions:
\begin{equation}
\begin{aligned}
\hat{\varphi}^{q,N-1}_n: B_n \rightarrow Aut_{U_q(sl(2))}\left(\hat{V}_{q,N-1}^{\otimes n}\right)\\
\varphi^{V_N}_n: B_n \rightarrow Aut_{U_q(sl(2))}\left(V_N^{\otimes n}\right).
\end{aligned}
\end{equation}
\end{definition}
In the sequel, we are interested in the root of unity case. First of all, the generic action $\hat{\varphi}_n$ will induce an action on the corresponding Verma module.

\begin{proposition} (Action on the Verma module at roots of unity) \label{P:7}

The specialisation $\eta_{\xi_N,\lambda}$ of the action $\hat{\varphi}_n$ leads to the following braid group action:
$$\ \ \hat{\varphi}^{\xi_N,\lambda}_{n}  :B_n \rightarrow \Aut \left( \hat{V}_{\xi_N,\lambda}^{\otimes n}\right). $$
\end{proposition}
\begin{remark}
Having in mind the definition of quantum invariants at roots of unity, we will want to consider braid group actions on tensor powers of $U_{\lambda}$. The subtlety is that at roots of unity, due to the choice of basis, we do not have the property that $U_{\lambda}$ is a submodule of $\hat{V}_{\xi_N,\lambda}$ over the quantum group. 

This can be corrected by a change of basis, but for our purpose, we will work with this version. In the sequel we will see an important property which ensures that the specialised $R$-matrix at roots of unity leads to a well defined action onto $U_{\lambda}^{\otimes n}$, which commutes with the inclusion into the Verma module. More specifically, we have the following.
\end{remark}
\begin{lemma}\label{L:2} (Action onto the finite part, at roots of unity)
  
 The action $\hat{\varphi}^{\xi_N,\lambda}_{n}$ on $\hat{V}_{\xi_N,\lambda}^{\otimes n}$ preserves the vector subspace $U_{\lambda}^{\otimes n}$ and the following diagram commutes:
\begin{center}
\begin{tikzpicture}
[x=1.2mm,y=1.4mm]

% Nodes of the diagram
\node (b1)  [color=blue]             at (0,0)    {$U_{\lambda}^{\otimes n}$};
\node (t1) [color=black] at (30,0)   {$\hat{V}_{\xi_N,\lambda}^{ \otimes n}$};
\node (b2) [color=blue] at (0,9)  {$\color{blue}B_n$};
\node (t2)  [color=black]             at (30,9)    {$B_n$};
\node (t2')  [color=black]             at (38,5)    {$\hat{\varphi}^{\xi_N,\lambda}_{n}$};

%\node (d2) [color=black] at (15,10)   {$\hookrightarrow$};
\node (d2) [color=black] at (15,0)   {$\subseteq$};
\node (d2) [color=black] at (0,5)   {$\color{blue}\curvearrowright$};
\node (d2) [color=black] at (30,5)   {$\curvearrowleft$};
\node (d2) [color=black] at (15,5)   {$\equiv$};
\end{tikzpicture}
\end{center} 
\end{lemma}
\begin{proof}
This special property follows from the explicit form of the coefficients of the $R$-matrix and the remark that the tensor powers $E^{\otimes n}$ act non-zero onto vectors from the finite part just if $n \leq N-1$. In this situation, the main point will be that the coefficients coming from the divided powers $F^{(n)}$ which would jump over the index $N$ on the second component, will contain $[N]_{\xi_N}$ which vanishes due to the root of unity. In the sequel, we explain this idea in details.

Having in mind the definition of the action $\hat{\varphi}^{\xi_N,\lambda}_{n} $, it is enough prove that:
\begin{center}
\begin{tikzpicture}
[x=1.2mm,y=1.4mm]

% Nodes of the diagram
\node (b1)  [color=blue]             at (0,0)    {$U_{\lambda}^{\otimes 2}$};
\node (t1) [color=black] at (30,0)   {$\hat{V}_{\xi_N,\lambda}^{ \otimes 2}$};
\node (b2) [color=blue] at (0,9)  {$\color{blue}B_2$};
\node (t2)  [color=black]             at (30,9)    {$B_2$};
\node (t2')  [color=black]             at (38,5)    {$\mathscr R|_{\eta_{\xi_N,\lambda}}$};

%\node (d2) [color=black] at (15,10)   {$\hookrightarrow$};
\node (d2) [color=black] at (15,0)   {$\subseteq$};
\node (d2) [color=black] at (0,5)   {$\color{blue}\curvearrowright$};
\node (d2) [color=black] at (30,5)   {$\curvearrowleft$};
\node (d2) [color=black] at (15,5)   {$\equiv$};
\end{tikzpicture}
\end{center}

More precisely, we want to show that:
 $$
 \mathscr{R}|_{\eta_{\xi_N,\lambda}}(v_i\otimes v_j) \in U_{\lambda}^{\otimes 2},  \ \ \ \forall i,j \in \{0,...,N-1\}.$$
Let us fix $i,j \in \{0,...,N-1\}.$ Using the formula \ref{R:R} for the $R$-matrix, together with specialisation $\eta_{\xi_N,\lambda}$ we obtain:
\begin{equation}\label{eq:5}   
\begin{aligned}
 \mathscr{R}|_{\eta_{\xi_N,\lambda}}(v_i\otimes v_j) = {\xi_N}^{-(i+j)\lambda} \sum_{n=0}^{i} & F_{i,j,n}(\xi_N) \cdot \\
 & \cdot \prod_{k=0}^{n-1}\{\lambda-(k+j)\}_{\xi_N} \ v_{j+n}\otimes v_{i-n}. 
\end{aligned}
 \end{equation}
Looking at the precise coefficients, we notice that the specialisation leads to the following:
$$F_{i,j,n}(\xi_N)={\xi_N}^{2(i-n)(j+n)}{\xi_N}^{\frac{n(n-1)}{2}}\cdot \frac{[j+1]_{\xi_N}\cdot...\cdot [j+n]_{\xi_N}}{[n]_{\xi_N}!}.$$
Now, using that in the equation \ref{eq:5} we have $n\leq i$, it follows that $n \leq N-1$ as well. This ensures that $$[n]_{\xi_N}! \neq 0.$$ On the other hand, we investigate the coefficients from the denominators. If the index $j+n\geq N$, then we remark that $[N]_{\xi_N}=0$ occurs as a factor in the product from the denominator of $F_{i,j,n}(\xi_N)$. This shows that:
$$F_{i,j,n}(\xi_N)=0 \text {   if   } i\leq N-1 \text { and } j+n\geq N.$$
The previous equation shows exactly that the $\mathscr R$-action preserves the module $U_{\lambda}^{\otimes 2}$ inside the specialised Verma module and concludes the proof. 
\end{proof}

\begin{definition} (Braid group action at roots of unity)\label{D:11}

Using the previous Lemma, we conclude that $\hat{\varphi}^{\xi_N,\lambda}_{n}$ leads to an induced braid group representation on the tensor powers of $U_{\lambda}$, which we denote by:
$$\ \ \varphi^{U_{\lambda}}_{n}: B_n \rightarrow \Aut_{\C}( U_{\lambda}^{\otimes n}). $$
\end{definition}
 %1) There exist a braid group action on the subcategory of $\U$-representations $Rep({\U})$ (finite dimensional or Verma module). This comes from the specialisation of the $R$-matrix $\mathscr R|_{\eta_{\lambda}\otimes\eta_{\lambda}} \in \U \hat{\otimes} \U$ and it has the following form:
%$$R_{V,V}=\mathscr R|_{\eta_{\lambda}\otimes\eta_{\lambda}}\circ \tau \in Hom_{U_q(sl(2))}(V\otimes V,V \otimes V),  \ \ \ \forall \ V \in Rep(\U).$$
\begin{notation}\label{P:2}  
 1) The category of finite dimensional representations of ${\U}$ has the following dualities:
$$\forall \ V_N \in Rep^{\text{f. dim}}_{\U}$$ 
\begin{align}\label{E:DualityForCat}
\tcoev_{V_N} :\, & \Li \rightarrow V_N\otimes V_N^{*} \text{ is given by } 1 \mapsto
  \sum v_j\otimes v_j^*,\notag
  \\
 \tev_{V_N}:\, & V_N^*\otimes V_N\rightarrow \Li \text{ is given by } f\otimes w
  \mapsto f(w),\notag
  \\
\coev_{V_N}:\, & \Li \rightarrow V_N^*\otimes V_N \text{ is given by } 1 \mapsto \sum
   v_j^*\otimes K^{-1} v_j,
  \\
\ev_{V_N}:\, & V_N\otimes V_N^{*}\rightarrow \Li \text{ is given by } v\otimes f
  \mapsto  f(K v),\notag
\end{align}
for $\{v_j\}$ a basis of $V_N$ and $\{v_j^*\}$ the dual basis
of ${V_N}^*$.\\

2) For the root of unity case, we will use the following maps: 
\begin{align} 
 \tcoev_{U_{\lambda}} :\, & \C\rightarrow U_{\lambda}\otimes U_{\lambda}^{*} \text{ is given by } 1 \mapsto
  \sum v_j\otimes v_j^*,\notag
  \\
 \tev_{U_{\lambda}}:\, & U_{\lambda}^*\otimes U_{\lambda}\rightarrow \C \text{ is given by } f\otimes w
  \mapsto f(w),\notag
  \\
\coev_{U_{\lambda}}:\, & \C \rightarrow U_{\lambda}^*\otimes U_{\lambda} \text{ is given by } 1 \mapsto \sum
   v_j^*\otimes K^{N-1}v_j,
  \\
\ev_{U_\lambda}:\, & U_{\lambda}\otimes U_{\lambda}^{*}\rightarrow \C \text{ is given by } v\otimes f
  \mapsto  f(K^{-N+1}v),\notag
\end{align}
 \end{notation}

\subsection{Quantum representations on weight spaces}\label{S:quan} 
In this part, we consider certain subspaces in tensor powers of $U_q(sl(2))$-representations, prescribed by the $K$-action. They will be important in the sequel because they carry homological information.
\begin{definition}(Weight spaces)\\\label{D:2}  
{\bf 1) Generic case ($q$ and $s$ parameters) }\\
The $n^{th}$ weight space of the generic Verma module $\hat{V}$
corresponding to the weight $m$ is given by:
\begin{equation}
\hat{V}_{n,m}:= \{v \in \hat V^{\otimes n} \mid Kv=s^nq^{-2m}v \}.
\end{equation}
{\bf 2) The case ($q$ generic, $\lambda=N-1 \in \N$)}\\
The weight space of $ \hat{V}^{\otimes n}_{q,N-1}$ corresponding to the weight $m$ is:
\begin{equation}
 \hat{V}^{q,N-1}_{n,m}:=\{v\in \hat{V}^{\otimes n}_{q,N-1} \mid Kv=q^{n(N-1)-2m}v \}.
\end{equation}
The weight space for the finite dimensional representation $V^{\otimes n}_{N}$ of weight $m$ is:
\begin{equation}
V^{q,N-1}_{n,m}:=\{v\in V^{\otimes n}_{N} \mid Kv=q^{n(N-1)-2m}v \}.
\end{equation}
\end{definition}
\begin{remark}
Since the generic braid group action commutes with the action of the quantum group, the representation $\hat{\varphi}_n$ induces a well defined action on the generic weight spaces $$\hat{\varphi}_{n,m}:B_n \rightarrow \Aut( \hat{V}_{n,m})$$
called the generic quantum representation on weight spaces of the Verma module.
\end{remark}
\begin{notation}
We will use the following indexing set:
$$E_{n,m}=\{e=(e_1,...,e_{n})\in \N^{n} \mid e_1+...+e_{n}=m \}.$$
\end{notation}
\begin{remark}(Basis for weight spaces)\\
A basis for the generic weight space is given by:
$$\mathscr{B}_{\hat{V}_{n,m}}=\{v_e:= v_{e_1}\otimes... \otimes v_{e_n}| e \in E_{n,m}\}.
$$
\end{remark}

\begin{remark} \label{P:hww} Similarly, using the specialisation with generic $q$, we get induced braid group actions as follows:
$$ \ \ \ \ \ \ \ 2) \  a) \ \hat{\varphi}^{q,N-1}_{n,m}  :B_n \rightarrow \Aut \left( \hat{V}^{q,N-1}_{n,m} \right) \text{ is a well defined action induced by } \hat{\varphi}^{q,N-1}_n.$$
$$ \ 2) \  b) \ \varphi^{q,N-1}_{n,m}:B_n \rightarrow \Aut \left( V^{q,N-1}_{n,m} \right) \text{ is induced by } \varphi^{V_N}_n. \hspace{32mm}$$ 
This action is called the quantum representation on weight spaces corresponding to the finite dimensional module.
\end{remark}
This allows us to define quantum representations for the root of unity case as follows:
\begin{definition}
{\bf 3) The case with $q=\xi_N$ root of unity and $\lambda \in \C$}
 
The weight space of $ \hat{V}^{\otimes n}_{\xi_N,\lambda}$ of weight $m$ is the following:
\begin{equation}
 \hat{V}^{\xi_N,\lambda}_{n,m}:=\hat{V}_{n,m}|_{\eta_{\xi_N,\lambda}} \subseteq \hat{V}^{\otimes n}_{\xi_N,\lambda}.
 \end{equation}
 %\{v\in \hat{V}^{\otimes n}_{\xi_N,\lambda} \mid Kv=\xi_N^{n\lambda-2m}v \}.

The weight space of the finite module $U^{\otimes n}_{\lambda}$  corresponding to the weight $m$ is:
\begin{equation}
 V^{\xi_N,\lambda}_{n,m}:= \hat{V}^{\xi_N,\lambda}_{n,m} \cap U^{\otimes n}_{\lambda} \subseteq U^{\otimes n}_{\lambda}. 
  \end{equation}
% \{v\in U^{\otimes n}_{\lambda} \mid Kv=\xi_N^{n\lambda-2m}v \}.

\end{definition}
\begin{remark}\label{R:8}
The braid group action from the Lemma \ref{L:2}, induces well defined braid group actions at roots of unity as follows.
$$3) \ a) \ \hat{\varphi}^{\xi_N,\lambda}_{n,m}  :B_n \rightarrow \Aut( \hat{V}^{\xi_N,\lambda}_{n,m}) \text{ is a well defined action induced by }\hat{\varphi}_n^{\xi_N,\lambda}.$$
%$\hat{\varphi}_n |_{\eta_{\xi_N,\lambda}}$.
$$  3) \ b) \ \varphi^{\xi_N,\lambda}_{n,m}:B_n \rightarrow \Aut( V^{\xi_N,\lambda}_{n,m})  \text{ is induced by } \varphi_n^{U_{\lambda}}.\hspace{35mm}$$ 
These are well defined using Lemma \ref{L:2} together with the fact that the braid group action preserves the weights of the vectors. We refer to the second action as the quantum representation on weight spaces corresponding to the finite dimensional module at root of unity.
\end{remark}
\begin{center}
\begin{tabular} { | l | c | c | c | l | } 
 \hline
Specialisation & Representation & \ \  Action \ \  & Weight space & Weight action\\    
\hline                                    
$ {\bf 1)} (q,s\text{ param})$ & $\hat{V}^{\otimes n}$  &  $ \hat{\varphi}_n $ & $ \hat{V}_{n,m}$ & $\hat{\varphi}_{n,m}$\\
\hline                                    
$ { \bf 2)} \ (q, \lambda=N-1)$ & $\hat{V}^{\otimes n}_{q,N-1}$  &  $\hat{\varphi}^{q,N-1}_{n}$ & $ \hat{V}^{q,N-1}_{n,m}$ & $\hat{\varphi}^{q,N-1}_{n,m}$\\
\hdashline                                    
 \ \ \ \ \ \ \ $\eta_{q, N-1}$ &$V_N^{\otimes n}$& $\varphi^{V_N}_{n}$&$ V^{q,N-1}_{n,m}$& $\varphi^{q,N-1}_{n,m}$\\
\hline
$ { \bf 3)} (q=\xi_N, \lambda \in \C)$ & $\hat{V}^{\otimes n}_{\xi_N,\lambda}$  &  $ \hat{\varphi}^{\xi_N, \lambda}_{n}$ & $ \hat{V}_{n,m}^{\xi_N,\lambda}$ & $\hat{\varphi}^{\xi_N,\lambda}_{n,m}$\\
\hdashline                                    
 \ \ \ \ \ \ \ $\eta_{\xi_N, \lambda}$ &$U_{\lambda}^{\otimes n}$&$ \varphi^{U_\lambda}_{n}$&$ V_{n,m}^{\xi_N,\lambda}$ &$\varphi^{\xi_N,\lambda}_{n,m}$\\
\hline
\end{tabular}
\end{center}

\subsection{Coloured Jones polynomials as renormalised invariants}
The coloured Jones polynomials $\{ J_N(L,q )| N \in \N \}$ form a sequence of invariants constructed from the finite dimensional representations $\{V_N\}_{N \in \N}$ described above, using the Reshetikhin-Turaev method.  
In the sequel, we denote by $w:B_n\rightarrow \Z$ the map given by the abelianisation.
\begin{notation}(The trivial braid)
For the next definitions, we denote by:\\
1) $\mathbb I_n$ the trivial braid in $B_n$ with all strands oriented upwards.\\
2) $\bar{\mathbb I}_n$ the trivial braid in $B_n$ with all strands oriented downwards.
\end{notation}
Further on, we refer to \cite{Cr2}- Section 2 for the details of the Reshetikhin-Turaev construction as well as the importance of the orientations of the strands. Roughly speaking if we have $\bar{\mathbb I}_n$, then we have the identity of the dual representation and $\mathbb I_n$ corresponds to the initial representation ($V_N^{\otimes n}$ or $U_{\lambda}^{\otimes n}$ for the root of unity case). However, as we will see, we will add extra morphisms such that we have the action just on the tensor power of these representations and not their duals and then use the notations that we have discussed in this section. 

 \begin{proposition}(\cite{Ito3},\cite{Cr2} Coloured Jones polynomial from a braid presentation)\label{P:1}
Let us fix $N \in \N$.  Consider $L$ to be an oriented knot and $\beta \in B_n$ such that $L=\hat{\beta}$ (braid closure). Then, the Reshetikhin-Turaev construction leads to the following formula for the $N^{th}$ coloured Jones invariant:
$$J_N(L,q)=\frac{1}{[N]_q}q^{-(N-1)w(\beta_n)}\left( \ev_{V_N^{\otimes n}} \  \circ \ \varphi^{V_{N}}_{2n} (\beta_{n} \cup \bar{\mathbb I}_n ) \  \circ \ {\tcoev}_{V_N^{\otimes n}} \right) (1).$$
\end{proposition}
\begin{notation}(Projection map given by the highest weight vector)

Let us denote by $r: V_N\rightarrow  \langle v_0 \rangle \ _{\Z[q^{\pm1}]}$ the projection onto the subspace generated by the vector $v_0$ and $c: \langle v_0 \rangle_{\Z[q^{\pm1}]}\rightarrow \Z[q^{\pm1}]$ the map given by the coefficient of the vector $v_0$. Then, we denote their composition by:
\begin{equation}
\begin{cases}
\pi: V_N\rightarrow \Z[q^{\pm 1}]\\
\pi=c \circ r.
\end{cases}
\end{equation}
\end{notation}

\begin{corollary}(Coloured Jones polynomial as a modified invariant)\label{defCJ}\\
We can see the coloured Jones polynomials by cutting a strand as follows:
\begin{equation}\label{eq:J}
\begin{aligned}
J_N(L,q)=&q^{-(N-1)w(\beta_n)} \ \pi \circ \\
&\circ \left((Id\otimes \ev_{V_N^{\otimes {n-1}}} ) \circ \varphi^{V_{N}}_{2n-1} (\beta_{n} \cup \bar{\mathbb I}_{n-1} ) \circ (\Id \otimes {\tcoev}_{V_N^{\otimes {n-1}}}) \right) (v_0).
\end{aligned}
\end{equation}
\end{corollary}
\begin{proof}
This comes from the formula presented in Proposition \ref{P:1} together with the normalisation procedure developed in \cite{GP} and the properties of the Reshetikhin-Turaev functorial construction.
\end{proof}
\subsection{Coloured Alexander Polynomials}
The quantum group $U_{\xi_{N}}(sl(2))$ at roots of unity of order $2N$ together with the modules $U_{\lambda}$ lead to the non-semisimple quantum invariant called coloured Alexander polynomials ( or ADO \cite{ADO}), as follows.
\begin{notation}
Let $ r^{\lambda}: U_{\lambda}\rightarrow \langle v_0 \rangle _{\C}$ the projection onto the subspace generated by the vector $v_0$ and $C: \langle v_0\rangle \  _{\C}\rightarrow \C$ the map which is given by the coefficients of the generator $v_0$. We denote their composition as follows:
\begin{equation}
\begin{cases}
p: U_{\lambda}\rightarrow \C\\
p=C \circ r^{\lambda}.
\end{cases}
\end{equation}
\end{notation}

\begin{proposition}(The ADO invariant from a braid presentation)\\{\label{defADO}}
Let $L$ be an oriented knot. Consider $\beta_n \in B_n $ such that $L=\hat{\beta_n}$.
Then, the $N^{th}$ coloured Alexander invariant of $L$ can be expressed as follows:
\begin{equation}\label{eq:A}
\begin{aligned}
\Phi_{N}(L,\lambda)=&{\xi_N}^{(N-1)\lambda w(\beta_n)} \ p \ \circ \\
&\left((Id\otimes \ev_{U_{\lambda}^{\otimes {n-1}}} ) \circ \varphi^{U_{\lambda}}_{2n-1} (\beta_{n} \cup \bar{\mathbb I}_{n-1} ) \circ (\Id \otimes {\tcoev}_{U_{\lambda}^{\otimes {n-1}}}) \right) (v_0).
\end{aligned}
\end{equation}
\end{proposition} 
\begin{proof}
In \cite{Ito2} it is presented a formula for the coloured Alexander polynomial which is the same as the one from equation \eqref{eq:A} except the part concerning the braid group action. The difference occurs from the fact that in that context, it is used the usual version of the quantum group $U_{q}(sl(2))$ instead of the one with divided powers. However, the formula for the braid group action from \cite{Ito2} (equation $(2)$) is the same as the one presented in remark \ref{R:R}.  Then, the action $\hat{\varphi}_n|_{\eta_{\xi_N,\lambda}}$ is used on different bases in our case and in the definition from \cite{Ito2}. However, since these actions are closed up with the same evaluations and coevaluation, they will lead to the same result. Now, using proposition \ref{P:7} together with Lemma \ref{L:2} and definition \ref{D:11}, we conclude that we can use the braid group action $\varphi^{U_{\lambda}}_{2n-1}$, which concludes the proof.
\end{proof}
\section{Lawrence representation}\label{3}
In this part, we present a version of Lawrence representations that will be suitable for our topological models. We will use certain definitions from \cite{Martel}.

Let us fix two natural numbers $n$ and $m$. We consider $\mathscr D_n$ to be the two dimensional disc with boundary, with $n$ punctures: $$\mathscr D_n=\mathbb D^2 \setminus \{1,...,n\}.$$ Let us define the unordered configuration space in this punctured disc, by taking the product and making the quotient with respect to the $m^{th}$ symmetric group $Sym_m$:
$$C_{n,m}:=\{ (x_1,...,x_m)\in (\mathscr D_n)^{\times m} \mid x_i \neq x_j, \forall \ 1 \leq i < j\leq m\} /Sym_m. $$

For the sequel, we fix $d_1,..d_m \in \partial\mathscr D_n$ and ${\bf d}=(d_1,...,d_m)$ the corresponding base point in the configuration space.
\subsection{Homology of the covering space}
\begin{definition}(Local system)

Let $\rho: \pi_1(C_{n,m}) \rightarrow H_1\left( C_{n,m}\right)$ be the abelianisation map. 
For $m \geq 2$, the homology of the unordered configuration space, is known to be: 
\begin{equation*}
\begin{aligned}
H_1\left( C_{n,m}\right) \simeq & \ \ \Z^{n} \ \oplus \ \Z \ \ \\
& \langle \rho(\sigma_i) \rangle \  \langle \rho(\delta) \rangle,  \ \ \ \ \ {i\in \{1,...,n\}}.
\end{aligned}
\end{equation*}

Here, $\sigma_i\in \pi_1(C_{n,m})$ is represented by the loop in $C_{n,m}$ with $m-1$ fixed components (the base points $d_2$,...,$d_n$) and the first one going on a loop in $\mathscr D_n$ around the puncture $p_i$, as in the picture below. 

The generator $\delta \in \pi_1(C_{n,m})$ is given by a loop in the configuration space with $(m-2)$ constant points ( given by $d_3$,...,$d_{n}$) and the first two components which swap the two initial points $d_1$ and $d_2$, as in figure \ref{fig2}.  

\begin{figure}[H]
\begin{tikzpicture}
[scale=4.3/5]
\foreach \x/\y in {0/2,2/2,4/2,2/1,2.6/1,3.6/1.15} {\node at (\x,\y) [circle,fill,inner sep=1pt] {};}
\node at (0.2,2) [anchor=north east] {$1$};
\node at (2.2,2) [anchor=north east] {$i$};
\node at (4.2,2) [anchor=north east] {$n$};
\node at (3,2) [anchor=north east] {$\sigma_i$};
\node at (2.2,1) [anchor=north east] {$d_1$};
\node at (2.9,1.02) [anchor=north east] {$d_2$};
\node at (4,1.05) [anchor=north east] {$d_m$};
\node at (2.68,2.3) [anchor=north east] {$\wedge$};
\draw (2,1.8) ellipse (0.4cm and 0.8cm);
\draw (2,2) ellipse (3cm and 1cm);
\foreach \x/\y in {7/2,9/2,11/2,8.5/1,9.5/1,10.5/1.10} {\node at (\x,\y) [circle,fill,inner sep=1pt] {};}
\node at (7.2,2) [anchor=north east] {$1$};
\node at (9.2,2) [anchor=north east] {$i$};
\node at (11.2,2) [anchor=north east] {$n$};
\node at (9,1) [anchor=north east] {$d_1$};
\node at (9.7,1.02) [anchor=north east] {$d_2$};
\node at (10.9,1.05) [anchor=north east] {$d_m$};
\node at (8.5,1.7) [anchor=north east] {$\delta$};
\draw (9,2) ellipse (3cm and 1cm);
%\draw[->, color=blue, very thick]             (12.2,1)     to[in=30,out=30] node[yshift=-3mm,font=\small]{}  (12.5,1.02);
\draw (9.5,1)  arc[radius = 5mm, start angle= 0, end angle= 180];
\draw [->](9.5,1)  arc[radius = 5mm, start angle= 0, end angle= 90];
\draw (8.5,1) to[out=50,in=120] (9.5,1);
\draw [->](8.5,1) to[out=50,in=160] (9.16,1.19);
\end{tikzpicture}
\caption{}
\label{fig2}
\end{figure}
Let $p:\Z^n\oplus \Z\rightarrow \Z \oplus \Z$ be the augmentation map given by:
$$p(x_1,...,x_m,y)=(x_1+...+x_m,y).$$
Combining the two morphisms, let us consider the local system:
\begin{equation}\label{eq:23}
\begin{aligned}
\phi: \pi_1(C_{n,m}) \rightarrow \ & \Z \oplus \Z\\ 
 & \langle x \rangle \ \langle d \rangle\\
\phi= p \circ \rho. \ \ \ \ \ \ \ \ 
\end{aligned}
\end{equation}
\end{definition}
\begin{definition}(Covering of the configuration space)\label{D:12}
Let $\tilde{C}_{n,m}$ be the covering of $C_{n,m}$ corresponding to the local system $\phi$. Then, the deck transformations of this covering are given by:
$$Deck(\tilde{C}_{n,m},C_{n,m})\simeq \langle x \rangle \langle d \rangle .$$
\end{definition}
One of the main tools in this construction is the homology of this covering space. Let us fix a point $w \in \partial \mathscr D_n$. We will work with the part of the Borel-Moore homology of this covering which comes from the corresponding Borel-Moore homology of the base space twisted by the local system. 
In the sequel, we denote by $C_{w}$ the part of the boundary of $C_{n,m}$ represented by configurations containing the base point $w$. Also, by $\pi^{-1}({w})$ we denote the part of the boundary of $\tilde{C}_{n,m}$ represented by the fiber over $C_{w}$.
\begin{proposition}(\cite{Cr2}) The braid group action arising from the mapping class group property and the action coming from Deck transformations are compatible at the homological level:
$$B_n \curvearrowright H^{\text{lf}}_m(\tilde{C}_{n,m},\pi^{-1}(x); \Z) \ (\text{ as a module over }\Z[x^{\pm1}, d^{\pm1}]).$$ 
\end{proposition}

\begin{proposition}(\cite{CrM})
There is a natural injective map, compatible with the braid group actions:
$$\iota: H^{\text{lf}}_m(C_{n,m}, C_{w}; \mathscr L_{\phi})\rightarrow H^{\text{lf}}_m(\tilde{C}_{n,m}, \pi^{-1}({w});\Z).$$
\end{proposition}
\begin{notation}\label{R:1}  
Let $H^{\text{lf},-}_m(\tilde{C}_{n,m}, \Z)$ be the image if the map $\iota$.  
\end{notation}
\begin{figure}[H]
\begin{center}
\begin{tikzpicture}\label{pic'}
[x=0.7mm,y=0.03mm,scale=0.1/5,font=\Large]
\foreach \x/\y in {-1.2/2, 0.4/2 , 1.3/2 , 2.5/2 , 3.6/2 } {\node at (\x,\y) [circle,fill,inner sep=1.3pt] {};}
\node at (-1,2) [anchor=north east] {$w$};
\node at (0.6,2.5) [anchor=north east] {$1$};
\node at (0.2,1.55) [anchor=north east] {$\color{red}\eta^e_1$};
\node at (1,2) [anchor=north east] {$\color{red} \eta^e_{e_1}$};
\node at (3.8,1.7) [anchor=north east] {$\color{red} \eta^e_{m}$};
\node (dn) at (8,1.5) [anchor=north east] {\Large \color{red}$\eta_e=(\eta^e_1,...,\eta^e_m)$};
\node at (2,5.1) [anchor=north east] {\Large \color{red}$\tilde{\eta}_e$};

%\node at (1.4,2) [anchor=north east] {i};
%\node at (2.9,2) [anchor=north east] {i+1};
\node at (4,2.5) [anchor=north east] {n};
\node at (0.8,0.8) [anchor=north east] {\bf d$=$};
\node at (0.8,4.4) [anchor=north east] {\bf $\bf \tilde{d}$};
\node at (0.3,2.6) [anchor=north east] {$\color{black!50!red}\text{Conf}e_1$};
\node at (3.6,2.6) [anchor=north east] {$\color{black!50!red}\text{Conf} e_{n}$};
\node at (2.1,3) [anchor=north east] {\huge{\color{black!50!red}$\FF_e$}};
\node at (2.1,6.2) [anchor=north east] {\huge{\color{black!50!red}$\tilde{\FF}_e$}};
\node at (-2.5,2) [anchor=north east] {\large{$C_{n,m}$}};
\node at (-2.5,6) [anchor=north east] {\large{$\tilde{C}_{n,m}$}};
\draw [very thick,black!50!red,-][in=-160,out=-10](-1.2,2) to (0.4,2);
\draw [very thick,black!50!red,->] [in=-145,out=-30](-1.2,2) to (3.6,2);
 \draw[very thick,black!50!red] (2.82, 5.6) circle (0.6);
\draw (2,2) ellipse (3.2cm and 1.3cm);
\draw (2,5.4) ellipse (3cm and 1.11cm);
\node (d1) at (1.3,0.8) [anchor=north east] {$d_1$};
\node (d2) at (1.9,0.8) [anchor=north east] {$d_{e_1}$};
\node (dn) at (2.8,0.8) [anchor=north east] {$d_m$};
\draw [very thick,dashed, red,->][in=-60,out=-190](1.2,0.7) to  (-0.2,1.9);
\draw [very thick,dashed,red,->][in=-70,out=-200](1.3,0.7) to (0,1.9);
\draw [very thick,dashed,red,->][in=-90,out=0](2.5,0.7) to (3,1.6);
\draw [very thick,dashed,red,->][in=-70,out=-200](0.8,4.4) to (3,5);
\end{tikzpicture}
\end{center}
\caption{}
\label{fig3}
\end{figure}
\subsection{Lawrence representation} In this part, we recall the definition of the version of the homological Lawrence representation the braid groups  which uses this set-up.
\begin{definition}({\bf Multiarcs} \cite{Martel})\\
Let us fix ${\bf \tilde{d}} \in \tilde{C}_{n,m}$ be a lift of the base point ${ \bf d}=\{d_1,...,d_m\}$ in the covering.

a) Let us consider a partition $e \in E_{n,m}$. Then, for each $i \in \{ 1,...,n \} $, we consider the segment in $\mathscr D_n$ which starts at the point $w$ and ends at the $i^{th}$ puncture. Then we look at the space of ordered configurations of $e_i$ points on this segment, as drawn in figure \ref{fig3} from above. 
Let us denote the projection onto the unordered configuration space by:
$$\pi_m : {\mathscr D}^{\times m}_n \setminus \{x=(x_1,...,x_m)| \exists \  i,j , \ x_i=x_j \}) \rightarrow C_{n,m}.$$
Then, the product of these configuration spaces leads to a submanifold in the configuration space: 
$$\FF_e:=\pi_m (Conf_{e_1} \times ... \times Conf_{{e}_{n}})\subseteq C_{n,m}$$
b) Now, we fix a set of paths between the fixed points from the boundary and the segments: $${\eta}^{e}_k: [0,1] \rightarrow \mathscr D_n, k \in \{1,...,m\}$$ as in picture \ref{fig3}.
The set of paths $\{ \eta^e_k \}$ leads to a path in the configuration space:
$$\eta^e := \pi_m \circ (\eta^{e}_1, ..., \eta^{e}_m ) : [0,1] \rightarrow C_{n,m}.$$
We remark that:
\begin{equation}
\begin{cases}
\eta^e(0)= {\bf d} \\
\eta^e(1)\in \FF_e.
\end{cases}
\end{equation}
Let us consider the unique lift of the path $\eta^e$ and denote it by $\tilde{\eta}^e$ such that: 
\begin{equation}
\begin{cases}
\tilde{\eta}^e:  [0,1] \rightarrow \tilde{C}_{n,m}\\
\tilde{\eta}^e(0)=\bf \tilde{ d}.
\end{cases}
\end{equation}
{\bf 3) Multiarcs}\\ \label{f}
Let $\tilde{\FF}_{e}$ be the unique lift of the submanifold $\FF_e$ with the properties:
\begin{equation}
\begin{cases}
\tilde{\FF}_{e}: (0,1)^m\rightarrow \tilde{C}_{n,m}\\ 
\tilde{\eta}^e(1) \in \tilde{\FF}_{e}.
\end{cases}
\end{equation}
Using this submanifold, we obtain a class in the Borel-Moore homology
$$[\tilde{\FF}_{e}] \in H^{\text{lf},-}_m(\tilde{C}_{n,m}, \Z).$$ 
This is called the multiarc corresponding to the partition $e \in E_{n,m}$.
%We denote by $$[\tilde{U}_{e}] \in H^{\text{lf},-}_m(\tilde{C}_{n,m}, \Z).$$ the class given by the submanifold which is obtained in the same way, using multi segments between punctures instead configuration points on a certain segment.
 \end{definition}
\begin{proposition}(\cite{Martel})
The set of all multiarcs 
$\{  [\tilde{\FF}_e] \ | \  e\in E_{n,m} \}$ is a basis for $ H^{\text{lf},-}_m(\tilde{C}_{n,m}, \Z).$
\end{proposition}

\begin{notation}(Normalised multiarc)
For $e \in E_{n,m}$, we denote the normalisation of the multiarc $[\tilde{\FF}_e]$ as below:
$$\F_{e}:= x^{\frac{1}{2}\sum_{i=1}^{n}(i-1) e_i} [\tilde{\FF}_e]  \in H^{\text{lf},-}_m(\tilde{C}_{n,m}, \Z).$$
(we see later that we will use this normalisation when we specialise the coefficients and the fraction from the exponent from above will not occur in practice)
\end{notation}
In the sequel, we will recall other homology classes, which will be more suitable for the topological model that we have in mind. We refer to \cite{Martel} for the detailed definition. For each partition $e\in E_{n,m}$, one defines $U_e$ to be the $m$-dimensional submanifold in $C_{n,m}$ constructed in the same manner as above but starts from $m$ different segments rather than configuration spaces on segments. More precisely, one consider the product of $m$ different segments between $w$ and the punctures, prescribed by the partition, as presented in the picture \ref{fig4}, and then quotient it by the symmetric group.   
\begin{figure}[H]
\begin{center}
\begin{tikzpicture}\label{pic}
[x=0.5mm,y=0.02mm,scale=0.1/5,font=\Large]
\foreach \x/\y in {-1.2/2, 0.4/2 , 1.3/2 , 2.5/2 , 3.6/2 } {\node at (\x,\y) [circle,fill,inner sep=1.3pt] {};}
\node at (-1,2) [anchor=north east] {$w$};
\node at (0.6,2.5) [anchor=north east] {$1$};
\node at (4,2.5) [anchor=north east] {n};
\node at (0.3,2.6) [anchor=north east] {$\color{red}e_1$};
\node at (3.5,2.4) [anchor=north east] {$\color{red}e_{n}$};
\node at (2.1,3) [anchor=north east] {\huge{\color{black!50!red}$U_e$}};
\node at (2.1,6.2) [anchor=north east] {\huge{\color{black!50!red}$\tilde{U}_e$}};
\node at (2,5.1) [anchor=north east] {\Large \color{red}$\tilde{\eta}^e$};
\node at (-2.5,2) [anchor=north east] {\large{$C_{n,m}$}};
\node at (-2.5,6) [anchor=north east] {\large{$\tilde{C}_{n,m}$}};
\draw [very thick,black!50!red,-](-1.2,2) to (0.4,2);
\draw [very thick,black!50!red,-][in=-155,out=-30](-1.2,2) to (0.4,2);
\draw [very thick,black!50!red,-] [in=-145,out=-30](-1.2,2) to (3.6,2);
\draw [very thick,black!50!red,-] [in=-130,out=-40](-1.2,2) to (3.6,2);

\draw [very thick,black!50!red,->](-1.2,2) to (-0.5,2);
\draw [very thick,black!50!red,->][in=-160,out=-30](-1.2,2) to (-0.3,1.8);
\draw [very thick,black!50!red,->] [in=-185,out=-35](-1.2,2) to (1.6,1.25);
\draw [very thick,black!50!red,->] [in=-180,out=-40](-1.2,2) to (1.6,1);

\draw (2,2) ellipse (3.2cm and 1.3cm);
\draw (2,5.4) ellipse (3cm and 1.11cm);
\node (d1) at (1.6,0.8) [anchor=north east] {$d_1$};
\node (d2) at (2.2,0.8) [anchor=north east] {$d_2$};
\node (dn) at (2.8,0.8) [anchor=north east] {$d_m$};
\node (dn) at (4,1.7) [anchor=north east] {\Large \color{red}$\eta^e$};
\draw [very thick,dashed, red,->][in=-60,out=-190](1.2,0.7) to  (-0.3,2);
\draw [very thick,dashed,red,->][in=-70,out=-200](1.3,0.7) to (0,1.9);
\draw [very thick,dashed,red,->][in=-90,out=0](2.5,0.7) to (3,1.5);
\node at (0.8,0.8) [anchor=north east] {\bf d$=$};
\node at (0.8,4.4) [anchor=north east] {\bf $\bf \tilde{d}$};
\draw [very thick,dashed,red,->][in=-70,out=-200](0.8,4.4) to (3,5);
\node at (0.8,4.4) [anchor=north east] {\bf $\bf \tilde{d}$};
 \draw[very thick,black!50!red] (2.82, 5.6) circle (0.6);
\end{tikzpicture}
\end{center}
\caption{}
\label{fig4}
\end{figure}

\begin{definition}(Code sequence \cite{Martel})\\
Let $\tilde{U}_e$ be the corresponding lift in $\tilde{C}_{n,m}$ of the submanifold $U_e$ through $\tilde{\eta}^e(1)$:
$$\U_{e}:= [\tilde{U}_e]  \in H^{\text{lf},-}_m(\tilde{C}_{n,m}, \Z).$$
\end{definition}
\begin{notation}
Let us denote the following quantum integers:
\begin{equation}
(n)_{y}=\frac{1-y^n}{1-y} \ \ \ \ \ \ (n)_{y}!=(1)_{y} \cdot ... \cdot (n)_{y}.
\end{equation}
\end{notation}

\begin{proposition}(Relation multiarcs and code sequences) \label{E:1}\\ 
Combing the relation between configurations on segments and multi-segments with relations concerning the breaking of an arc by a puncture from \cite{Martel}(Corollary 4.9), we have:
\begin{equation}
\U_{e}=\prod_{i=1}^{n} (e_i)_{d}! \ [\tilde{\FF}_e].
\end{equation}
This shows that:
\begin{equation}
\mathscr F_e=\frac{x^{\frac{1}{2}\sum_{i=1}^{n}(i-1) e_i}}{\prod_{i=1}^{n} (e_i)_{d}!} \U_{e}.
\end{equation}
\end{proposition}
\begin{notation}(Lawrence representation)\label{T1}\\
We denote the braid group action in the basis given by multiarcs $\mathscr{B}_{H^{\text{lf},-}_m(\tilde{C}_{n,m}, \Z)}:=\{ \F_{e}, e \in E_{n,m}\}\subseteq H^{\text{lf},-}_m(\tilde{C}_{n,m}, \Z)$ by:
$$l_{n,m}: B_n\rightarrow Aut(H^{\text{lf},-}_m(\tilde{C}_{n,m}, \Z)).$$
\end{notation}

For our case, we will use a Poincar\'e-Lefschetz duality, between middle dimensional homologies of the covering space with respect to different parts of its boundary. 
\begin{notation}\label{T2}
For the computational part, we will change slightly the infinity part of the configuration space and denote the following:
\begin{equation*}
\begin{aligned}
H_{n,m}:=&H^{\text{lf},\infty,-}_m(\tilde{C}_{n,m}; \Z) \text{ the homology relative to the infinity part that encodes }\\
& \text{ the open boundary of the covering of the configuration space consisting }\\
&\text{ in the configurations that project in the base space to a multipoint which }\\
& \text { touches a puncture from the punctured disc and relative to the } \\
& \text { boundary part defined above, whose projection contains the base point } w.\\
H^{\partial}_{n,m}:=&H^{lf, \Delta}_{m}(\tilde{C}_{n,m}, \partial^{-}; \Z) \text{ the homolgy relative to the boundary of } \tilde{C}_{n,m} \\
& \text{  which is not in the fiber over w and Borel-Moore}\\
& \text{ corresponding to collisions of points in the configuration space.}
\end{aligned}
\end{equation*}
\end{notation}
The details of this construction are presented in \cite{CrM}. In the sequel, we use the homology classes presented above, seen in the modified version of the homology $H_{n,m}$. However, following \cite{CrM}, all relations between the homology classes still hold in this version of the homology. 
\begin{definition}(Version of the Lawrence representation)\label{T1'}\\
We denote the braid group action in the basis given by multiarcs $\mathscr{B}_{H_{n,m}}:=\{ \F_{e}, e \in E_{n,m}\}\subseteq H_{n,m}$ by:
$$L_{n,m}: B_n\rightarrow Aut(H_{n,m}).$$
\end{definition}
\subsection{Identification between weight space representations and homological representations}
The advantage of the basis presented in the above section on the homological side consists in the fact that it corresponds to the natural basis in the generic weight space on the quantum side. More precisely, we have the following identification.
\begin{notation}
We will use the following specialisation:
\begin{equation}
\begin{cases}
\gamma: \Z[x^{\pm1},d^{\pm1}]\rightarrow \Z[q^{\pm1},s^{\pm1}]\\
\gamma(x)=s^2; \ \ \gamma(d)=q^{-2}.
\end{cases}
\end{equation}
\end{notation}
\begin{theorem}(\cite{Martel})\label{T:1}   
The quantum representation on weight spaces is isomorphic to the homological representation of the braid group:
\begin{equation}
\begin{aligned}
\left( \hat{\varphi}_{n,m}, \ \mathscr{B}_{\hat{V}_{n,m}} \right) \ \ \ & \ \ \ \left( l_{n,m}|_{\gamma}, \ \mathscr{B}_{H^{\text{lf},-}_m(\tilde{C}_{n,m}, \Z)}  \right)\\
B_n\curvearrowright \hat{V}_{n,m} & \simeq H^{\text{lf},-}_m(\tilde{C}_{n,m}, \Z)|{_{\gamma}}\curvearrowleft B_n\\
\Theta(v_{e_1}\otimes ... \otimes v_{e_{n}}) & =\mathscr F_{e}, \text{ for }  e=(e_1,...,e_n) \in E_{n,m}.\\
\end{aligned}
\end{equation}
\end{theorem}
\begin{corollary}\hspace{-1.1mm}(Identification specialised at generic parameters or roots of unity)\label{T:1'''}          
The corresponding specialisations of the representations from Theorem \ref{T:1} are isomorphic as below:
\begin{equation}
\hat{\varphi}_{n,m}|_{\eta_{q,N-1}} \simeq l_{n,m}|_{\psi_{q,N-1}}
\end{equation}
\begin{equation}
\hat{\varphi}_{n,m}|_{\eta_{\xi_N,\lambda}} \simeq l_{n,m}|_{\psi_{\xi_N,\lambda}}
\end{equation}
\end{corollary}
\subsection{Intersection pairing}\label{T3}
Going back to the version of the Lawrence representation that we will use, we present the duality that leads to a topological pairing between the two types of homology of the covering space, discussed in \cite{CrM}, which is defined as follows:
$$< , >:H_{n,m} \otimes H_{n,m}^{\partial}\rightarrow \Z[x^{\pm 1}, d^{\pm 1}].$$

\begin{definition}(Intersection form in the covering space)(\cite{Big}) \label{D:55}   

Let us consider two homology classes $\mathscr F \in H_{n,m}, \mathscr G \in H^{\partial}_{n,m}$. Suppose that these classes are represented by two $m$-manifolds $\tilde{M},\tilde{N}\subseteq \tilde{C}_{n,m}$, which intersect transversely such that:
\begin{equation}
\begin{aligned}
\mathscr F=[\tilde{M}]; \mathscr G=[\tilde{N}] \ \ \ \ \ \ \ \ \ \ \ \ \ \ \ \ \ \ \ \ \ \ \ \ \ \ \ \ \ \ \ \ \\
\text{ card }| \tilde{M} \cap t \tilde{N}  |< \infty, \ \forall t \in Deck(\tilde{C}_{n,m},C_{n,m}).
\end{aligned}
\end{equation}
Then, the intersection form is given by:
\begin{equation}
<[\tilde{M}],[\tilde{N}]>=\sum_{(u,v)\in \Z^{2}} \left(x^ud^v \tilde{M},\tilde{N}\right) x^u d^v.
\end{equation}
(here, $(\cdot, \cdot)$ is the usual geometric intersection number) 
\end{definition}
In the next part, we will see that if the homology classes are given by classes of certain submanifolds which are actually lifts of submanifolds in the base  configuration space, the intersection pairing in the covering space is encoded in the base configuration space. 

More precisely, let us suppose that there exist immersed submanifolds $M,N \subseteq C_{n,m}$ which intersect transversely in a finite number of points such that $
\tilde{M}$ is a  lift of M and  $\tilde{N}$ is a lift of $N$.
\begin{proposition}(Computing the intersection pairing directly from intersections in the base space and the local system \cite{Big})\label{P:3} 

Let $x \in M \cap N$. For each such $x$, we will construct an associated loop $l_x \subseteq C_{n,m}$.
Let us denote the geometric intersection number between $M$ and $N$ in $x$ by $\alpha_x$.

a) {\bf Construction of $l_x$}

Suppose we have two paths $\gamma_M, \delta_N:[0,1]\rightarrow C_{n,m}$ such that:
\begin{equation}
\begin{cases}
\gamma_M(0)={\bf d}; \ \gamma_M(1)\in M; \ \tilde{\gamma}_{M}(1) \in \tilde{M}\\
\delta_N(0)={ \bf d}; \ \delta_N(1)\in N; \ \tilde{\delta}_N(1) \in \tilde{N}
\end{cases}
\end{equation}
where $\tilde{\gamma}_{M},\tilde{\delta}_N$ are the unique lifts of $\gamma_M, \delta_N$ through $\bf \tilde{d}$.

Moreover, consider $\bar{\gamma}_M, \bar{\delta}_N:[0,1]\rightarrow C_{n,m}$ such that:
\begin{equation}
\begin{cases}
Im(\bar{\gamma}_M)\subseteq M; \bar{\gamma}_{M}(0)=\gamma_{M}(1);  \bar{\gamma}_{M}(1)=x\\
Im(\bar{\delta}_N)\subseteq N; \ \bar{\delta}_{N}(0)=\delta_{N}(1);  \ \bar{\delta}_{N}(1)=x.\\
\end{cases}
\end{equation}
Then, consider the loop as follows:
$$l_x:=\delta_N\circ\bar{\delta}_N\circ \bar{\gamma}_M^{-1}\circ \gamma_M^{-1}.$$
\end{proposition}

b) {\bf Formula for the intersection form}

Then, using these loops and the local system we obtain the intersection form as follows:
\begin{equation}\label{eq:1}  
<[\tilde{M}],[\tilde{N}]>=\sum_{x \in M \cap N} \alpha_x \cdot \phi(l_x) \in \Z[x^{\pm1}, d^{\pm1}].
\end{equation}

\section{Globalised definitions over two variables}\label{4}
In this part, we define certain evaluations and coevaluations over two variables, which correspond to the previous evaluations and coevaluations through the two specialisations of the coefficients. This construction will be used further on, when we will show that we can see the coloured Jones and Alexander polynomials from certain formulas over two variables, specialised in two different manners. 
 \begin{remark}
An important point for the further construction concerns the fact that the $R$-matrix used in the construction of the coloured Alexander polynomial (whose formula is presented in \cite{Ito}) can be seen from the generic $R$-matrix described in section \ref{SS:1} through the specialisation $\eta_{\xi_N,\lambda}$. Thus, at the level of braid group representations, we can use the generic $R$-matrix and then specialise it. The difference between the two cases will come from the dualities presented in equation \eqref{E:DualityForCat}.
 \end{remark}
 
 Let us start with an oriented knot $L$ that can be presented as a closure of a braid $\beta_n$ with $n$ strands. The Reshetikhin-Turaev construction is obtained from the corresponding functor applied to the three main levels of the diagram: the caps, the cups and the braid (see \cite{Cr2} for details about this construction):
 \begin{equation} \label{p1}  
 \begin{aligned}
1) \text{ the evaluation}                              
 \ \ \ \ \ \ \ \ \   \uparrow \tikz[x=1mm,y=1mm,baseline=0.5ex]{\draw[<-] (3,0) .. controls (3,3) and (0,3) .. (0,0); \draw[<-] (6,0) .. controls (6,6) and (-3,6) .. (-3,0); \draw[draw=none, use as bounding box](0,0) rectangle (3.5,3);} \ \ \ \ \ \ \ \ \ \ \ \ \ \ \ \ \ \ \ \ \ \ \ \ \ \ \ \ \\
\ \ \ \ \ \ \ \ \ \ \ \ \ \ \  2) \text{ braid level } \ \ \ \ \ \ \ \ \ \ \ \ \ \ \beta_n \  \cup \ \bar{\unit}_{n-1}. \ \ \ \  \ \ \ \ \ \ \ \ \ \ \ \ \ \ \ \ \ \ \ \\
3) \text{ the coevaluation }                                 
 \ \ \ \ \ \  \uparrow  \tikz[x=1mm,y=1mm,baseline=0.5ex]{\draw[<-] (0,3) .. controls (0,0) and (3,0) .. (3,3); \draw[<-] (-3,3) .. controls (-3,-3) and (6,-3) .. (6,3); \draw[draw=none, use as bounding box](-0.5,0) rectangle (3,3);} \ \ \ \ \ \ \ \ \ \ \ \ \ \ \ \ \ \ \ \ \  \ \ \ \ \ \  
\end{aligned}
\end{equation}
As we have seen above, for both quantum invariants coloured Jones case and coloured Alexander case, the braid part comes from the generic action $\hat{\varphi}_n$, as presented in Corollary \ref{defCJ} and Proposition \ref{defADO} (using Proposition \ref{P:6}, Proposition \ref{P:7} and Lemma \ref{L:2}). 
%In the next part, we will use the definitions of the coevaluation and evaluation for the two situations.

Now, we look at the bottom part of the knot $L$, corresponding to the level $3)$ in the diagram \eqref{p1}. We are interested in the image of the corresponding coevaluation. Since coevaluations are morphisms which commute with the $K$ action, we get the following property: 
\begin{equation}\label{eq:6}
\begin{cases}
K\left(\Id \otimes {\tcoev}_{V_{N}^{\otimes {n-1}}}\right)(v_0)=q^{N-1} \left(\Id \otimes {\tcoev}_{V_{N}^{\otimes {n-1}}}\right) (v_0) =\\ \ \ \ \ \ \ \ \ \ \ \ \ \ \ \ \ \ \ \ \ =q^{(2n-1)(N-1)-2{\color{red}(n-1)(N-1)}} \left(\Id \otimes {\tcoev}_{V_{N}^{\otimes {n-1}}}\right)(v_0) \\
K\left(\Id \otimes {\tcoev}_{U_{\lambda}^{\otimes {n-1}}}\right)(v_0)=\xi_N^{\lambda} \left(\Id \otimes {\tcoev}_{U_{\lambda}^{\otimes {n-1}}}\right)(v_0) =\\
 \ \ \ \ \ \ \ \ \ \ \ \ \ \ \ \ =\xi_N^{n\lambda+(n-1)(2N-2-\lambda)-2{\color{red}(n-1)(N-1)}} \left(\Id \otimes {\tcoev}_{U_{\lambda}^{\otimes {n-1}}}\right)(v_0) .\\
\end{cases}
\end{equation}
%\begin{notation}
%We denote by $\hat{V}_{n1,n2,m} \subseteq \hat{V}^{\otimes n_1}\otimes (\hat{V}^{\star})^{\otimes n_2}$ the space of vectors on which $K$ acts by $s^{}$
%\end{notation}
This shows precisely that the coevaluations from above send the vector $v_0$  in a vector from the weight space of weight $(n-1)(N-1)$ in the mixt tensor product with $2n-1$ components (at roots of unity it has the extra term $(n-1)(2N-2-\lambda)$ because in this case the representations are not self dual). However, we will not enter into these details and for our purpose, we will add an extra function such that we arrive in a weight space as presented in definition \ref{D:2}. 
\subsection{Construction }\label{SS:2}
In this part, we will modify the evaluation and coevaluation by an isomorphism such that it will make the computations on the homological side easier. 

\begin{notation}(Generic finite dimensional vector space)\\
We consider the generic vector space over $\Z[q^{\pm 1},s^{\pm 1}]$ generated by the first $N$ vectors in the Verma module:
$$G_N:=<v_0,...,v_{N-1}>\subseteq \hat{V}.$$
\end{notation}
We will introduce the notion of weight spaces corresponding to this subspace, in the generic situation, defined over two variables. They will encode the weight spaces from the finite dimensional modules, which are defined using specialisations of coefficients.
\begin{definition}(Generic weight spaces truncated up to level $N$) 
\label{D:8}\\
Let us consider the following:
\begin{equation}
V^{N-1}_{n,m}:= Ker (K-s^{n}q^{-2m}Id)\subseteq G^{\otimes n}_N.
\end{equation}
In other words, these weight spaces are obtained from the generic ones, by intersection with the finite dimensional part:
\begin{equation}
V^{N-1}_{n,m}=\hat{V}_{n,m}\cap G_{N}^{\otimes n}\subseteq \hat{V}^{\otimes n}.
\end{equation}
\end{definition}
\begin{definition}(Generic evaluations and coevaluations up to level $N$) \label{D:9}\\  
Let us define generic coevaluations and evaluations, up to level $N$, as below:
\begin{equation}\label{eq:0}
\begin{aligned}
& \hspace{-6mm} \bullet {\tcoev}^{\otimes {n-1}}_N:\Z[q^{\pm1}, s^{\pm1}]\rightarrow G_N^{\otimes n-1}\otimes (G_N^{*})^{\otimes n-1}\\
&{\tcoev}^{\otimes {n-1}}_N(1)=\sum_{i_1,...,i_{n-1}=0}^{N-1} v_{i_1}\otimes ... \otimes v_{i_{n-1}} \otimes v_{i_{n-1}}^* \otimes ... \otimes v_{i_1}^*.\\
&\hspace{-6mm} \bullet {\ev}^{\otimes {n-1}}_{N}: G_N^{\otimes n-1}\otimes (G_N^{*})^{\otimes n-1} \rightarrow  \Z[q^{\pm 1},s^{\pm 1}] \\
&{\ev}^{\otimes {n-1}}_{N}(w)=\\
&=s^{(n-1)} \begin{cases}  
q^{-2\sum_{i=k}^{n-1}i_k}, \ \ w=v_{i_1}\otimes ... \otimes v_{i_{n-1}} \otimes v_{i_{n-1}}^* \otimes ... \otimes v_{i_1}^*\\
0, \ \ otherwise.
\end{cases}\\
&\hspace{-6mm} \bullet {\ev}^{\otimes {n-1}}_{\xi_N}: G_N^{\otimes n-1}\otimes (G_N^{*})^{\otimes n-1} \rightarrow  \Z[q^{\pm 1},s^{\pm 1}] \\
&{\ev}^{\otimes {n-1}}_{\xi_N}(w)=\\
&\hspace{-4mm}=s^{(n-1)(1-N)}\begin{cases}  
q^{-2(1-N)\sum_{i=k}^{n-1}i_k}, \ \ w=v_{i_1}\otimes ... \otimes v_{i_{n-1}} \otimes v_{i_{n-1}}^* \otimes ... \otimes v_{i_1}^*\\
0, \  \ otherwise.
\end{cases}
\end{aligned}
\end{equation}
\end{definition}
In the following part, we notice that this lifting of the evaluations and coevaluations towards two variables recovers the original ones, by appropriate specialisations. 

\begin{remark}\label{R:5}
We have the following properties:
\begin{equation}
\begin{aligned}
& \left( {\tcoev}^{\otimes {n-1}}_N \right)|_{\eta_{q,N-1}}={\tcoev}_{{V_N}^{\otimes {n-1}}}\\
& \left( {\tcoev}^{\otimes {n-1}}_N \right)|_{\eta_{\xi_N,\lambda}}={\tcoev}_{{U_{\lambda}}^{\otimes {n-1}}}\\
& \left( {\ev}^{\otimes {n-1}}_N \right)|_{\eta_{q,N-1}}={\ev}_{{V_N}^{\otimes {n-1}}}\\
& \left( {\ev}^{\otimes {n-1}}_{\xi_N} \right)|_{\eta_{\xi_N,\lambda}}={\ev}_{{U_{\lambda}}^{\otimes {n-1}}}\\
\end{aligned}
\end{equation}
\end{remark}
Moreover, we remark that the same property as the one mentioned after  equation \eqref{eq:6} holds for the truncated generic coevaluation. More precisely, the coevaluation  $Id \otimes {\tcoev}^{\otimes {n-1}}_N$ arrives in the weight space of weight $ \color{red}(n-1)(N-1)$ inside the mixt tensor product $G_N^{\otimes n}\otimes (G_N^{*})^{\otimes n-1}$.

Now, we twist the evaluations and coevaluations by a certain morphism which preserves the weights in the tensor product. Let us start with the following isomorphism of vector spaces:
\begin{equation}\label{eq:7}   
\begin{aligned}
& f:(G_N^{*})^{\otimes n-1}\rightarrow G_N^{\otimes n-1}\\
& f \left(v_{i_{n-1}}^{*}\otimes ... \otimes v_{i_1}^{*}\right)=
v_{N-1-i_{n-1}}\otimes ... \otimes v_{N-1-i_1},\\ 
 & \ \ \ \ \ \ \ \ \ \ \ \ \ \ \ \ \ \ \ \ \ \ \ \ \ \ \ \ \ \ \ \ \ \ \ \ \ \ \forall \ 0 \leq i_1,...,i_{n-1} \leq N-1.
\end{aligned}
\end{equation}
%for fixed sequence of polynomial coefficients $c(i_1,...,i_{n-1}) \in \Z[q^{\pm 1},s^{\pm 1}]$.
\begin{definition}(Deformed evaluation and coevaluation)\\ \label{D:5} 
We start with the function $f$ from above. We denote the $f$-deformed evaluations and coevaluation as follows:
\begin{equation}
\begin{aligned}
\bullet \ {\tcoev}^{\otimes {n-1}}_{f}\hspace{-1mm}: & \ \Z[q^{\pm 1},s^{\pm 1}] \rightarrow V^{N-1}_{2n-2,(n-1)(N-1)}\\
&{\tcoev}^{\otimes {n-1}}_{f}=\left( Id^{\otimes n-1}_{G_N}\otimes f\right) \circ {\tcoev}^{\otimes {n-1}}_{N} \\
\bullet \ \ { \ev}^{\otimes {n-1}}_{f} \hspace{-1mm}: & \ V^{N-1}_{2n-2,(n-1)(N-1)}  \rightarrow  \Z[q^{\pm 1},s^{\pm 1}] \\
&{\ev}^{\otimes {n-1}}_{f}=  {\ev}^{\otimes {n-1}}_{N} \circ \left( Id^{\otimes n-1}_{G_N}\otimes f^{-1}\right)\\
\bullet \ \ {\ev}^{\otimes {n-1}}_{f,\xi_N}\hspace{-1mm} : & \ V^{N-1}_{2n-2,(n-1)(N-1)}  \rightarrow  \Z[q^{\pm 1},s^{\pm 1}] \\
&{\ev}^{\otimes {n-1}}_{f,\xi_N}=  {\ev}^{\otimes {n-1}}_{\xi_N} \circ \left( Id^{\otimes n-1}_{G_N}\otimes f^{-1}\right).\\
\end{aligned}
\end{equation}
\end{definition}
\begin{remark}(Well defined)\\
We remark that the function $f$ given by equation \eqref{eq:7} preserves the weights. This property combined with equation \eqref{eq:6} leads to the following: 
\begin{equation} {\tcoev}_f^{\otimes {n-1}}(1) \in V^{N-1}_{2n-2,(n-1)(N-1)}\subseteq G_{N}^{\otimes 2n-2}.
\end{equation}
This shows that the twisted coevaluations from definition \ref{D:5} are indeed well defined with values in the weight space $V^{N-1}_{2n-2,(n-1)(N-1)}$.
\end{remark}
\begin{definition}(Extension of the deformed evaluation)\\ \label{D:5'} 
Let us extend the $f$-deformed evaluations from the ``small'' generic weight spaces presented in definition \ref{D:8} towards the weight spaces from the Verma module as follows:
\begin{equation}\label{eq:21}
\begin{aligned}
\bullet \tilde{ \ev}^{\otimes {n-1}}_{f}: \ & \hat{V}_{2n-2,(n-1)(N-1)}  \rightarrow  \Z[q^{\pm 1},s^{\pm 1}] \\
&\tilde{\ev}^{\otimes {n-1}}_{f}= 
\begin{cases}  
\ev^{\otimes {n-1}}_{f}(w), \ w \in V^{N-1}_{2n-2,(n-1)(N-1)} \\
0, \ \ \ otherwise.
\end{cases}\\
\bullet \tilde{\ev}^{\otimes {n-1}}_{f,\xi_N}: \ & \hat{V}_{2n-2,(n-1)(N-1)}  \rightarrow  \Z[q^{\pm 1},s^{\pm 1}] \\
&\tilde{\ev}^{\otimes {n-1}}_{f,\xi_N}=  \begin{cases}  
\ev^{\otimes {n-1}}_{f,\xi_N}(w), \ w \in V^{N-1}_{2n-2,(n-1)(N-1)} \\
0, \ \ \ otherwise.
\end{cases}\\
\end{aligned}
\end{equation}
\end{definition}
\begin{proposition}\label{P:4}  (Twisted evaluations and coevaluations do not depend on $f$)

If we act with $\beta_n \cup \unit_{n-1}$, we could use the deformed evaluation and coevaluation instead of the usual ones. More specifically, we have the following property:
\begin{equation}\label{eq1} 
\begin{aligned}
\left(Id\otimes \ev^{\otimes {n-1}}_{N} \right)|_{\eta_{q,N-1}} \circ \varphi^{q,N-1}_{{2n-1,(n-1)(N-1)}} \left(\beta_{n} \cup \bar{\mathbb I}_{n-1} \right) \circ \left(\Id \otimes {\tcoev}^{\otimes {n-1}}_{N}\right) |_{\eta_{q,N-1}}&\\
\left(v_0\right)=&\\
=\left(Id\otimes \ev^{\otimes {n-1}}_{f} \right)|_{\eta_{q,N-1}} \circ \varphi^{q,N-1}_{{2n-1,(n-1)(N-1)}} \left(\beta_{n} \cup {\mathbb I}_{n-1} \right) \circ \left(\Id \otimes {\tcoev}^{\otimes {n-1}}_{f}\right)|_{\eta_{q,N-1}}&\\
\left(v_0\right).&\\
\end{aligned}
\end{equation}
\begin{equation}\label{eq1'} 
\begin{aligned}
\left(Id\otimes \ev^{\otimes {n-1}}_{\xi_N} \right)|_{\eta_{\xi_N,\lambda}} \circ \varphi^{\xi_N,\lambda}_{{2n-1,(n-1)(N-1)}} \left(\beta_{n} \cup \bar{\mathbb I}_{n-1} \right) \circ \left(\Id \otimes {\tcoev}^{\otimes {n-1}}_{N}\right)|_{\eta_{\xi_N,\lambda}}&\\
\left(v_0\right)=&\\
=\left(Id\otimes \ev^{\otimes {n-1}}_{f} \right)|_{\eta_{\xi_N,\lambda}} \circ \varphi^{\xi_N,\lambda}_{{2n-1,(n-1)(N-1)}} \left(\beta_{n} \cup {\mathbb I}_{n-1} \right) \circ \left(\Id \otimes {\tcoev}^{\otimes {n-1}}_{f}\right)|_{\eta_{\xi_N,\lambda}}&\\
\left(v_0\right).
\end{aligned}
\end{equation}

\end{proposition}
\begin{proof}
We notice that the part of the tensor product where $f$ acts corresponds to the last $n-1$ components, where the braid is trivial. So all what we have to do is to reverse the orientation of the strait strands from $\bar{\mathbb I}_{n-1}$ to $\mathbb I_{n-1}$ and then cancel $f$ with $f^{-1}$. We draw below the corresponding picture, which encodes the morphisms that we need to compose in order to get the formulas from equation \eqref{eq1} and \ref{eq1'}, reading from bottom to top.

 \begin{equation} \label{p2}  
 \begin{aligned}
1) \text{ evaluation}                              
 \ \ \ \ \  \  \ \ \uparrow \tikz[x=1mm,y=1mm,baseline=0.5ex]{\draw[<-] (3,0) .. controls (3,3) and (0,3) .. (0,0); \draw[<-] (6,0) .. controls (6,6) and (-3,6) .. (-3,0); \draw[draw=none, use as bounding box](0,0) rectangle (3.5,3);} \ \ \ \ \ \ \ \  \ \uparrow \tikz[x=1mm,y=1mm,baseline=0.5ex]{\draw[<-] (3,0) .. controls (3,3) and (0,3) .. (0,0); \draw[<-] (6,0) .. controls (6,6) and (-3,6) .. (-3,0); \draw[draw=none, use as bounding box](0,0) rectangle (3.5,3);} \ \ \ \ \ \ \ 1') \text{ f-evaluation} \ \ \ \\
Id^{\otimes n-1}_{G_N}\otimes f^{-1}  \ \  \ \ \ \ \ \ \ \ \ \ \ \ \ \ \ \ \ \ \ \ \ \ \ \ \\
 2) \text{ braid level } \ \ \ \  \ \  \beta_n \  \cup \ \bar{\unit}_{n-1} \ \ \ \ \ \beta_n \  \cup \ \unit_{n-1} \ \ \ \ \ 2') \text{ braid level } \ \ \ \ \\
Id^{\otimes n-1}_{G_N}\otimes f  \ \ \ \  \ \ \ \ \ \ \ \ \ \ \ \ \ \ \ \ \ \ \ \ \ \ \ \ \\
3) \text{ coevaluation }                                 
 \ \ \ \  \uparrow  \tikz[x=1mm,y=1mm,baseline=0.5ex]{\draw[<-] (0,3) .. controls (0,0) and (3,0) .. (3,3); \draw[<-] (-3,3) .. controls (-3,-3) and (6,-3) .. (6,3); \draw[draw=none, use as bounding box](-0.5,0) rectangle (3,3);} \ \  \ \ \ \ \ \ \ \ \uparrow  \tikz[x=1mm,y=1mm,baseline=0.5ex]{\draw[<-] (0,3) .. controls (0,0) and (3,0) .. (3,3); \draw[<-] (-3,3) .. controls (-3,-3) and (6,-3) .. (6,3); \draw[draw=none, use as bounding box](-0.5,0) rectangle (3,3);} \ \ \ \ \ 3) \text{ f-coevaluation } \ 
\end{aligned}
\end{equation}
Now, if we specialise the relation corresponding to the diagram \ref{p2} from above through $\eta_{q,N-1}$ we obtain equation \eqref{eq1}. Similarly, the specialisation $\eta_{\xi_N,\lambda}$ leads to equation \eqref{eq1'}. 
\end{proof}
%\subsection{Coefficients }\label{SS:3}
For our purpose, we will need to pass from a multiarc with all multiplicities less than $N-1$ to the corresponding code sequence. In order to do this, we will need to invert the correspoding coefficients, which are given by the following formula:
$$\mathscr C_{i_1,..,i_{n-1}}=\prod_{k=1}^{n-1}(i_k)_{q^{-2}}!~(N-1-i_k)_{q^{-2}}!, \ \ 0 \leq i_1,...,i_{n-1} \leq N-1.$$
This motivates the following definition.

\begin{definition}(Choice of normalisation)
We will work at a certain point over a slightly bigger ring, where we invert the quantum factorials smaller than $N-1$ and use the following notations:
\begin{equation}
\begin{aligned}
&\tilde{\Li}_N:=\Z[q^{\pm 1},s^{\pm 1}](I_N)^{-1} \subseteq \Q(q,s) \ \  \\
&I_N:=<(k)_{q^{-2}} \mid  0 \leq k \leq N-1 >.
\end{aligned}
\end{equation}
($I_N$ is the multiplicatively closed system generated by the above quantum numbers)
Then, let us consider the inclusion map $$\iota_N: \Z[q^{\pm 1},s^{\pm 1}] \rightarrow \Z[q^{\pm 1},s^{\pm 1}](I_N)^{-1}.$$
Further on, we will use the specialisation map
$\tilde{\gamma}_N:\Z[x^{\pm},d^{\pm}]\rightarrow \tilde{\Li}_N$, given by 
$$\tilde{\gamma}_N:=\iota_N \circ \gamma.$$
\end{definition}
We have presented in figure \ref{fig:7} all these specialisations as well as the relations between them.
%\begin{remark}
%The specialisations $\eta_{q,N-1}$ and $\eta_{\xi_N,\lambda}$ lead towards well defined extensions $\tilde{\eta}_{q,N-1}$ and $\tilde{\eta}_{\xi_N,\lambda}$ on $\tilde{\Li}_N$. Here, we use that $\xi_N$ is a root of unity of order $2N$ and so the corresponding quantum factorials from $I_N$ do not vanish, when specialised at this root of unity. \end{remark}

In the following part we will discuss about the precise formulas for the dualities that lead to the two quantum invariants. We start from the remark that the evaluations for the coloured Jones case and for the coloured Alexander case ( presented in proposition \ref{P:2}), differ by the following factor: the first one is twisted the $K$ action whereas the evaluation at roots of unity has the action of the element $K^{1-N}$. Now, we remind that we can use the evaluations and coevaluations over two variables $q,s$ and recover the $U_q(sl(2))$ dualities through the two specialisations (as we have discussed in definition \ref{D:9} and remark \ref{R:5}). More precisely, we have the following formulas:
\begin{equation}\label{D:3}  
\begin{aligned}
\ev_{V^{\otimes n-1}_{N}}\left(w\right)&=\\
=\eta_{q,N-1} \circ &  
\begin{cases}
\begin{aligned}
&s^{(n-1)}q^{-2\sum_{i=k}^{n-1}i_k}, w=v_{i_1}\otimes ... \otimes v_{i_{n-1}} \otimes v_{i_{n-1}}^* \otimes ... \otimes v_{i_1}^*\\
&0, \ \ \ \ \ \ \ \ \ \ \ \ \ \ \ \ \ \ \ \ \ otherwise
\end{aligned}
\end{cases}\\
\end{aligned}
\end{equation}
\begin{equation}\label{E:4}  
\begin{aligned}
\ev & _{U^{\otimes n-1}_{\lambda}} \left(w\right) =\\
= \eta_{\xi_N,\lambda} \circ & 
\begin{cases}
\begin{aligned}
&s^{(n-1)(1-N)}q^{-2(1-N)\sum_{i=k}^{n-1}i_k}, w=v_{i_1}\otimes ... \otimes v_{i_{n-1}} \otimes v_{i_{n-1}}^* \otimes ... \otimes v_{i_1}^*\\
&0, \ \ \ \ \ \ \ \ \ \ \ \ \ \ \ \ \ \ \ \ \ otherwise.
\end{aligned}
\end{cases}
\end{aligned}
\end{equation}
%\end{remark}
We will keep in mind these coefficients, which motivate the following definiton.
\begin{notation}
For a fixed set of indices $i_1,...,i_{n-1} \in \{0,...,N-1 \}$, let us consider the corresponding polynomials in $\Z[q^{\pm1},s^{\pm 1 }]$ that encode these evaluation maps as follows:
\begin{equation}\label{E:2}
\begin{cases}
p_n^{N}(i_1,..,i_{n-1})=s^{(n-1)}q^{-2\sum_{i=k}^{n-1}i_k}\\
p_n^{\xi_N}(i_1,..,i_{n-1})=s^{(n-1)(1-N)}q^{-2(1-N)\sum_{k=1}^{n-1}i_k}.
\end{cases}
\end{equation}
\end{notation}
With these notations, we have that:
\begin{equation}
\begin{cases}
\eta_{q,N-1} \circ p_n^{N}(i_1,..,i_{n-1})=~ \ev_{V^{\otimes n-1}_{N}}v_{i_1}\otimes ... \otimes v_{i_{n-1}} \otimes v_{i_{n-1}}^* \otimes ... \otimes v_{i_1}^*)\\
\eta_{\xi_N,\lambda} \circ p_n^{\xi_N}(i_1,..,i_{n-1})=~ \ev_{U^{\otimes n-1}_{\lambda}}(v_{i_1}\otimes ... \otimes v_{i_{n-1}} \otimes v_{i_{n-1}}^* \otimes ... \otimes v_{i_1}^*).
\end{cases}
\end{equation}
\begin{remark}\label{R:9}
In the next sections we will use the extended evaluations introduced in definition \ref{D:5'}. Using remark \ref{R:5} together with the notations from above, we conclude the following property.

\begin{equation*} \label{eq:22}
\begin{aligned}
 \bullet \ & \tilde{ \ev}^{\otimes {n-1}}_{f}: \  \hat{V}_{2n-2,(n-1)(N-1)}  \rightarrow  \Z[q^{\pm 1},s^{\pm 1}] \\
&\tilde{\ev}^{\otimes {n-1}}_{f}(w)= 
\begin{cases}  
 p_n^{N}(i_1,..,i_{n-1}) , \ \text{ if } \exists \  0 \leq i_1,...,i_{n-1} \leq\ N-1 \text{ such that }\\
 \hspace{24mm} w=v_{i_1}\otimes ... \otimes v_{i_{n-1}} \otimes v_{N-1-i_{n-1}}\otimes ... \otimes v_{N-1-i_1} \\
\hspace{8mm }0, \hspace{35mm} otherwise.
\end{cases}\\
 \bullet \ & \tilde{\ev}^{\otimes {n-1}}_{f,\xi_N}: \  \hat{V}_{2n-2,(n-1)(N-1)}  \rightarrow  \Z[q^{\pm 1},s^{\pm 1}] \\
&\tilde{\ev}^{\otimes {n-1}}_{f,\xi_N}(w)= 
\begin{cases}  
p_n^{\xi_N}(i_1,..,i_{n-1}) , \ \text{ if } \exists \  0 \leq i_1,...,i_{n-1} \leq\ N-1 \text{ such that } \\
 \hspace{24mm} w=v_{i_1}\otimes ... \otimes v_{i_{n-1}} \otimes v_{N-1-i_{n-1}}\otimes ... \otimes v_{N-1-i_1} \\
\hspace{8mm }0, \hspace{35mm} otherwise.
\end{cases}\\
\end{aligned}
\end{equation*}

\end{remark}
\subsection{Change of the homological basis}\label{}
Now, we go back to the homological side. We will see in the main proof that for our case it will be more convenient to use code sequences instead of multiarcs. In the following part we prove that actually for the braid actions which occur in our constructions, those which have identity on the last strands, we will be able to use these elements instead.

Let us fix a colour $N \in \N$ and two parameters $n,m \in \N $. The identification from Theorem \ref{T:1} leads to the following correspondence:
\begin{equation} \label{L:4}  
\begin{aligned}
(B_n \cup \mathbb I_{n-1}) \curvearrowright \hat{V}_{2n-1,m} & \simeq H^{\text{lf},-}_{m}(\tilde{C}_{2n-1,m}, \Z)|{_{\gamma}}\curvearrowleft (B_n \cup \mathbb I_{n-1})\\
\Theta(v_{e_1}\otimes ... \otimes v_{e_{2n-1}})& =\mathscr F_{e}, \ \ \ \forall \  e=(e_1,...,e_{n-1}) \in E_{2n-1,m}.
\end{aligned}
\end{equation}
We want to pass from the multiarc $\mathscr F_e$ to the corresponding code sequence. In order to do that, we use the following algebraic remark. 
\begin{lemma}\label{L:1}
Let $g: \hat{V}_{2n-1,m}|_{\iota_N} \rightarrow \hat{V}_{2n-1,m}|_{\iota_N}$ be an isomorphism of modules such that:
\begin{equation}
\begin{aligned}
g(v_{e_1}\otimes...\otimes v_{e_{2n-1}})= y(e_{n+1},...,e_{2n-1}) \ & v_{e_1} \otimes...\otimes v_{e_{2n-1}},\\
& \forall e=(e_1,...,e_{2n-1}) \in E_{2n-1,m}
\end{aligned}
\end{equation}
 with $y(e_{n+1},...,e_{2n-1}) \in \tilde{\Li}_N$. 
 In other words, $y$ can be seen as a rational function which depends only on the last $n-1$ coordinates of the $(2n-1)$-tuple $e$.
Then $$\Theta\circ g: \hat{V}_{2n-1,m}|_{\iota_N} \simeq H^{\text{lf},-}_{m}(\tilde{C}_{2n-1,m}, \Z)|{_{\tilde{\gamma}_N}}$$ is still an equivariant isomorphism with respect to the $B_n \cup \mathbb I_{n-1}$-action and the bases correspond as follows:
$$\Theta\circ g \left(g^{-1}(v_{e_1}\otimes ... \otimes v_{e_{2n-1}})\right)=\mathscr F_{e}.$$.
\end{lemma}
\begin{proof}
We notice that the isomorphism $g$ is equivariant with respect to the action of braids from $B_n \cup \mathbb I_{n-1} $. Combining this with the property that $\Theta$ is equivariant with respect to the $B_{2n-1}$-action, we obtain the desired correspondence. 
\end{proof}
\begin{remark}
This lemma will allow us to use the code sequences instead of the normalised multiarcs in the specialised cases. We will have to work over $\tilde{\Li}_N$, but actually the elements that we are interested in will be defined over $\Li$, so overall we will obtain results over $\Li$.
\end{remark}
\section{Construction of the generic homology classes} \label{5}

In this section we aim to construct certain homology classes in the generic Lawrence representation, which correspond to the evaluation and coevaluation on the quantum side. Since the evaluations for the root of unity case and generic case differ by a coefficient, we will take this into account in our construction. In the following two sections, we will prove that these classes specialised to the two cases, the generic case and the root of unity case lead to the coloured Jones polynomial and coloured Alexander polynomial respectively. 

In the first part, we aim to construct a homology class in the dual homology of the covering  $H^{\partial}_{2n-1,(n-1)(N-1)}$ that will encode the algebraic evaluation. This will be done in two main steps: first we will find the right manifold in the base configuration space and secondly we will choose a good lift to the covering space. 

\subsection*{ Step I- Choice of a good geometric support}

\

For the following part, for $i_1,..,i_{n-1}\in \{0,...,N-1\}$ we will denote by $$\bar{U}_{0,i_1,...,i_{n-1},N-1-i_{n-1},...,N-1-i_{1}}$$ the submanifold which has the same support as $U_{0,i_1,...,i_{n-1},N-1-i_{n-1},...,N-1-i_{1}}$, but which is oriented using the orientations of each individual red segment as in picture \ref{fig4}. More specifically, we change the orientations of the red segments which arrive in a puncture with the index bigger than $n$.
\begin{remark}\label{sign}
The effect of this change of orientations on the corresponding homology classes (given by lifts of these submanifolds) will be that they will differ by the sign $(-1)^{n-1}$.
\end{remark}
Let us fix an arbitrary lift of the submanifold $\bar{U}_{0,i_1,...,i_{n-1},N-1-i_{n-1},...,N-1-i_{1}}$ in $\tilde{C}_{2n-1,(n-1)(N-1)}$ and denote its class by $$\tilde{\mathscr U}_{0,i_1,...,i_{n-1},N-1-i_{n-1},...,N-1-i_{1}} \in H_{2n-1,(n-1)(N-1)}.$$
\begin{figure}[H]
$${\color{red} \Huge \tilde{\mathscr U}_{0,i_1,...,i_{n-1},N-1-i_{n-1},...,N-1-i_{1}} \in H_{2n-1,(n-1)(N-1)}} \ \ \ \ \ \ \ \ \ \ \ \ \ {\color{green} \tilde{G}_{n}^N \in H^{\partial}_{2n-1,(n-1)(N-1)}} \ \ \ \ \ 
$$
\centering
\includegraphics[scale=0.5]{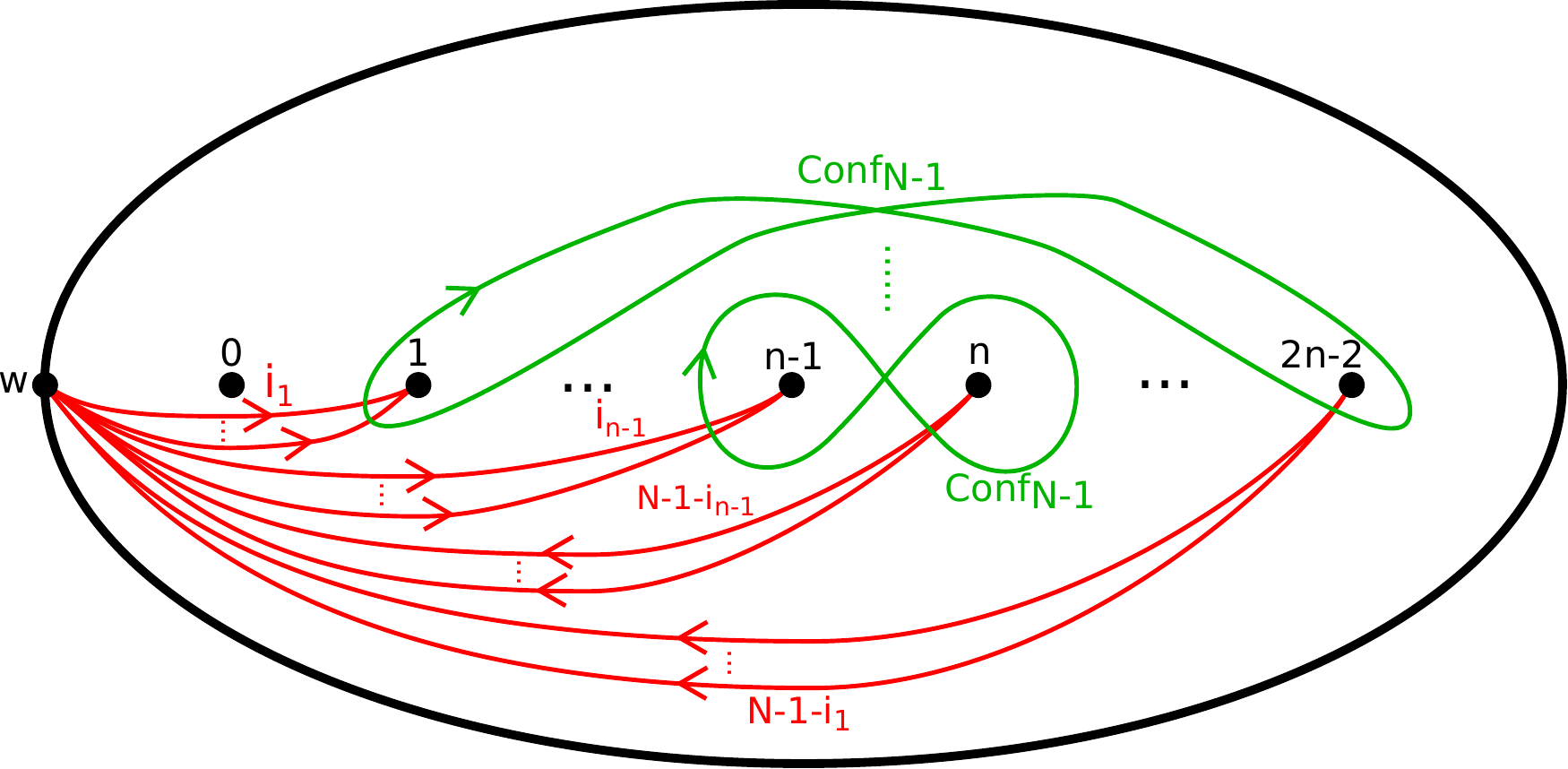}
$${\color{red} \Huge  \bar{U}_{0,i_1,...,i_{n-1},N-1-i_{n-1},...,N-1-i_{1}}} \ \ \ \ \ \ \ \ \ \ \ \ {\color{green} {G}_{n}^N} \ \ \ \ \ \ \ \    C_{2n-1,(n-1)(N-1)} 
$$
\vspace{-5 mm}
\caption{}
\label{fig4}
\end{figure}
\begin{definition}(Dual manifold given by an arbitrary lift)

Let us consider the immersed submanifold $ G_{n}^N \subseteq C_{2n-1,(n-1)(N-1)}$ given by the product of $(n-1)$ configuration spaces of $N-1$ points on eights around the symmetric punctures, as in figure \ref{fig4}. 

For the first part, we consider an arbitrary lift of ${G}_{n}^N$ to the covering $\tilde{C}_{2n-1,(n-1)(N-1)}$ and denote its homology class by $$\tilde{G}_{n}^N \in H^{\partial}_{2n-1,(n-1)(N-1)}.$$ 
\end{definition}
\begin{lemma}(Figure-eight intersection)\label{R:2}\\
The intersection pairing between these classes has the following form:
\begin{equation*}
<\tilde{\mathscr U}_{0,i_1,...,i_{n-1},N-1-i_{n-1},...,N-1-i_{1}}, \tilde{G}_{n}^N>=x^{m(i_1,...,i_{n-1})}d^{m'(i_1,...,i_{n-1})}
\end{equation*}
where $m,m'$ are polynomials in $(n-1)$ variables, which depend on the particular choice of the lift $\tilde{G}_{n}^N$ in the covering space.
\end{lemma}
\begin{proof}
We notice that each $i \in \{1,...,n-1\}$, the space of configurations of $(N-1)$ points on the fixed figure eight around the punctures $i$ and $2n-1-i$ intersects uniquely the collection of red segments with multiplicities ($e_i$, $N-1-e_i$). 

Doing this for all pairs indexed by $i \in \{1,...,n-1\}$, we conclude that $$\bar{U}_{0,i_1,...,i_{n-1},N-1-i_{n-1},...,N-1-i_{1}}  \ \ \ \ \text {    and    } \ \ \ \ G_{n}^N$$ intersect in an unique point $\bf p$ in the configuration space $C_{2n-1,(n-1)(N-1)}$. Thus, in order to compute the intersection between their lifts:
$\tilde{\mathscr U}_{0,i_1,...,i_{n-1},N-1-i_{n-1},...,N-1-i_{1}}$ and $ \tilde{G}_{n}^N$, we only need to compute the scalar that corresponds to the point $\bf p$. The definition of the intersection form from definition \ref{D:55}, shows that this scalar will be an element of the deck transformation group, given by a monomial in $x$ and $d$ which we denote by $x^{m(i_1,...,i_{n-1})}d^{m'(i_1,...,i_{n-1})}$.

Related to the sign that occurs in this computation, we notice that for all intersection points from above, the orientation given by the tangent vectors at the red segments and the tangent vector at the figure eight is always positive.
\end{proof}
\subsection*{ Step II- Choice of particular lifts} \label{S:1}  

\

Now, we pass to the second part and we will proceed explicit computations. We aim to choose particular lifts for the submanifolds  $\bar{U}_{0,i_1,...,i_{n-1},N-1-i_{n-1},...,N-1-i_{1}}$ and  ${G}_{n}^N$ such that their intersection pairing is easy to compute. More precisely, we will define two paths in the base configuration space, whose lifts through $\bf \tilde{d}$ will prescribe the lifts for the two submanifolds:

\vspace{-4.5mm}
\begin{figure}[H]
\centering
$${\color{red} \tilde{\mathscr U}_{0,i_1,...,i_{n-1},N-1-i_{n-1},...,N-1-i_{1}} \in H_{2n-1,(n-1)(N-1)}} \ \ \text{ and }\ \  {\color{green}\mathscr G_{n}^N \in H^{\partial}_{2n-1,(n-1)(N-1)}}.$$
\includegraphics[scale=0.5]{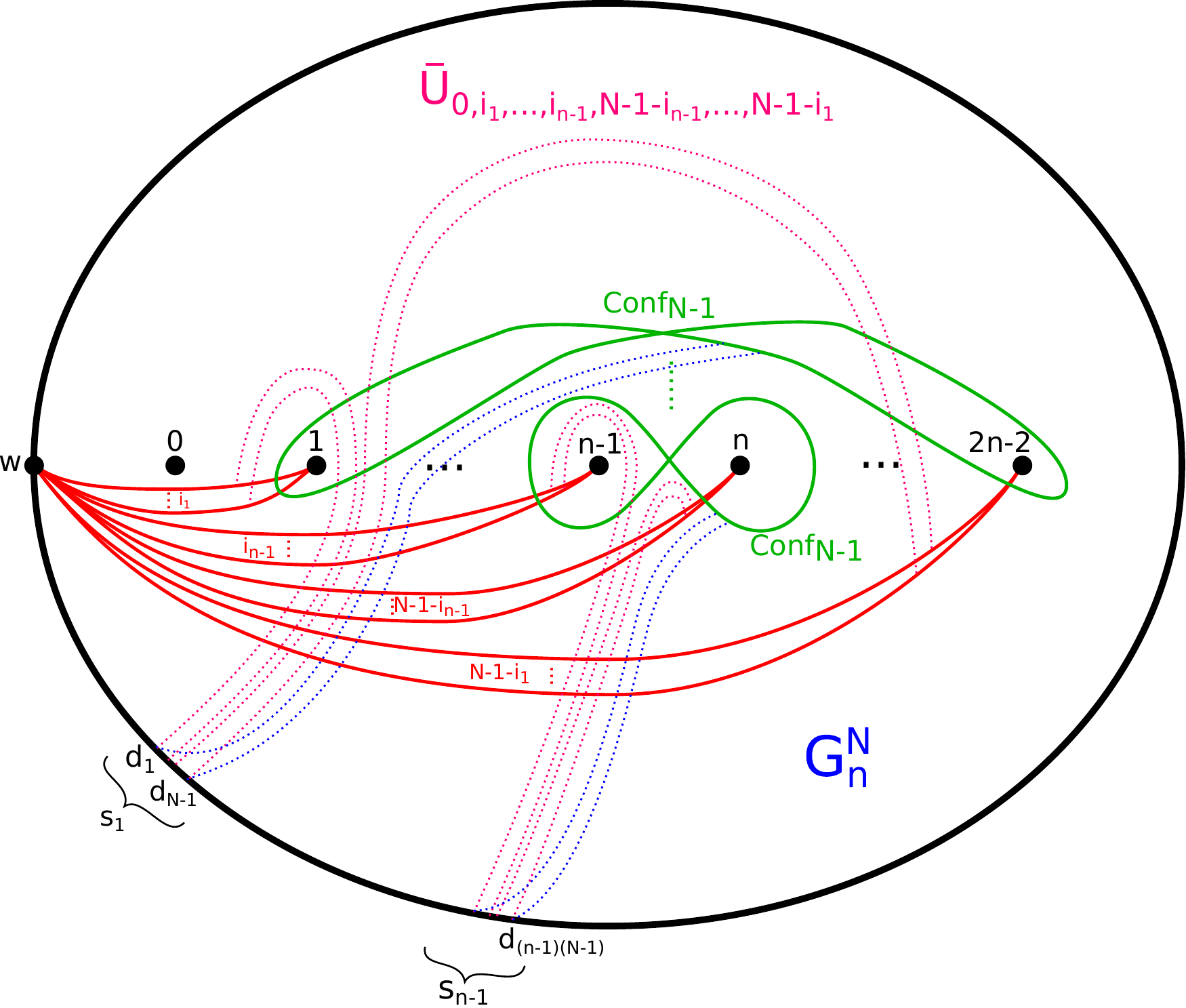}
\vspace{-3mm}
\caption{}
\label{fig5}
\end{figure}
%\vspace{-30mm}
\begin{definition}(Paths in the configuration space)\\ \label{D:4} 
Let us start by splitting the set of base points into $(n-1)$ sets, each of which has $N-1$ points as follows:
\begin{equation}
\begin{cases}
S_1=\{d_1,...,d_{N-1}\}\\
S_2=\{d_N,...,d_{2N-1}\}\\
...\\
S_{n-1}=\{d_{(N-1)(n-2)+1},...,d_{(N-1)(n-1)}\}. 
\end{cases}
\end{equation}
Then, from each set $S_k \ (k \in \{1,...,n-1 \})$, we construct the following paths:
\begin{itemize}
\item[1)] $N-1$ paths to the red segments which end at the punctures $(k, 2n-1-k)$
\item[2)] $N-1$ paths towards the figure eight which ends/ go around the punctures $(k, 2n-1-k)$, as in figure \ref{fig5}. 
\end{itemize}
We denote these sets of paths, which form two paths in the configuration space $C_{2n-1,(n-1)(N-1)}$ by:
\begin{itemize}
\item[1)] $\sigma^F_{i_1,...,i_{n-1}}$-~from the base point ${\bf d}=\{d_1,...,d_{(n-1)(N-1)}\} $ towards the submanifold $\bar{U}_{0,i_1,...,i_{n-1},N-1-i_{n-1},...,N-1-i_{1}}$
\item[2)] $\sigma^G$-~from the base point ${\bf d}=\{d_1,...,d_{(n-1)(N-1)}\}$ to ${G}_{n}^N$. 
\end{itemize}
\end{definition}
\begin{notation}(Lifts of the base paths)\label{N:2}  

Let us consider $\tilde{\sigma}^F_{i_1,...,i_{n-1}}$ and $\tilde{\sigma}^G$ to be the lifts of the paths $\sigma^F_{i_1,...,i_{n-1}}$ and $\sigma^G$ to the covering, which satisfy the following:
\begin{equation}
\tilde{\sigma}^F_{i_1,...,i_{n-1}}(0)={\bf \tilde{d}} \ \ \ \text{ and } \ \ \tilde{\sigma}^G(0)={\bf \tilde{d}}.
\end{equation}
\end{notation}
\begin{definition}(Homology classes given by these paths)\label{D:10}

a) Let $\tilde{\mathscr U}_{0,i_1,...,i_{n-1},N-1-i_{n-1},...,N-1-i_{1}} \in H_{2n-1,(n-1)(N-1)}$ be the homology class given by the lift of $ \bar{U}_{0,i_1,...,i_{n-1},N-1-i_{n-1},...,N-1-i_{1}}$ through $\tilde{\sigma}^F_{i_1,...,i_{n-1}}(1)$. 

b) Similarly, we define $ \mathscr G_{n}^N \in H^{\partial}_{2n-1,(n-1)(N-1)}$ to be the class given by the lift of the immersed submanifold $ G_{n}^N$ in the covering of the configuration space, through $\tilde{\sigma}^G(1)$. 
\end{definition}
\begin{lemma}(Intersection pairing between the chosen homology classes)\label{R:1}\\
The intersection pairing between these particular lifts gives the following:
\begin{equation}
\begin{aligned}
<\tilde{\mathscr U}_{0,i_1,...,i_{n-1},N-1-i_{n-1},...,N-1-i_{1}}, & \mathscr G_{n}^N>=1\\
  & \forall \  i_1,...,i_{n-1} \in \{0,...,N-1\}.
\end{aligned}
\end{equation}
\end{lemma}
\begin{proof}
We compute the intersection pairing using the formula presented in equation \eqref{eq:1}. In order to do this, we have three main steps: 
\begin{enumerate}
\item first we need to compute the sign which corresponds to the geometric intersection
\item find out the power of the variable $d$ 
\item compute the power of the variable $x$ which corresponds to the appropriate deck transformation.
\end{enumerate}
{\bf 1) Sign from the orientations } As we have seen in Lemma \ref{R:2}, the submanifolds $\bar{U}_{0,i_1,...,i_{n-1},N-1-i_{n-1},...,N-1-i_{1}}$ and $G_n^N$ intersect exactly in one point in the configuration space and all the local signs coming from their orientations are positive. 

In order to compute the next part, we introduce the following notations. For $k \in \{1,...,n-1\}$ let us denote by:
\begin{itemize}
\item $x^k_1,...,x^k_{i_k}$ to be the intersection points between the $k^{th}$ figure eight and the red segments with multiplicity $i_k$
\item $x^k_{i_k+1},...,x^k_{N-1}$ to be the intersection points between the $k^{th}$ figure eight and the red segments with multiplicity $N-1-i_k$.
\end{itemize}
We denote the corresponding point in $C_{2n-1,(n-1)(N-1)}$ by: $$ {\bf x}= \left\lbrace x^1_{1},..,x^{1}_{N-1},..., x^{n-1}_{1},..,x^{n-1}_{N-1} \right\rbrace.$$
\begin{figure}[H]
\centering
\includegraphics[scale=0.5]{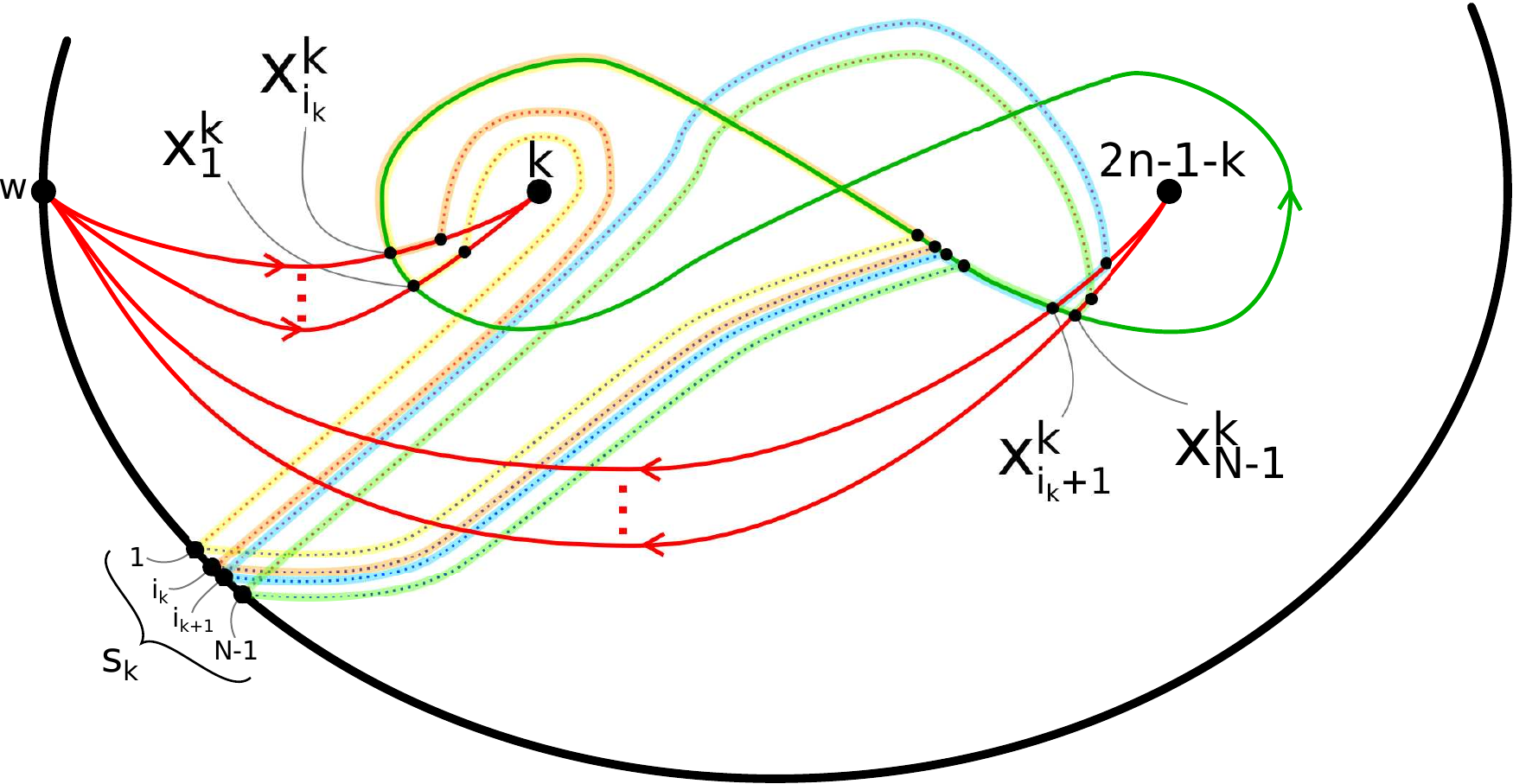}
\caption{}
\label{fig6}
\end{figure}
This multipoint gives exactly the intersection between the two submanifolds in the configuration space:
$$\bar{U}_{0,i_1,...,i_{n-1},N-1-i_{n-1},...,N-1-i_{1}} \cap G_n^N= {\bf x}$$
With this notation, we have $\alpha_{\bf x}=1$.
Then, the intersection pairing between the two lifts $\tilde{\mathscr U}_{0,i_1,...,i_{n-1},N-1-i_{n-1},...,N-1-i_{1}}$ and $ \mathscr G_{n}^N$ is given by the element from the deck transformation corresponding to this unique intersection point.  

Now we look at the loop $l_{\bf x}$ associated to $\bf x$. For each intersection point $x^k_r$, let $l_{x^k_{r}}$ be the corresponding path, constructed following the recipe given in proposition \ref{P:3}. We have drawn in figure \ref{fig6} the loops 
\definecolor{corn}{rgb}{0.98, 0.75, 0.6}
\definecolor{babyblue}{rgb}{0, 0.8, 0.92}
$$\color{black!20!yellow} l_{x^k_{1}},...,\color{orange}l_{x^k_{i_k}}, \color{babyblue}l_{x^k_{i_{k+1}}},...,\color{green}l_{x^k_{N-1}}$$
corresponding to the intersection points around the punctures $(k,2n-1-k)$.

Then, we consider the loop in the configuration space given by the union of these paths:
$$l_{\bf x}=\left\lbrace  l_{x^1_1},...,l_{x^{n-1}_{N-1}} \right\rbrace.$$
Using formula \eqref{eq:1}, we obtain that the intersection between the two lifts can computed using the loop $l_{\bf x}$ and the local system: 
\begin{equation}
<\tilde{\mathscr U}_{0,i_1,...,i_{n-1},N-1-i_{n-1},...,N-1-i_{1}}, \mathscr G_{n}^N>=\phi(l_{\bf x}).
\end{equation}
In the sequel, we compute $\phi(l_{\bf x})$ in two steps.

{\bf 2) Rotation in the configuration space: coefficient of $d$}

Grace to the choice of the base points, we remark that the paths $l_{x^1_1},...,l_{x^{n-1}_{N-1}}$ are all closed loops. Moreover, we notice that they have no winding number one around the other (as paths in the configuration space). This shows that the exponent corresponding to the power of $d$ in the evaluation of the loop $l_{\bf x}$ is zero.

{\bf 3) Winding number around punctures: coefficient of $x$}

The last part that remains to be checked concerns the winding around punctures. Looking at each individual loop $l_{x^k_r} $ in the punctured disc, we see that it has no winding around any puncture.

 These two steps show that the variables $x$ and $d$ do not appear in the intersection formula. Putting everything together, we conclude that:
$$\phi(l_{\bf x})=1$$ and so we have the desired formula for the intersection pairing.
 \end{proof}
\subsection{First and Second homology classes}
\begin{definition}(Second homology class)\label{shc}

We fix the class given by this choice of the lift $\mathscr G_{n}^N \in H^{\partial}_{2n-1,(n-1)(N-1)}$ to be the homology class which will be used for the model from Theorem \ref{THEOREM}.
\end{definition}

In this moment, we have two types of homology classes which are geometrically dual one to the other. In the next part we aim to construct the first homology class for our topological model. Let us start with the class $\mathscr{P}_n^N$ which is the sum of all classes given by manifolds with $\bar{U}$-supports as below.
\begin{definition}\label{D:7}  
Let us consider $\mathscr{P}_n^N\in H_{2n-1,(n-1)(N-1)}$ given by:
\begin{equation}
\mathscr P_n^N:=\sum_{i_1,...,i_{n-1}=0}^{N-1} \tilde{\mathscr U}_{0,i_1,...,i_{n-1},N-1-i_{n-1},...,N-1-i_{1}}.
\end{equation}
\end{definition}
This will be the direct correspondent of the coevaluation from the algebraic side. 

On the other hand, since the algebraic evaluation is twisted by the action of the quantum group, we need to take this into account on the homological side. We choose to modify the first class $\mathscr P_n^N$ (encoding this twisting) and keep $\mathscr G_{n}^N$ as the second homology class which will correspond to the evaluation without any twisting.
\begin{definition}(Global class over two variables)\label{fhc}

Let us consider $\mathscr{E}_n^N\in H_{2n-1,(n-1)(N-1)}$ given by:
\begin{equation}
\mathscr E_n^N:=\sum_{i_1,...,i_{n-1}=0}^{N-1} d^{\sum_{k=1}^{n-1}i_k} \  \tilde{\mathscr U}_{0,i_1,...,i_{n-1},N-1-i_{n-1},...,N-1-i_{1}}.
\end{equation}
\end{definition}
In order to make the connection with the evaluations that occur for the generic case and the root of unity case respectively, the computations of $p_n^N$ and $p_n^{\xi_N}$ from equation \eqref{E:2} lead us to the following classes.
\begin{definition}(Global classes for the generic and roots of unity case)\label{E:3}\\
Let us consider $\mathscr{F}_n^N, \mathscr{F}_n^{\xi_N}\in H_{2n-1,(n-1)(N-1)}|_{{\gamma}}$ given by:
\begin{equation}
\mathscr F_n^N:=s^{(n-1)} \mathscr E_n^N|_{{\gamma}}; \ \ \  \mathscr F_n^{\xi_N}:=s^{(1-N)(n-1)} \mathscr E_n^N|_{{\gamma}}.
\end{equation}
More precisely, we have that:
\begin{equation}\label{eq:19} 
\begin{aligned}
\mathscr F_n^N:= \ \ \ s^{(n-1)} \ \ \ \sum_{i_1,...,i_{n-1}=0}^{N-1} q^{-2 \sum_{k=1}^{n-1}i_k}   \tilde{\mathscr U}_{0,i_1,...,i_{n-1},N-1-i_{n-1},...,N-1-i_{1}}|_{\gamma}\\
\mathscr F_n^{\xi_N}:=s^{(1-N)(n-1)}\sum_{i_1,...,i_{n-1}=0}^{N-1} q^{-2\sum_{k=1}^{n-1}i_k} \tilde{\mathscr U}_{0,i_1,...,i_{n-1},N-1-i_{n-1},...,N-1-i_{1}}|_{\gamma}.
\end{aligned}
\end{equation}
\end{definition}
We end this section with a remark which will be very useful for the topological models from the next two sections. 

\begin{proposition}(Relation between specialisations of the homology classes)\label{P:8}\\ 
With these notations, we have the following properties:
\begin{equation}
\begin{aligned}
&\mathscr F_n^N|_{\eta_{q,N-1}}=\sum_{i_1,...,i_{n-1}=0}^{N-1} \eta_{q,N-1} \circ p_n^{N}(i_1,..,i_{n-1})\\
&\hspace{40mm}  \tilde{\mathscr U}_{0,i_1,...,i_{n-1},N-1-i_{n-1},...,N-1-i_{1}}|_{\psi_{q,N-1}}\\
&\mathscr F_n^{\xi_N}|_{\eta_{\xi_N,\lambda}}:=\sum_{i_1,...,i_{n-1}=0}^{N-1} \eta_{\xi_N,\lambda} \circ p_n^{\xi_N}(i_1,..,i_{n-1}) \\
&\hspace{40mm} \tilde{\mathscr U}_{0,i_1,...,i_{n-1},N-1-i_{n-1},...,N-1-i_{1}}|_{\psi_{\xi_N,\lambda}}.
\end{aligned}
\end{equation}
\end{proposition}
\begin{proof}
We remind the following property regarding the specialisations of coefficients (see diagram \ref{fig:7}):
\begin{equation}
\psi_{q,N-1}=\eta_{q,N-1} \circ \gamma.
\end{equation}
The first equation comes directly from the definition of the class $\mathscr F_n^N$ and the formula for $p_n^N$  from equation \eqref{E:2}.
The second formula has a subtlety. There, the specialisation at roots of unity plays a very important role and the key remark is the following relation:
\begin{equation}
\eta_{\xi_N,\lambda} \left( p_n^{\xi_N}(i_1,..,i_{n-1}) \right)= \eta_{\xi_N,\lambda}\left(   s^{(1-N)(n-1)} q^{-2\sum_{k=1}^{n-1}i_k} \right).
\end{equation}
This equation combined with the definition of the class $\mathscr F_n^{\xi_N}$ concludes the second relation. 
\end{proof}
\section{Topological intersection model for the coloured Jones invariants} \label{6}
In this part, we prove the topological model for the coloured Jones polynomials, as stated in Theorem \ref{THEOREM}. We remind the description of this invariant, as presented in equation \eqref{eq:J}:
\begin{equation}\label{eq:J1}
\begin{aligned}
J_N(L,q)=&q^{-(N-1)w(\beta_n)} \ \pi \circ \\
&\left(\left(Id\otimes \ev_{V_N^{\otimes {n-1}}} \right) \circ \varphi^{V_{N}}_{2n-1} \left(\beta_{n} \cup \bar{\mathbb I}_{n-1} \right) \circ \left(\Id \otimes {\tcoev}_{V_N^{\otimes {n-1}}}\right) \right) (v_0).
\end{aligned}
\end{equation}
\subsection{Step I-Invariant through the weight spaces}\label{SCJ:1}
Let us start with the morphism that corresponds to the bottom part of the diagram. The first step is based on the remark which tells us that this morphism arrives in a particular weight space, as shown in equation \eqref{eq:6}. 

More specifically, we have that $Id \otimes {\tcoev}_{V_N^{\otimes {n-1}}}(v_0)$ belongs to the weight space of weight $ \color{red}(n-1)(N-1)$ inside $V_N^{\otimes n}\otimes (V_N^{*})^{\otimes n-1}$. Since braid group actions preserve weight spaces, we can see the whole invariant through this particular weight space, as below.
\begin{equation}
\begin{aligned}
J_N(L,q)=& q^{-(N-1)w(\beta_n)} \ \pi \circ
\left( Id\otimes \ev_{V_N^{\otimes {n-1}}} \right) 
\circ  \\
& \circ \varphi^{q,N-1}_{2n-1,(n-1)(N-1)} \left(\beta_{n} \cup \bar{\mathbb I}_{n-1} \right) \circ
\left(\Id \otimes {\tcoev}_{V_N^{\otimes {n-1}}}\right) (v_0).
\end{aligned}
\end{equation}
Using the lift of evaluations and coevaluations from definition \ref{D:9} together with remark \ref{R:5}, we obtain the following formula:
\begin{equation}
\begin{aligned}
J_N(L,q)=& q^{-(N-1)w(\beta_n)} \ \pi \circ
\left( Id\otimes \ev_{N}^{\otimes {n-1}}|_{\eta_{q,N-1}} \right) 
\circ  \\
& \circ \varphi^{q,N-1}_{2n-1,(n-1)(N-1)} \left(\beta_{n} \cup \bar{\mathbb I}_{n-1} \right) \circ
\left(\Id \otimes {\tcoev}_{N}^{\otimes {n-1}}|_{\eta_{q,N-1}}\right) (v_0).
\end{aligned}
\end{equation}

Now we will use the morphism $f$ of the form presented in equation \eqref{eq:7}. Then, proposition \ref{P:4} tells us that if we add the extra morphisms corresponding to this twisting function, they do not change the definition of the invariant.
\begin{equation}
\begin{aligned}
\hspace{-6mm}J_N(L,q)= & q^{-(N-1)w(\beta_n)} \ \pi \circ \left( Id\otimes \ev_{N}^{\otimes {n-1}} \right) {\color{blue}\left( Id \otimes f^{-1}|_{\eta_{q,N-1}}\right) } \circ \\
& \hspace{20mm}\circ \varphi^{q,N-1}_{2n-1,(n-1)(N-1)} \left(\beta_{n} \cup {\mathbb I}_{n-1} \right) \circ \\
& \hspace{20mm}\circ {\color{blue}\left( Id\otimes f|_{\eta_{q,N-1}} \right) } \circ  \left(\Id \otimes {\tcoev}_{N}^{\otimes {n-1}}\right) (v_0).
\end{aligned}
\end{equation}
 In the following expression, we use the normalised evaluations and coevaluations from definition \ref{D:5}, which lead to the following formula:
\begin{equation}\label{eq:12}
\begin{aligned}
J_N(L,q)=&q^{-(N-1)w(\beta_n)}\pi \circ \left(Id\otimes \ev^{\otimes {n-1}}_{f}{|_{\eta_{q,N-1}} } \right) \\
& \circ \varphi^{q,N-1}_{2n-1,(n-1)(N-1)} \left(\beta_{n} \cup {\mathbb I}_{n-1} \right) \circ \left( \Id \otimes {\tcoev}_f^{\otimes {n-1}}{|_{\eta_{q,N-1}}} \right) (v_0).
\end{aligned}
\end{equation}
%Since braid group actions preserve the weight spaces, it follows that we can see the whole construction of the invariant through this particular weight space, as below: 
%\begin{equation}
%\begin{aligned}
%J_N(L,q)=&q^{-(N-1)w(\beta)}\pi \circ \left(Id\otimes \ev^{\otimes {n-1}}_{f}{|_{\eta_{q,N-1}} }\right) \\
%& \circ \varphi^{q,N-1}_{2n-1,(n-1)(N-1)} \left(\beta_{n} \otimes {\mathbb I}_{n-1} \right) \circ \left(\Id \otimes {\tcoev}_f^{\otimes {n-1}}{|_{\eta_{q,N-1}}}\right) (v_0).
%\end{aligned}
%\end{equation}
%where small action=restriction of the big action
\subsection{Step II-Using the weight spaces from the Verma module}\label{SCJ:2} Having in mind that there are homological correspondents for the weight spaces in the Verma module, we aim to use these ones instead of the ``small'' weight spaces. More precisely, let us consider the embedding of the weight spaces into the generic ones as below: 
$$\iota:V^{q,N-1}_{2n-1,(n-1)(N-1)}\hookrightarrow \hat{V}^{q,N-1}_{2n-1,(n-1)(N-1)}.$$
\begin{remark}\label{R:3}   
Based on the properties presented in subsection \ref{S:quan}, we remark that the quantum representation specialised by $\psi_{q,N-1}$ preserves small weight spaces inside the weight spaces from generic Verma module. More precisely, we have the following commutative diagram:
\begin{center}
\begin{tikzpicture}
[x=1.2mm,y=1.4mm]

% Nodes of the diagram
\node (b1)  [color=blue]             at (0,10)    {$V^{q,N-1}_{2n-1,(n-1)(N-1)}$};
\node (t1) [color=black] at (30,10)   {$\hat{V}^{q,N-1}_{2n-1,(n-1)(N-1)}$};
\node (b2) [color=blue] at (0,0)  {$\varphi^{q,N-1}_{2n-1,(n-1)(N-1)}$};
\node (t2)  [color=black]             at (30,0)    {$\hat{\varphi}^{q,N-1}_{2n-1,(n-1)(N-1)}$};
\node (d2) [color=black] at (15,10)   {$\hookrightarrow$};
\node (d2) [color=black] at (15,12)   {$\iota$};
\node (d2) [color=black] at (0,5)   {$\circlearrowleft$};
\node (d2) [color=black] at (30,5)   {$\circlearrowleft$};
\node (d2) [color=black] at (15,5)   {$\equiv$};

%\draw[->,color=black]             (b2)      to node[left,xshift=6mm,yshift=3mm, font=\small]{$ \iota|_{\psi_{\xi,\lambda}} $}                           (t2);
%\draw[->,color=blue] (b1)      to node[left,xshift=6mm,yshift=3mm, font=\small]{$ \iota|_{\psi_{\xi,\lambda}} $}                           (t1);
%\draw[->,color=blue]   (b1)      to node[right,font=\small]{$$}                           (b2);
%\draw[->,color=blue]  (t1)      to node[right,font=\small]{$$}                           (t2);
\
% Middles of squares
\end{tikzpicture}
\end{center}
\end{remark}
Further on, if we use the $f$-coevaluation and compose it with this inclusion, it leads to the same thing as the extended coevaluation (which is defined in definition \ref{D:5'}) and so we have that:
\begin{equation}
\iota \circ \left(\Id \otimes {\tcoev}_f^{\otimes {n-1}}{|_{\eta_{q,N-1}}}\right)=\tilde{{\tcoev}}_f^{\otimes {n-1}}{|_{\eta_{q,N-1}}}.
\end{equation}
Moreover, thanks to remark \ref{R:3}, the whole construction corresponding to the first two floors (the cups and the braid part) arrives anyway in the small weight space. So, when we close up the formula by the evaluation (corresponding to the caps), we could use the extended evaluation instead. 

Combining all these remarks we conclude that we can describe the invariant through the weight spaces from the Verma module as below:
\begin{equation}\label{eq:13}  
\begin{aligned}
J_N(L,q)=&q^{-(N-1)w(\beta_n)}\pi \circ \left(Id\otimes \tilde{\ev}^{\otimes {n-1}}_{f}{|_{\eta_{q,N-1}} }\right) \circ \\
& \circ \hat{\varphi}^{q,N-1}_{2n-1,(n-1)(N-1)} \left(\beta_{n} \cup {\mathbb I}_{n-1} \right) \circ \left(\Id \otimes \tilde{{\tcoev}}_f^{\otimes {n-1}}{|_{\eta_{q,N-1}}}\right) (v_0).
\end{aligned}
\end{equation}

%\begin{equation}
%\begin{aligned}
%J_N(L,q)=&q^{-(N-1)w(\beta)}\pi \circ (Id\otimes \ev^{\otimes {n-1}}_{f}{|_{\eta_{q,N-1}} }) \\
%& \circ \hat{\varphi}^{q,N-1}_{2n-1,(n-1)(N-1)} (\beta_{n} \otimes {\mathbb I}_{n-1} ) \circ \iota \circ (\Id \otimes {\tcoev}_f^{\otimes {n-1}}{|_{\eta_{q,N-1}}}) (v_0)=\\
%=&q^{-(N-1)w(\beta)}\pi \circ {\color{green} (Id\otimes \ev^{\otimes {n-1}}_{f}{|_{\eta_{q,N-1}} })} \\
%& \circ {\color{blue} \Theta^{-1}} \circ \hat{\varphi}^{q,N-1}_{2n-1,(n-1)(N-1)} (\beta_{n} \otimes {\mathbb I}_{n-1} ) \circ  {\color{blue} \Theta}{ \color{red}\circ \iota \circ (\Id \otimes {\tcoev}_f^{\otimes {n-1}}{|_{\eta_{q,N-1}}}) (v_0)}.
%\end{aligned}
%\end{equation}
\subsection{Step III-Twisting the evaluations and coevaluations corresponding to the geometric part}\label{SCJ:3}

\

From now on we pass towards the homological part. We will do this using the construction of the homology classes presented in section \ref{5}, specialised by $\psi_{q,N-1}$.

We start with the formulas from equation \eqref{eq:7} and definition \ref{D:5'} which lead to the following:
\begin{equation} 
\begin{aligned}
&\hspace{-10mm} \left(Id\otimes \tilde{{\tcoev}}_f^{\otimes {n-1}}\right)(v_0)=\\
=& \sum_{i_1,...,i_{n-1}=0}^{N-1}  \ v_0 \otimes v_{i_1}\otimes ... \otimes v_{i_{n-1}} \otimes v_{N-1-i_{n-1}}\otimes ... \otimes v_{N-1-i_1}.
\end{aligned}
\end{equation}
This gives the following formula for the invariant:
\begin{equation}\label{eq:8}
\begin{aligned}
& \hspace{-9mm}J_N(L,q)=q^{-(N-1)w(\beta_n)}  \ \pi \circ {\color{green} (Id\otimes \ev^{\otimes {n-1}}_{f}{|_{\eta_{q,N-1}} })} \ \circ \\
& \hspace{4mm} \circ \hat{\varphi}^{q,N-1}_{2n-1,(n-1)(N-1)} (\beta_{n} \cup {\mathbb I}_{n-1} )\\
 & \hspace{5mm}\left(  \sum_{i_1,...,i_{n-1}=0}^{N-1}  \ v_0 \otimes v_{i_1}\otimes ... \otimes v_{i_{n-1}} \otimes v_{N-1-i_{n-1}}\otimes ... \otimes v_{N-1-i_1} \right).
\end{aligned}
\end{equation}

In the next part we aim pass from the quantum side from above towards the topological one. Looking at the identification from Theorem \ref{T1} we have:
%\begin{equation} \label{L:4}  
%\begin{aligned}
%\Theta(v_{e_1}\otimes ... \otimes v_{e_{2n-1}})& =\mathscr F_{e}, \ \ \ \forall \  e=(e_1,...,e_{n-1}) \in E_{2n-1,m}.
%\end{aligned}
%\end{equation}

\begin{equation}
\begin{aligned}
\hat{\varphi}_{2n-1,(n-1)(N-1)} \hspace{20mm} l_{2n-1,(n-1)(N-1)} \ \ \ \ \ \ \ \ &\\
(B_n \cup \mathbb I_{n-1}) \curvearrowright \hspace{25mm} \curvearrowleft (B_n \cup \mathbb I_{n-1}) \ \ \ \ \ \ \ \ &\\
 \hat{V}_{2n-1,(n-1)(N-1)}  \ \simeq \ H^{\text{lf},-}_{(n-1)(N-1)}(\tilde{C}_{2n-1,(n-1)(N-1)}, \Z)|{_{\gamma}}&\\
\hspace{-25mm} \Theta\left(v_0 \otimes v_{i_1}\otimes ... \otimes v_{i_{n-1}} \otimes v_{N-1-i_{n-1}}\otimes ... \otimes v_{N-1-i_1}\right) =\hspace{1mm}& \\
&\hspace{-40mm}= \mathscr F_{(0,i_1,...,i_{n-1}, N-1-i_{n-1},...,N-1-i_{1})}.
\end{aligned}
\end{equation}

This shows that the coevaluation corresponds to a sum of multiarcs. However, as we have discussed in section \ref{S:1}, we are interested in using code sequences instead. Concerning this question,  proposition \ref{E:1} tells us that the difference occurs just in certain coefficients:
$$\mathscr F_e=\frac{1}{\prod_{i=1}^{n} (e_i)_{d}!} \cdot \ x^{\frac{1}{2}\sum_{i=1}^{n}(i-1) e_i} \cdot \U_{e}, \forall e \in E_{n,m}.
$$
This shows that:
\begin{equation}
\begin{aligned}
&\mathscr F_{(0,i_1,...,i_{n-1}, N-1-i_{n-1},...,N-1-i_{1})}= \frac{1}{\prod_{k=1}^{n-1} (i_k)_{d}!(N-1-i_k)_{d}!}  \cdot \\
& \hspace{10mm} \cdot x^{\frac{1}{2}\sum_{k=1}^{n-1}\left(k i_k+(2n-1-k)(N-1-i_k)\right)} \cdot  \U_{(0,i_1,...,i_{n-1}, N-1-i_{n-1},...,N-1-i_{1})}.
\end{aligned}
\end{equation}
On the other hand, in subsection \ref{S:1} we worked on the geometrical side in order to understand a good pairing between the two types of homologies and we have chosen very particular classes for which the pairing was easy to compute. More precisely, we have constructed in definition \ref{D:10} $\tilde{\U}_{(0,i_1,...,i_{n-1}, N-1-i_{n-1},...,N-1-i_{1})}$ to be another lift of the same geometric submanifold $\bar{U}_{(0,i_1,...,i_{n-1}, N-1-i_{n-1},...,N-1-i_{1})}$ from $C_{2n-1,(n-1)(N-1)}$. This means that $$\tilde{\U}_{(0,i_1,...,i_{n-1}, N-1-i_{n-1},...,N-1-i_{1})} \ \text {  and  } \ \U_{(0,i_1,...,i_{n-1}, N-1-i_{n-1},...,N-1-i_{1})}$$ differ by an element of the deck transformations 
(up to the sign $(-1)^{n-1}$ given by orientations, as in remark \ref{sign})
\begin{remark}
For the next part, we look at the classes 
$$\tilde{\U}_{(0,i_1,...,i_{n-1}, N-1-i_{n-1},...,N-1-i_{1})} \ \text {  and  } \ \U_{(0,i_1,...,i_{n-1}, N-1-i_{n-1},...,N-1-i_{1})}$$ 
in the homology $H^{\text{lf},-}_{(n-1)(N-1)}(\tilde{C}_{2n-1,(n-1)(N-1)}, \Z)$ and later we will look at them in our version of the homology $H_{2n-1,(n-1)(N-1)}$.
\end{remark}
More concretely, we have:
\begin{equation}
\begin{aligned}
\U_{(0,i_1,...,i_{n-1}, N-1-i_{n-1},...,N-1-i_{1})}& = (-1)^{n-1}x^{\alpha(i_1,...,i_{n-1})}d^{\beta(i_1,...,i_{n-1})} \cdot \\&\cdot \tilde{\U}_{(0,i_1,...,i_{n-1}, N-1-i_{n-1},...,N-1-i_{1})}
\end{aligned}
\end{equation}
for some integer numbers $\alpha(i_1,...,i_{n-1}),\beta(i_1,...,i_{n-1})$. Putting these together we conclude the following relation:
%\begin{equation}
%\begin{aligned}
%&\mathscr F_{(0,i_1,...,i_{n-1}, N-1-i_{n-1},...,N-1-i_{1})}=&\\&=\frac{x^{\sum_{k=1}^{n-1}\left(k i_k+(N-1-k)i_{N-1-k}\right)}}{\prod_{k=1}^{n-1} (i_k)_{d}!(N-1-i_k)_{d}!}
%\cdot x^{\alpha(i_1,...,i_{n-1})}d^{\beta(\alpha(i_1,...,i_{n-1}))}\\ & \ \ \ \ \ \ \ \ \ \ \ \ \ \ \ \ \ \ \ \ \ \ \ \ \ \ \ \ \ \ \ \ \ \ \ \ \ \cdot \tilde{\U}_{(0,i_1,...,i_{n-1}, N-1-i_{n-1},...,N-1-i_{1})}.
%\end{aligned}
%\end{equation}
\begin{equation}
\begin{aligned}
\tilde{\U}_{(0,i_1,...,i_{n-1}, N-1-i_{n-1},...,N-1-i_{1})}= \ &{x^{- \frac{1}{2}\sum_{k=1}^{n-1}\left(k i_k+(2n-1-k){(N-1-i_k)}\right)}}\cdot\\
\cdot  \prod_{k=1}^{n-1} (i_k)_{d}!(N-1-i_k)_{d}! & \cdot  (-1)^{n-1}   x^{-\alpha(i_1,...,i_{n-1})}d^{-\beta(i_1,...,i_{n-1})}\\ 
& \cdot \mathscr F_{(0,i_1,...,i_{n-1}, N-1-i_{n-1},...,N-1-i_{1})}.
\end{aligned}
\end{equation}
In the following part, we want to define the coefficients of the twisting function $g$ that we will use for the identification. We start with the following definition.
\begin{definition}(Choice of coefficients)

For a fixed set of indices $i_1,...,i_{n-1} \in \{0,...,N-1\}$, let us define the coefficient $c(i_1,...,i_{n-1}) \in \Z[q^{\pm 1}, s^{\pm 1}]$ which counts exactly the difference between $ \mathscr F_{(0,i_1,...,i_{n-1}, N-1-i_{n-1},...,N-1-i_{1})}$ and $\tilde{\U}_{(0,i_1,...,i_{n-1}, N-1-i_{n-1},...,N-1-i_{1})}$ (computed above), specialised by the function $\gamma$:
\begin{equation}
\begin{aligned} 
c(i_1,...,i_{n-1})& =
 (-1)^{n-1} {s^{-\sum_{k=1}^{n-1}\left(k i_k+(2n-1-k)(N-1-i_k)\right)}} \cdot \\
 & \cdot  s^{-2\alpha(i_1,...,i_{n-1})}q^{2\beta(i_1,...,i_{n-1})} \cdot  {\prod_{k=1}^{n-1} (i_k)_{q^{-2}}!(N-1-i_k)_{q^{-2}}! }\cdot\\ 
\end{aligned}
\end{equation}
\end{definition}
With this particular choice of coefficients we have:
\begin{equation}\label{eq:9}   
\begin{aligned}
\tilde{\mathscr U}_{(0,i_1,...,i_{n-1}, N-1-i_{n-1},...,N-1-i_{1})}& =c(i_1,...,i_{n-1}) \cdot \\
\cdot & \mathscr F_{(0,i_1,...,i_{n-1}, N-1-i_{n-1},...,N-1-i_{1})}\\
& \forall \  0\leq i_1,...,i_{n-1}\leq N-1.
\end{aligned}
\end{equation}
Now we want to use the identification between homological and quantum representation as in Lemma \ref{L:1}. We aim to correlate the basis given by monomials from the weight spaces with the corresponding code sequences. In order to do that, we use a function $g$ whose twisting coefficients correspond exactly to the change between the code sequence and the normalised multiarcs. More precisely, we choose the coefficients as follows.
\begin{definition} (Twisting of the quantum basis via the function $g$)\label{D:6} \\
Let us define the function $g$ as in \ref{L:1}, corresponding to the following coefficients:
\begin{equation}
\begin{aligned}
& y(e_{n+1},...,e_{2n-1})=\\
&=\begin{cases}
c(N-1-e_{2n-1},...,N-1-e_{n+1}), \  if \  0\leq e_{n+1},...,e_{2n-1}\leq N-1\\
 \ \ \ \ \ \ \ 1,  \ \ \ \ \ \ \ \ \ \ \ \ \ \ \ \ \ otherwise.
\end{cases}
\end{aligned}
\end{equation}
\end{definition}
\begin{remark}\label{R:4} 
Using this twisting function $g$, we have the following correspondence over $\tilde{\Li}_N$, with respect to the braid group action $B_n \cup \mathbb I_{n-1}$, for any indices $i_1,...,i_{n-1} \in \{0,...,N-1\}$:
\begin{equation}
\begin{aligned}
 c(i_1,...,i_{n-1})^{-1} \ v_0 \otimes v_{i_1}\otimes ... \otimes v_{i_{n-1}} \otimes v_{N-1-i_{n-1}}\otimes ... \otimes v_{N-1-i_1}\\ \longleftrightarrow^{\color{blue} \Theta \circ g}  \ \ \ \ \ \ \ \ \ \ \ \ \F_{(0,i_1,...,i_{n-1}, N-1-i_{n-1},...,N-1-i_{1})}. 
 \end{aligned}
 \end{equation}
Using equation \eqref{eq:9}, this is equivalent to:
\begin{equation}\label{eq:10''}
\begin{aligned}
 v_0 \otimes v_{i_1}\otimes ... \otimes v_{i_{n-1}} \otimes v_{N-1-i_{n-1}}\otimes ... \otimes v_{N-1-i_1}\\ \longleftrightarrow^{\color{blue} \Theta \circ g}  \ \ \ \ \ \ \ \ \ \ \ \ \tilde{\U}_{(0,i_1,...,i_{n-1}, N-1-i_{n-1},...,N-1-i_{1})}. 
 \end{aligned}
 \end{equation}
\end{remark}
\begin{notation}
Let us consider the set of ``symmetric idices'' as follows:
\begin{equation}
\begin{aligned}
E^{N,symm}=\{e=(0,e_1,...,e_{2n-2})\in E_{2n-1,(n-1)(N-1)}\mid & e_i=N-1-e_{2n-1-i}\\
& \hspace{7mm}\forall 1 \leq i \leq n-1 \}.
\end{aligned}
\end{equation}

\end{notation}
\begin{corollary}\label{C:2}
From the identification from equation \eqref{eq:10''}, which holds over $\tilde{\Li}_N$, we conclude the following correspondence over $\Z[s^{\pm1}, q^{\pm 1}]$: 
\begin{equation}\label{eq:10}
\begin{aligned}
\hat{\varphi}_{2n-1,(n-1)(N-1)} \hspace{20mm} l_{2n-1,(n-1)(N-1)}|_{\gamma} \ \ \ \ \ \ \ \ \ \ \  \ \ &\\
(B_n \cup \mathbb I_{n-1}) \curvearrowright \hspace{15mm} \curvearrowleft (B_n \cup \mathbb I_{n-1}) \hspace{29mm} &\\
  \ \ \ \ \ \ \ \ \ \ \ \hat{V}_{2n-1,(n-1)(N-1)}  \ \simeq \ H^{\text{lf},-}_{(n-1)(N-1)}(\tilde{C}_{2n-1,(n-1)(N-1)}, \Z)|_{\gamma} &\\
 {\color{blue} \Theta' \circ g} \left(v_0 \otimes v_{i_1}\otimes ... \otimes v_{i_{2n-2}} \right) =\hspace{30mm}& \\
& \hspace{-95 mm }=\begin{cases}
\tilde{\U}_{(0,i_1,...,i_{n-1}, N-1-i_{n-1},...,N-1-i_{1})}, \text{ if } (0,i_1,...,i_{2n-2}) \in E^{N,symm}\\
\F_{(0,i_1,...,i_{2n-2})}, \ \ \ \ \ \ \ \ \ \ \ \ \text{otherwise}.
\end{cases}
\end{aligned}
\end{equation}
\end{corollary}
\begin{proof}
Following the identification from equation \eqref{eq:10''}, we have the correspondence over $\tilde{\Li}_N$:
\begin{equation}
\begin{aligned}
(B_n \cup \mathbb I_{n-1}) \curvearrowright \hspace{15mm} \curvearrowleft (B_n \cup \mathbb I_{n-1}) \hspace{29mm}&\\
  \ \ \ \ \ \ \ \ \ \ \ \hat{V}_{2n-1,(n-1)(N-1)}|_{\iota_N}  \ \simeq \ H_{2n-1,(n-1)(N-1)}|{_{\tilde{\gamma}_N}} \ \ \ \ \ \ \ \ \ \ \ \ \ \ \ \ \ \ \ &\\
 {\color{blue} \Theta' \circ g} \left(v_0 \otimes v_{i_1}\otimes ... \otimes v_{i_{2n-2}} \right) =\hspace{45mm}& \\
& \hspace{-96 mm }=\begin{cases}
\tilde{\U}_{(0,i_1,...,i_{n-1}, N-1-i_{n-1},...,N-1-i_{1})}, \text{ if } (0,i_1,...,i_{2n-1}) \in E^{N,symm}\\
\F_{(0,i_1,...,i_{2n-2})}, \ \ \ \ \ \ \ \ \ \ \ \ \text{otherwise}.
\end{cases}
\end{aligned}
\end{equation}
However, we know that the braid group action on the quantum side $$\hat{\varphi}_{2n-1,(n-1)(N-1)}(v_0 \otimes v_{i_1}\otimes ... \otimes v_{i_{2n-2}})$$ has all the coefficients in $\Z[s^{\pm 1}, q^{\pm 1}]$. The same holds for the homological side as well. This shows that the two representations, which we know that are equal (in the corresponding bases) over $\tilde{\Li}_N$, actually have all the coefficients belonging to $\Z[s^{\pm1},q^{\pm1}]$. But the inclusion:
$$\Z[s^{\pm1},q^{\pm1}] \subseteq {\tilde{\Li}_N=\mathbb Z[s^{\pm1},q^{\pm1}]\left(I_N \right)^{-1}}$$ is injective. Thus, overall, the identification is true over $\Z[s^{\pm 1}, q^{\pm 1}]$.
\end{proof}
%with the inverse of the coefficient between the code sequence and the %normalised multiarc from formula \ref{E:1}, on the vectors from the small weight space $V^N_{2n-1,(n-1)(N-1)}\hookrightarrow \hat{V}^N_{2n-1,(n-1)(N-1)}$ of the form:
%$$v_0 \otimes v_{i_1} \otimes ... \otimes v_{i_{n-2}}\otimes v_{N-1-i_{1}}\otimes... \otimes v_{N-1-i_{n-2}}. $$
%Since on the top level, the evaluation sees anyway just this kind of vectors, using the Lemma \ref{L:1} we can express the invariant as:
\subsection{Step IV-Homological correspondent of the coevaluation}
\label{SCJ:4}
\

In this step we combine these two ideas and use the twisted evaluations together with the new correspondence in order to reach the geometric part that is convenient for us, given by the code sequences. More precisely, using the formula from equation \eqref{eq:8} and the identification from Corollary \ref{C:2} we have:
%We will compose the braid group action with the isomorphism from above which uses the function $g$ such that we arrive at the basis of code sequences. We would like to emphasise here, that there is a subtle point related to the coefficients that occur in the function $g$, which contain non-trivial denominators. This means that for the identification $\Theta' \circ g$ we should work over $\tilde{\Li}_N$ (from \ref{fig:7}).%However, we are interested in the action of the braid just on elements corresponding to $$\tilde{\U}_{(0,i_1,...,i_{n-1}, N-1-i_{n-1},...,N-1-i_{1})}.$$ Using the formulas from the algebraic side together with the identification between the quantum side and the homological side, we conclude that after the braid group action on such a code sequence we will arrive in a linear combination of classes given by $$\tilde{\U}_{(j_0,j_1,...,j_{n-1}, N-1-i_{n-1},...,N-1-i_{1})}.$$ This shows that we do not have denominators and so we can work over the ring $\Li$ rather that $\tilde{\Li}_N$ .
\begin{equation}\label{eq:14}
\begin{aligned}
J_N(L,q) &  =q^{-(N-1)w(\beta_n)}\pi \circ {\color{green} \left(Id\otimes \tilde{\ev}^{\otimes {n-1}}_{f}{|_{\eta_{q,N-1}} }\right)}\circ \\
& \circ {\color{blue} ( \Theta' \circ g)^{-1}} \circ l_{2n-1,(n-1)(N-1)}|_{\psi_{q,N-1}} (\beta_{n} \cup {\mathbb I}_{n-1} ) \circ  {\color{blue} ( \Theta' \circ g )}\\
& \left( \sum_{i_1,...,i_{n-1}=0}^{N-1}v_0 \otimes v_{i_1}\otimes ... \otimes v_{i_{n-1}} \otimes v_{N-1-i_{n-1}}\otimes ... \otimes v_{N-1-i_1} \right)=\\
=^{Eq \ \eqref{eq:10}}&q^{-(N-1)w(\beta_n)}\pi \circ {\color{green} \left(Id\otimes \tilde{\ev}^{\otimes {n-1}}_{f}{|_{\eta_{q,N-1}} }\right)} \circ {\color{blue}( \Theta' \circ g)^{-1}} \circ \\
& \circ  l_{2n-1,(n-1)(N-1)}|_{\psi_{q,N-1}} (\beta_{n} \cup {\mathbb I}_{n-1} ) \\
&  \ \ \ \ \ \ \ \ \ \ \ \ \ \ \ \left( \sum_{i_1,...,i_{n-1}=0}^{N-1} \tilde{\U}_{(0,i_1,...,i_{n-1}, N-1-i_{n-1},...,N-1-i_{1})} \right).
\end{aligned}
\end{equation}
Using the notation from definition \ref{D:7}, we conclude the following formula, where the coevaluation is transported to the homological side:
\begin{equation}\label{eq:11}
\begin{aligned}
J_N(L,q)=&q^{-(N-1)w(\beta_n)}\pi \circ {\color{green} \left(Id\otimes \tilde{\ev}^{\otimes {n-1}}_{f} {|_{\eta_{q,N-1}} }\right)}  \circ (\Theta' \circ g)^{-1} \\
&  \circ l_{2n-1,(n-1)(N-1)} |_{\psi_{q,N-1}}(\beta_{n} \cup {\mathbb I}_{n-1} ){ \color{red} \mathscr P_n^{N}}.
\end{aligned}
\end{equation}
(for now, we look at the class $\mathscr P_n^{N}$ in the homology $H^{\text{lf},-}_{(n-1)(N-1)}(\tilde{C}_{2n-1,(n-1)(N-1)}, \Z)$).
\subsection{Step V-Passing to our version of Lawrence representation}\label{SCJ:5}

\

In this part, we will focus on the upper part of the diagram, aiming to understand the homological correspondent of the caps. Using the formula that gives the evaluation, presented in remark \ref{R:9}, we have the following:
\begin{equation}\label{eq:15}
\begin{aligned}
& \hspace{0mm}\pi \circ \left( Id \otimes \tilde{\ev}^{\otimes {n-1}}_{f} {|_{\eta_{q,N-1}} }\right): \hat{V}^{q,N-1}_{2n-1,(n-1)(N-1)} \rightarrow \Z[q^{\pm 1}]\\
& \pi \circ \left( Id \otimes \tilde{\ev}^{\otimes {n-1}}_{f} {|_{\eta_{q,N-1}} }\right)(w)= \\
& =\begin{cases}
\eta_{q,N-1} \circ p_n^N(i_1,...,i_{n-1}), \ \ if  \ \ \ 0 \leq i_1,...,i_{n-1}\leq N-1 \ \ and \\
 \ \ \ \ \ \ \ \ \ \ \ \ \ \ \ w=  v_0 \otimes v_{i_1}\otimes ... \otimes v_{i_{n-1}} \otimes v_{N-1-i_{n-1}}\otimes ... \otimes v_{N-1-i_1}\\
0, \ \ \ \ \ \ otherwise.
\end{cases}
\end{aligned}
\end{equation}

%On the quantum side, we know that the whole coloured Jones invariant can be constructed through the weight space corresponding to the finite dimensional module (equation \ref{eq:12}).  In other words, from the expression presented in equation  \eqref{eq:11} that describes the invariant, we know that: $$(\Theta' \circ g)^{-1} \circ l_{2n-1,(n-1)(N-1)} |_{\psi_{q,N-1}}(\beta_{n} \cup {\mathbb I}_{n-1} ){ \color{red} \mathscr P_n^{N}} \in V^{q,N-1}_{2n-1,(n-1)(N-1)}.$$ This shows that $(\Theta' \circ g)^{-1}  \circ l_{2n-1,(n-1)(N-1)} |_{\psi_{q,N-1}}(\beta_{n} \cup {\mathbb I}_{n-1} ){ \color{red} \mathscr P_n^{N}}$ is a linear combination of monomials, which have the property that all their indices are all less or equal to $N-1$. Homologically, this shows that $$l_{2n-1,(n-1)(N-1)} |_{\psi_{q,N-1}}(\beta_{n} \cup {\mathbb I}_{n-1} ){ \color{red} \mathscr P_n^{N}}$$ will be a linear combination of code sequences corresponding to partitions where all components are less than $N-1$.

Going back to the geometrical picture, we remind correspondence  $\Theta' \circ g $ from Corollary \ref{C:2}, for partitions $e=(e_0,e_1,...,e_{2n-2})\in E_{2n-1,(n-1)(N-1)}$: 
\begin{equation}\label{R:7}
\begin{aligned} 
w=v_{e_0}\otimes v_{e_1} \otimes ... \otimes v_{e_{2n-2}} \leftrightarrow  \tilde{\mathscr U}_{e_0,e_1,...,e_{2n-2}} \text{ if } (e_0,e_1,...,e_{2n-2})\in E^{N,symm}\\
\F_{e_0,e_1,...,e_{2n-2}} \text{ if } (e_0,e_1,...,e_{2n-2})\notin E^{N,symm}
\end{aligned}
\end{equation}
On the quantum side, let us look at one of the terms from the formula given in equation \eqref{eq:8}:
\begin{equation} \label{eq:25}
\begin{aligned}
\hspace{1mm}\hat{\varphi}^{q,N-1}_{2n-1,(n-1)(N-1)} (\beta_{n} \cup {\mathbb I}_{n-1} )&\\
&\hspace{-32mm} \left( v_0 \otimes v_{i_1}\otimes ... \otimes v_{i_{n-1}} \otimes v_{N-1-i_{n-1}}\otimes ... \otimes v_{N-1-i_1}\right)=\\
&\hspace{-49mm}=\sum_{j_0,...,j_{n-1}=0}^{N-1} \alpha(j_0,...,j_{n-1}) v_{j_0} \otimes v_{j_1}\otimes ... \otimes v_{j_{n-1}} \otimes v_{N-1-i_{n-1}}\otimes ... \otimes v_{N-1-i_1},
\end{aligned}
\end{equation}
for some coefficients $\alpha(j_0,...,j_{n-1}) \in \Z[q^{\pm 1}, s^{\pm 1}]$.
\begin{remark}\label{R:6} 
However, the only vector that gets evaluated non-trivially by the above evaluation is the monomial which corresponds to a symmetric partition:
$$v_{0} \otimes v_{i_1}\otimes ... \otimes v_{i_{n-1}} \otimes v_{N-1-i_{n-1}}\otimes ... \otimes v_{N-1-i_1}.$$
\end{remark}
On the other hand, the element from \eqref{eq:25} corresponds through the identification $\Theta' \circ g$ to the following sum:
\begin{equation}
\begin{aligned}
&\hspace{-3mm}l_{2n-1,(n-1)(N-1)}(\beta_{n} \cup {\mathbb I}_{n-1} )(\tilde{\mathscr U}_{0,i_1...,i_{n-1},N-1-i_{n-1},...,N-1-i_1})=\\
&\hspace{5mm}=\alpha(0,i_1,...,i_{n-1}) \tilde{\mathscr U}_{0,i_1...,i_{n-1},N-1-i_{n-1},...,N-1-i_1}+\\
&\hspace{5mm}+\sum_{\substack {j_0,...,j_{n-1}=0\\ (j_0,...,j_{n-1})\neq (0,i_1,...,i_{n-1})}}^{N-1}  \alpha(j_0,...,j_{n-1}) \cdot 
\F_{j_0,...,j_{n-1},N-1-i_{n-1},...,N-1-i_1}.
\end{aligned}
\end{equation}
Now, we would like to pass from the Borel-Moore homology $ H^{\text{lf},-}_{m}(\tilde{C}_{n,m}, \Z)$ to the homology $H_{2n-1,(n-1)(N-1)}$ defined in definition \ref{T1'}. We use the following result.

\begin{proposition}(\cite{CrM})\label{P:9}
The map induced by the inclusion at the homological level:
$$I: H_{2n-1,(n-1)(N-1)}\rightarrow H^{\text{lf},-}_{(n-1)(N-1)}(\tilde{C}_{2n-1,(n-1)(N-1)}, \Z)$$ is injective. Moreover braid group action $L_{n,m}$ on $H_{n,m}$ in the basis of code sequences is isomorphic to the action coming from $l_{n,m}$ on the free subspace generated by code sequences in $H^{\text{lf},-}_{m}(\tilde{C}_{n,m}, \Z)$. 
\end{proposition}
Combining the correspondence from \eqref{R:7} with remark \ref{R:6}, we obtain that the only element which should be evaluated on the homological side is $$\tilde{\mathscr U}_{0,i_1...,i_{n-1},N-1-i_{n-1},...,N-1-i_1}.$$ This means that from the braid group action $$l_{2n-1,(n-1)(N-1)}(\beta_{n} \cup {\mathbb I}_{n-1} )(\tilde{\mathscr U}_{0,i_1...,i_{n-1},N-1-i_{n-1},...,N-1-i_1})$$ we need the coefficient $\alpha(0,i_1,...,i_{n-1})$. Now, based on proposition \ref{P:9}, we see that this coefficient is the same as the coefficient of the corresponding class $\tilde{\mathscr U}_{0,i_1...,i_{n-1},N-1-i_{n-1},...,N-1-i_1}$ in the expression written using the braid action $L_{2n-1,(n-1)(N-1)}(\beta_{n} \cup {\mathbb I}_{n-1})$.

We conclude that we can use the homological action $L_{2n-1,(n-1)(N-1)}$ instead of $l_{2n-1,(n-1)(N-1)}$. This leads to the formula:
\begin{equation}\label{eq:11''}
\begin{aligned}
J_N(L,q)=&q^{-(N-1)w(\beta_n)}\pi \circ {\color{green} \left(Id\otimes \tilde{\ev}^{\otimes {n-1}}_{f} {|_{\eta_{q,N-1}} }\right)}  \circ (\Theta' \circ g)^{-1} \\
&  \circ L_{2n-1,(n-1)(N-1)} |_{\psi_{q,N-1}}(\beta_{n} \cup {\mathbb I}_{n-1} ){ \color{red} \mathscr P_n^{N}}.
\end{aligned}
\end{equation}
(now, the class $\mathscr P_n^{N}$ is seen in the homology $H_{2n-1,(n-1)(N-1)}$).

\subsection{Step VI- Homology classes given by figure eights are a natural choice, predicted from the quantum side}\label{SCJ:6}

\

In the following, we will show that the homology classes given by figure eights are a natural choice, predicted from the quantum side. More precisely, we want to find a geometric correspondent of the algebraic evaluation. 

From the previous remarks, we conclude that we need a submanifold which intersects $\tilde{ \mathscr U }_{e}$ non-empty if and only if $e$ is a partition which satisfies the following condition:
\begin{equation} 
\begin{aligned}
(*)\begin{cases}
\ \forall i \in \{1,...,n-1\} \ e_i=e_{2n-1-i}, \ e_0=0, \  \text { and }\\
 0\leq e_i \leq N-1,  \forall i\in \{1,...,n-1\}.
\end{cases}
\end{aligned}
\end{equation}
Actually, these are exactly the partitions that occur in the homological correspondent of the evaluation, given by $\mathscr P_{n}^N$. This specific requirement will motivate and explain why the {\em submanifolds given by figure eights are a natural choice.} 
 
This observation suggests that the dual manifold should be given by a geometric support which is {\em symmetric up to the reflection of the punctured disc with respect to its middle axis} (where we forget the first puncture), in the sense that the corresponding symmetric components add up to $N-1$. Now if we draw a figure eight between each such symmetric punctures $(k,2n-1-k)$ and take configuration spaces of $N-1$ points on such figures (as in diagram \ref{fig9}), we see that this has exactly the property that we need. 
 
 We conclude that these figure eights in the punctured disc lead to the good building blocks for our construction and taking their product, we obtain the homology class $\mathscr G_{n}^N$ as discussed in definition \ref{shc}.

\subsection{Step VII-Proof of the intersection formula}\label{SCJ:6}
\begin{figure}[H]
\centering
$ \color{red} \mathscr F_n^{N} \ \ \ =  \ \  \ \ \ \ \  s^{(n-1)} q^{-2\sum_{k=1}^{n-1}i_k} \ \ \ \ \ \ \ \ \ \ \ \ \ \ \ \ \ \ \ \ \ \ \ \ \ \ \ {\color{green} \mathscr G_{n}^N} \ \ \ \ \ \ \ \ \ $\\
${ \color{red}  \nwarrow \textit{  deformation} } $\\
${\color{red} \Huge \mathscr P_n^N= \sum_{i_1,...,i_{n-1}=0}^{N-1} \ \  \ \ \tilde{\mathscr U}_{0,i_1,...,i_{n-1},N-1-i_{n-1},...,N-1-i_{1}}} \ \ \ \ \ \ \ \ \ \ \ \ \ \ \ \ \ \ \ \ \ \ \ \ \ 
$
\includegraphics[scale=0.5]{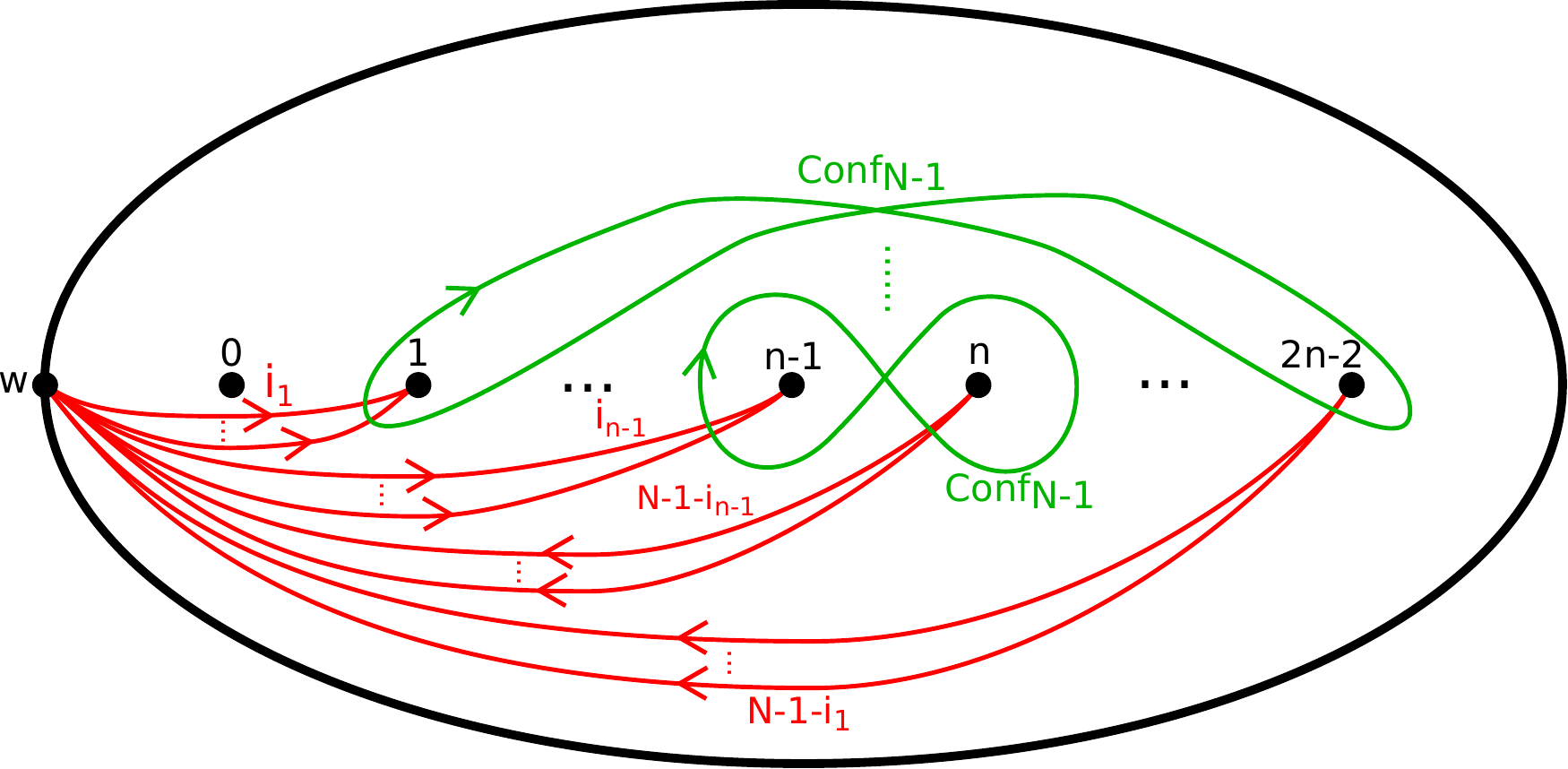}
\caption{}
\label{fig9}
\end{figure} In the next part, in order to make the notation easier, we will replace the action $L_{2n-1,(n-1)(N-1)} |_{\psi_{q,N-1}}(\beta_{n} \cup {\mathbb I}_{n-1} )$ directly by $(\beta_{n} \cup {\mathbb I}_{n-1} )$.
Putting everything together, we compute the model as follows: 
 \begin{equation}\label{eq:16}
\begin{aligned}
J_N(L,q)& =^{Eq \ \eqref{eq:11''}}q^{-(N-1)w(\beta_n)} \ \pi \circ {\color{green} \left(Id\otimes \tilde{\ev}^{\otimes {n-1}}_{f}{|_{\eta_{q,N-1}} }\right)} {\color{blue} (\Theta' \circ g)^{-1}} \\
&  \left( (\beta_{n} \cup {\mathbb I}_{n-1} ){ \color{red} \sum_{i_1,...,i_{n-1}=0}^{N-1} \tilde{\mathscr U}_{0,i_1,...,i_{n-1},N-1-i_{n-1},...,N-1-i_{1}}} \right).
\end{aligned}
\end{equation}
Now, we look at each term which occurs in this formula:
$$(\beta_{n} \cup {\mathbb I}_{n-1} ) \ { \color{red}  \tilde{\mathscr U}_{0,i_1,...,i_{n-1},N-1-i_{n-1},...,N-1-i_{1}}}.$$
We remark that this expression will be a linear combination of classes of the form:
\begin{equation}\label{eq:20}
\tilde{\mathscr U}_{{\color{teal}j_0,j_1,...,j_{n-1}},{\color {red}N-1-i_{n-1},...,N-1-i_{1}}}|_{\psi_{q,N-1}}, \ \text { for } {\color{teal} j_0,...,j_{n-1} \in \{0,...,N-1\}}.
\end{equation}

More precisely, we notice that the last $n-1$ indices of the classes that occur after the action of $\beta_n \cup \mathbb I_{n-1}$ are not changed. On the other hand, we know from from Corollary \ref{T:1'''} that this homological action (specialised through $\psi_{q,N-1}$) corresponds to the action onto weight spaces from the Verma module (specialised through $\eta_{q,N-1}$). However, we start with all indices less than $N-1$, corresponding to the ``small'' weight spaces, which we know that are preserved by the specialised braid group action. Going back on the homological side we conclude that we arrive with all indices less than $N-1$ as in equation \eqref{eq:20}.

Now, we want to apply the evaluation. From equation \eqref{eq:15}, we see that the only way in which a class is evaluated non-trivially by the identity union with the caps is if its indices are symmetric and the first one is zero, meaning the following form: 
\begin{equation}\label{eq:17} 
\begin{aligned}
\tilde{\mathscr U}_{0,j_1,...,j_{n-1},N-1-i_{n-1},...,N-1-i_{1}} \ \  \text{  where } & j_1=i_1,...,j_{n-1}=i_{n-1}. 
\end{aligned}
\end{equation}
Furthermore, the corresponding coefficient is given by the evaluation $\eta_{q,N-1}$ of the polynomial $p_n^N(i_1,...,i_{n-1})$, as we have discussed in remark \ref{eq:22}. 
Now, we look at the dual manifold. We notice the following.
\begin{lemma}
For indices $j_0,j_1,...,j_{2n-2} \in \{0,...,N-1\}$ such that $(j_0,j_1,...,j_{2n-2}) \in E_{2n-1,(n-1)(N-1)}$ we have the intersection formula:
\begin{equation}
\begin{aligned}
<\tilde{\mathscr U}_{j_0,j_1,...,j_{2n-2}}, \mathscr G_{n}^N>=\begin{cases}
1, \ \text {if} \  (j_0,...,j_{2n-2}) \in E^{N,symm}\\
0, \text{otherwise}.
\end{cases}
\end{aligned}
\end{equation}
\end{lemma}
\begin{proof}
We start by investigating when we can have a non-zero intersection pairing. In order to have a non-trivial intersection between lifts in the covering, we have to have at least an intersection point in the base configuration space. This will be encoded by $(n-1)(N-1)$ points in the disc, at the intersection between red and green segments. Since $j_0+j_1+...+j_{2n-2}=(n-1)(N-1)$ and the segments with multiplcities $j_0$ do not intersect the manifold $G_n^{N}$, it means that $j_0=0$. Further on, in order to have an intersection in the configuration space, each figure eight around the points $(k,2n-1-k)$ should support exactly $N-1$ points at the intersection with the red segments with multiplicities $(j_k,j_{2n-1-k})$. This shows that:
\begin{equation}\label{eq:26}
j_k+j_{2n-1-k}\geq N-1, \forall k\in \{1,...,n-1\}.
\end{equation}
However, we have that $j_1+...+j_{2n-2}=(n-1)(N-1)$. These two conditions imply that we have equalities in equation \eqref{eq:26} and so:
$$ j_k+j_{2n-1-k}=N-1, \forall k\in \{1,...,n-1\}.$$
This shows that if the intersection that we are interested in is non-zero, then the partition $(j_0,...,j_{2n-2}) \in E^{N,symm}$. 
 For the last part, we remind the computation of the intersection pairing from remark \ref{R:1}:
$$<\tilde{\mathscr U}_{0,i_1,...,i_{n-1},N-1-i_{n-1},...,N-1-i_{1}},  \mathscr G_{n}^N>=1, \ \ \ 
  \forall \  i_1,...,i_{n-1} \in \{0,...,N-1\}$$
which concludes the proof.
\end{proof}
Going back to the braid group action from \eqref{eq:20}, we have the following property:
\begin{equation}
\begin{aligned}
<\tilde{\mathscr U}_{j_0,j_1,...,j_{n-1},N-1-i_{n-1},...,N-1-i_{1}},  \mathscr G_{n}^N>&=\\
&\hspace{-20 mm}=\begin{cases}
1, \ \text {if} \  (j_0,...,j_{n-1})=(0,i_1,...,i_{n-1})\\
0, \text{otherwise}.
\end{cases}
\end{aligned}
\end{equation}
Using this, we continue the computation as below.

 \begin{equation}
\begin{aligned}
J_N(L,q)=& ^{Eq \eqref{eq:16}} q^{-(N-1)w(\beta_n)} \sum_{i_1,...,i_{n-1}=0}^{N-1} < { \color{green}  p_n^{N}(i_1,..,i_{n-1})}\\
&(\beta_{n} \cup {\mathbb I}_{n-1} )  { \color{red}  \tilde{\mathscr U}_{0,i_1,...,i_{n-1},N-1-i_{n-1},...,N-1-i_{1}}|_{\gamma}}, {\color{green} \mathscr G_n^N|_{\gamma}}> |_{\eta_{q,N-1}}.
\end{aligned}
\end{equation}
%Now, we use that the coefficient $p_n^{N}(i_1,..,i_{n-1})$ depends just on the last $n-1$ indices of the classes that appear in $$(\beta_{n} \cup {\mathbb I}_{n-1} )  { \color{red}  \tilde{ \mathscr U}_{0,i_1,...,i_{n-1},N-1-i_{n-1},...,N-1-i_{1}}},$$ which remain unchanged before and after the action of $ (\beta_{n} \cup {\mathbb I}_{n-1} ) $. This allow us to move it as the coefficient of the class ${ \color{red}  \tilde{ \mathscr U}_{0,i_1,...,i_{n-1},N-1-i_{n-1},...,N-1-i_{1}}}$ before we act with the braid and leads to the following:
Moving the coefficient on the other side of the braid group action and using proposition \ref{P:8}, we obtain the following formula :
\begin{equation}
\begin{aligned}
& J_N(L,q)  = q^{-(N-1)w(\beta_n)} \cdot \hspace{-3mm} \sum_{i_1,...,i_{n-1}=0}^{N-1} \hspace{-3mm} < \left(\beta_{n} \cup {\mathbb I}_{n-1} \right) { \color{green}  p_n^{N}(i_1,..,i_{n-1})} \\
& \hspace{40mm}{ \color{red}\tilde{\mathscr U}_{0,i_1,...,i_{n-1},N-1-i_{n-1},...,N-1-i_{1}}|_{\gamma}},{\color{green} \mathscr G_n^N|_{\gamma}}> |_{\eta_{q,N-1}}=\\
&\hspace{14mm} = q^{-(N-1)w(\beta_n)} \cdot  < \left(\beta_{n} \cup {\mathbb I}_{n-1} \right) \sum_{i_1,...,i_{n-1}=0}^{N-1} { \color{green}   p_n^{N}(i_1,..,i_{n-1})} \cdot \\
& \hspace{38mm} \cdot { \color{red}\tilde{\mathscr U}_{0,i_1,...,i_{n-1},N-1-i_{n-1},...,N-1-i_{1}}|_{\gamma}},{\color{green} \mathscr G_n^N|_{\gamma}}> |_{\eta_{q,N-1}}=\\
& \ \ \ \ \ \ \ \ \ \ \ =^{Prop \ref{P:8}} q^{-(N-1)w(\beta_n)} < (\beta_{n} \cup {\mathbb I}_{n-1} )  { \color{red}  \mathscr F_n^N}, {\color{green}\mathscr G_n^N|_{\gamma}}> |_{\eta_{q,N-1}}.
\end{aligned}
\end{equation}
Using the property of the homology classes introduced in definition \ref{E:3}:
$$\mathscr F_n^N=s^{(n-1)} \mathscr E_n^N|_{{\gamma}}$$ together with the relation between the specialisations from diagram \ref{fig:7} $$\psi_{q,N-1}=\eta_{q,N-1} \circ \gamma $$ we obtain the following formula: 
\begin{equation}
\ J_N(L,q) \ = \ q^{-(N-1)w(\beta_n)} \cdot \ q^{(N-1)(n-1)} \ \ <(\beta_{n} \cup {\mathbb I}_{n-1} ) \ { \color{red} \mathscr E_n^N}, {\color{green} \mathscr G_n^N}> |_{\psi_{q,N-1}}.
\end{equation}
 This concludes the topological model for the coloured Jones invariants.
\section{Topological intersection model for the coloured Alexander invariants}
\label{7}
In this section, we aim to prove the topological model for the coloured Alexander polynomials, which is described in Theorem \ref{THEOREM}.
Following equation \eqref{eq:A}, the coloured Alexander invariant can be expressed as:
\begin{equation}
 \begin{aligned}
\Phi_{N}(L,\lambda)& ={\xi_N}^{(N-1)\lambda w(\beta_n)} \ p \ \circ\\
& \circ \left( (Id\otimes \ev_{U^{\otimes {n-1}}_{\lambda}} ) \circ \varphi^{U_{\lambda}}_{2n-1} (\beta_{n} \cup \bar{\mathbb I}_{n-1} ) \circ (\Id \otimes {\tcoev}_{U^{\otimes {n-1}}_{\lambda}}) \right) (v_0).
\end{aligned}
\end{equation}

\subsection{Step I-The invariant through weight spaces}
We start with a discussion concerning the algebraic construction which is specific to the root of unity. For this case, as we have seen in section \ref{SS:1}, there is a subtletly which occurs from our choice of the version of the quantum group (given by divided powers of the generator $F$).

%However, we noticed in Lemma \ref{L:2} that if we look at the action of the $R$-matrix on the Verma module from \cite{Ito2}, its action on elementary vectors has the same formula as the one that we have described, specialised at a root of unity. So, at the level of the braid group action, we can use the version of the $R$-matrix, presented in . Now, the subtlety occurs at the level of the finite dimensional modules. 
The problem occurs because in this version $\hat{V}_{\xi_N,\lambda}$ does not gain an $N$-dimensional submodule as in the case of the Verma module over the quantum group with the usual generators. However, we have defined  $U_{\lambda}$ to be the vector subspace generated by the first $N$-vectors inside the Verma module. We remind Lemma \ref{L:2}, which shows that even if  $U_{\lambda}$ is not a submodule over the quantum group, its tensor power is preserved inside the tensor power of the specialised Verma module, with respect to the specialised braid group action. 

Pursuing this line, an analog argument as the one presented in the first step of the proof for the coloured Jones polynomials (from subsection \ref{SCJ:1}) tells us that we can see the coloured Alexander polynomial through the weight space $$ V^{\xi_N,\lambda}_{2n-1,(n-1)(N-1)}.$$ 
More precisely, we have the following formula.
 \begin{equation}
 \begin{aligned}
\Phi_{N}(L,\lambda) & ={\xi_N}^{(N-1)\lambda w(\beta_n)} \ p \  \circ \left( Id\otimes \ev_{U^{\otimes {n-1}}_{\lambda}} \right) \circ \\ 
& \hspace{5mm}\circ \varphi^{\xi_N,\lambda}_{2n-1,(n-1)(N-1)} \left(\beta_{n} \cup \bar{\mathbb I}_{n-1} \right) \circ
\left(\Id \otimes {\tcoev}_{U_{\lambda}^{\otimes {n-1}}}\right) (v_0).
\end{aligned}
\end{equation}
\subsection{Step II-Using the normalised dualities} Now, we want to see the whole invariant coming from a construction over two variables. We have seen in remark \ref{R:5} that the evaluations and coevaluations at roots of unity can be obtained from the evaluations and coevaluations over two variables, introduced in equation \eqref{eq:0}. 

We will twist these dualities with the function $f$ (defined in equation \eqref{eq:7}), in order to use the action through a particular weight space inside $U^{\otimes 2n-1}_{\lambda}$. Then, proposition \ref{P:4} leads us to the following description.
\begin{equation}
\begin{aligned}
\Phi_{N}(L,\lambda)& = {\xi_N}^{(N-1)\lambda w(\beta_n)} \ p \circ \left( Id\otimes \ev_{\xi_N}^{\otimes {n-1}}|_{\eta_{\xi_N,\lambda}} \right) {\color{blue}\left( Id\otimes f^{-1}|_{\eta_{\xi_N,\lambda}}\right) } \circ \\
& \circ \varphi^{\xi_N,\lambda}_{2n-1,(n-1)(N-1)} \left(\beta_{n} \cup {\mathbb I}_{n-1} \right) \circ {\color{blue}\left( Id\otimes f|_{\eta_{\xi_N,\lambda}} \right) } \circ \\
&  \circ \left(\Id \otimes {\tcoev}_{N}^{\otimes {n-1}}|_{\eta_{\xi_N,\lambda}}\right) (v_0).
\end{aligned}
\end{equation}
Using the notations for the twisted evaluations from definition \ref{D:5}, we conclude the formula:
\begin{equation}
\begin{aligned}
\Phi_{N}(L,\lambda)&={\xi_N}^{(N-1)\lambda w(\beta_n)} \ p \circ \left(Id\otimes \ev^{\otimes {n-1}}_{f,\xi_N}{|_{\eta_{\xi_N,\lambda}} } \right) \\
& \circ \varphi^{\xi_N,\lambda}_{2n-1,(n-1)(N-1)} \left(\beta_{n} \cup {\mathbb I}_{n-1} \right) \circ \left( \Id \otimes {\tcoev}_f^{\otimes {n-1}}{|_{\eta_{\xi_N,\lambda}}} \right) (v_0).
\end{aligned}
\end{equation}
This equation together with the discussion from subsection \ref{SS:2} and definition \ref{D:11}, shows that we can see the construction of the coloured Alexander polynomial coming from a formula over two variables and then specialised using the function $\eta_{\xi_N,\lambda}$.
%This shows that we obtain the coloured Alexander polynomials through the generic weight spaces, corresponding to the finite dimensional part, via the specialisation $\eta_{\xi_N,\lambda}$.
\subsection{Step III- The invariant through the bigger weight spaces}
Similar to the generic case, we aim to see the invariant through the weight spaces from the Verma module. We remind that the inclusion of the weight space corresponding to the finite dimensional part at root of unity into the corresponding weight spaces from the Verma module is preserved by the braid group action (following Lemma \ref{L:2} and remark \ref{R:8}):
%As in the previous proof, the inclusion of the weight spaces commutes with the $B_n$ action, as follows.
\begin{center}
\begin{tikzpicture}
[x=1.2mm,y=1.4mm]

% Nodes of the diagram
\node (b1)  [color=blue]             at (0,10)    {$V^{\xi_N,\lambda}_{2n-1,(n-1)(N-1)}$};
\node (t1) [color=black] at (30,10)   {$\hat{V}^{\xi_N,\lambda}_{2n-1,(n-1)(N-1)}$};
\node (b2) [color=blue] at (0,0)  {$\varphi^{\xi_N,\lambda}_{2n-1,(n-1)(N-1)}$};
\node (t2)  [color=black]             at (30,0)    {$\hat{\varphi}^{\xi_N,\lambda}_{2n-1,(n-1)(N-1)}$};
\node (d2) [color=black] at (15,10)   {$\hookrightarrow$};
\node (d2) [color=black] at (15,12)   {$\iota$};
\node (d2) [color=black] at (0,5)   {$\circlearrowleft$};
\node (d2) [color=black] at (30,5)   {$\circlearrowleft$};
\node (d2) [color=black] at (15,5)   {$\equiv$};
\end{tikzpicture}
\end{center}
Pursuing the same argument as the one for the coloured Jones invariant (from subsection \ref{SCJ:2}), we conclude that we can obtain the coloured Alexander invariant through the weight spaces from the Verma module, using the extended evaluation and coevaluation (presented in definition \ref{D:5'}). 
\begin{equation}\label{eq:24} 
\begin{aligned}
 \Phi_N(L,\lambda)
& ={\xi_N}^{(N-1)\lambda w(\beta_n)} p \ \circ \left(Id\otimes \tilde{\ev}^{\otimes {n-1}}_{f,\xi_N}{|_{\eta_{\xi_N,\lambda}} }\right) \\
& \circ \hat{\varphi}^{\xi_N,\lambda}_{2n-1,(n-1)(N-1)} \left( \beta_{n} \cup {\mathbb I}_{n-1} \right) \circ \left(\Id \otimes \tilde{{\tcoev}}_f^{\otimes {n-1}}|_{\eta_{\xi_N,\lambda}}\right) (v_0).
\end{aligned}
\end{equation}
\subsection{Step IV-Twisting the identification} 
Now, we aim to pass to the homological side. We will use the discussion from section \ref{SCJ:3}. We choose the function $g$ as in definition \ref{D:6} and using Corollary \ref{C:2}, we have following correspondence over $\Z[q^{\pm1},s^{\pm 1}]$:
\begin{equation}
\begin{aligned}
 v_0 \otimes v_{i_1}\otimes ... \otimes v_{i_{n-1}} \otimes v_{N-1-i_{n-1}}\otimes ... \otimes v_{N-1-i_1}\\ \longleftrightarrow^{\color{blue} \Theta' \circ g}  \ \ \ \ \ \ \ \ \ \ \ \ \tilde{\U}_{(0,i_1,...,i_{n-1}, N-1-i_{n-1},...,N-1-i_{1})} 
 \end{aligned}
\end{equation}
%Similar to the previous case, this identification holds for the actions over $\tilde{\Li}_N$. However, having in mind that we close up with the cups, we notice that the elements that are evaluated non-trivially are defined over $\Li$, so in the end we can work over $\Li$.
From this identification together with the formula from equation \eqref{eq:24}, we obtain the following description:
\begin{equation}\label{eq:14}
\begin{aligned}
& \Phi_N(L,\lambda)= ^{Cor \ref{C:2}} {\xi_N}^{(N-1)\lambda w(\beta_n)} \ p \circ {\color{green} \left(Id\otimes 
\tilde{\ev}^{\otimes {n-1}}_{f,\xi_N} {|_{\eta_{\xi_N,\lambda}} }\right)} \\
& \hspace{14mm} \circ {\color{blue} ( \Theta' \circ g)^{-1}} \circ L_{2n-1,(n-1)(N-1)}|_{{\psi}_{\xi_N,\lambda}} (\beta_{n} \cup {\mathbb I}_{n-1} ) \circ  {\color{blue} ( \Theta' \circ g )}\\
& \hspace{12mm}\left( \sum_{i_1,...,i_{n-1}=0}^{N-1}v_0 \otimes v_{i_1}\otimes ... \otimes v_{i_{n-1}} \otimes v_{N-1-i_{n-1}}\otimes ... \otimes v_{N-1-i_1} \right)=\\
& \hspace{12mm}= {\xi_N}^{(N-1)\lambda w(\beta_n)} \ p \circ {\color{green} \left(Id\otimes \tilde{\ev}^{\otimes {n-1}}_{f,\xi_N}{|_{\eta_{\xi_N,\lambda}} }\right)} \circ {\color{blue}( \Theta' \circ g)^{-1}} \circ \\
& \hspace{14mm} \circ  L_{2n-1,(n-1)(N-1)}|_{\psi_{\xi_N,\lambda}} (\beta_{n} \cup {\mathbb I}_{n-1} ) \\
&  \hspace{18mm}\left( \sum_{i_1,...,i_{n-1}=0}^{N-1} \tilde{\U}_{(0,i_1,...,i_{n-1}, N-1-i_{n-1},...,N-1-i_{1})} \right).
\end{aligned}
\end{equation}
\subsection{Step V-Homological correspondent of the coevaluation} Moving the coevaluation on the topological side, which corresponds to the submanifold $\mathscr P_n^N$ we obtain:
\begin{equation}
\begin{aligned}
\Phi_N(L,\lambda)
& ={\xi_N}^{(N-1)\lambda w(\beta_n)} p \ \circ {\color{green} \left(Id\otimes \tilde{\ev}^{\otimes {n-1}}_{f,\xi_N} {|_{\eta_{\xi_N,\lambda}} }\right)} {\color{blue} ( \Theta' \circ g)^{-1}} \circ  \\
& \hspace{8mm} \circ L_{2n-1,(n-1)(N-1)} |_{\psi_{\xi_N,\lambda}}(\beta_{n} \cup {\mathbb I}_{n-1} ) \ { \color{red} \mathscr P_n^{N}}.
\end{aligned}
\end{equation}
\begin{figure}[H]
$ \color{red} \mathscr F_n^{\xi_N} \ \ \ =  \ \  \ \ \ \ \  s^{(1-N)(n-1)} \  q^{-2 \sum_{k=1}^{n-1} \ i_k} \ \ \ \ \ \ \ \ \ \ \ \ \ \ \ \ \ \ \ \ \ \ \ \ \ \ \  \ \ \ \ {\color{green} \mathscr G_{n}^N}  $\\
$ { \ \ \ \ \ \ \ \ \ \ \ \ \ \ \ \ \ \ \ \ \ \ \ \  \ \ }  { \color{red}  \nwarrow \textit{  deformation} } $\\
${\color{red} \Huge \mathscr P_n^N= \sum_{i_1,...,i_{n-1}=0}^{N-1} \ \  \ \ \tilde{\mathscr U}_{0,i_1,...,i_{n-1},N-1-i_{n-1},...,N-1-i_{1}}} \ \ \ \ \ \ \ \ \ \ \ \ \ \ \ \ \ \ \ \ \ \ \ \ \ 
$

\

\centering
\includegraphics[scale=0.5]{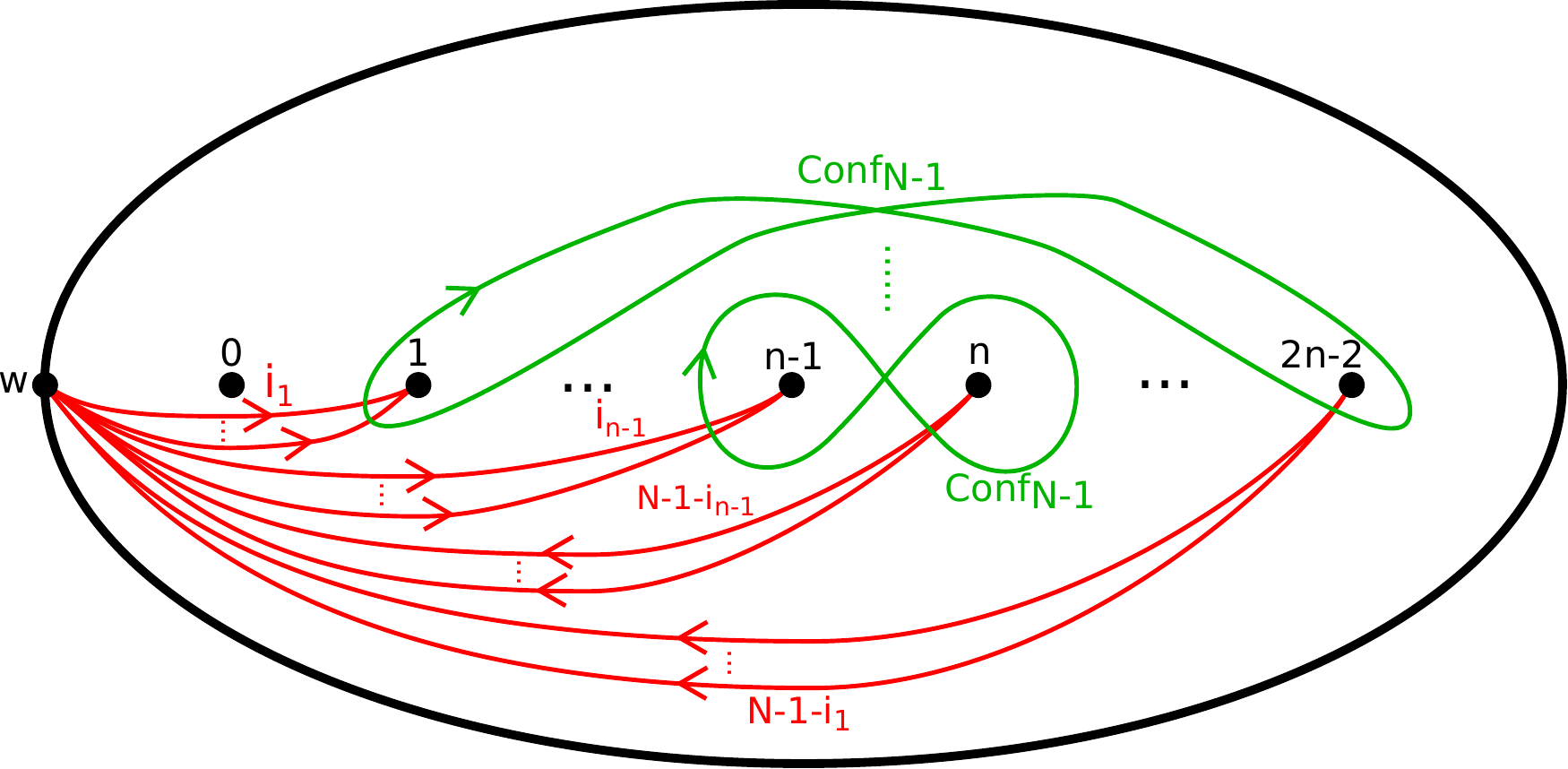}
\caption{}
\end{figure}

\subsection{Step VI-Intersection formula at roots of unity}

In this part, we will transport the evaluation to the toplogical side. 
Looking at the evaluation at roots of unity and using remark \ref{R:9} we have:

\begin{equation}
\begin{aligned}
& p\circ \left( Id \otimes \tilde{\ev}^{\otimes {n-1}}_{f, \xi_N} {|_{\eta_{\xi_N,\lambda}} }\right): \hat{V}^{\xi_N,\lambda}_{2n-1,(n-1)(N-1)} \rightarrow \C \\
& p \circ \left( Id \otimes \tilde{\ev}^{\otimes {n-1}}_{f, \xi_N} {|_{\eta_{\xi_N,\lambda}} }\right)(w)= \\
& =\begin{cases}
\eta_{\xi_N,\lambda} \circ p_n^{\xi_N}(i_1,...,i_{n-1}), \ \ if  \ \ \ 0 \leq i_1,...,i_{n-1}\leq N-1 \ \ and \\
 \ \ \ \ \ \ \ \ \ \ \ \ \ \ \ w=  v_0 \otimes v_{i_1}\otimes ... \otimes v_{i_{n-1}} \otimes v_{N-1-i_{n-1}}\otimes ... \otimes v_{N-1-i_1}\\
0, \ \ \ \ \ \ otherwise.
\end{cases}
\end{aligned}
\end{equation}
Now, we remind the relation presented in proposition \ref{P:8}:
\begin{equation}
\begin{aligned}
\mathscr F_n^{\xi_N}|_{\eta_{\xi_N,\lambda}} =\sum_{i_1,...,i_{n-1}=0}^{N-1} \eta_{\xi_N,\lambda} \circ ~& p_n^{\xi_N}(i_1,...,i_{n-1}) \cdot \\ & \cdot \tilde{\mathscr U}_{0,i_1,...,i_{n-1},N-1-i_{n-1},...,N-1-i_{1}}|_{\psi_{\xi_N,\lambda}}.
\end{aligned}
\end{equation}
\begin{remark}
This property uses extensively the properties of the specialisation $\eta_{\xi_N,\lambda}$, and in contrast to the coloured Jones case, this relation does not hold over two variables.
\end{remark}
Using a similar argument concerning the intersection pairing as the one from step \ref{SCJ:5}, we obtain the following:
\begin{equation}\label{E:5}
\begin{aligned}
& \Phi_N(L,\lambda)=  {\xi_N}^{(N-1)\lambda w(\beta_n)} \cdot \hspace{-3mm} \sum_{i_1,...,i_{n-1}=0}^{N-1} \hspace{-3mm} < \left(\beta_{n} \cup {\mathbb I}_{n-1} \right) { \color{green} p_n^{\xi_N}(i_1,..,i_{n-1})} \\
& \hspace{36mm}{ \color{red}\tilde{\mathscr U}_{0,i_1,...,i_{n-1},N-1-i_{n-1},...,N-1-i_{1}}|_{\gamma}},{\color{green} \mathscr G_n^N|_{\gamma}}> |_{\eta_{\xi_N,\lambda}}=\\
&\hspace{14mm} = {\xi_N}^{(N-1)\lambda w(\beta_n)} \cdot  < \left(\beta_{n} \cup {\mathbb I}_{n-1} \right) \sum_{i_1,...,i_{n-1}=0}^{N-1} { \color{green}   p_n^{\xi_N}(i_1,..,i_{n-1})} \cdot \\
& \hspace{35mm} \cdot { \color{red}\tilde{\mathscr U}_{0,i_1,...,i_{n-1},N-1-i_{n-1},...,N-1-i_{1}}|_{\gamma}},{\color{green} \mathscr G_n^N|_{\gamma}}> |_{\eta_{\xi_N,\lambda}}=\\
& \ \ \ \ \ \ \ \ \ \ \ =^{Prop \ref{P:8}} {\xi_N}^{(N-1)\lambda w(\beta_n)}  < (\beta_{n} \cup {\mathbb I}_{n-1} )  { \color{red}  \mathscr F_n^{\xi_N}}, {\color{green}\mathscr G_n^N|_{\gamma}}> |_{\eta_{\xi_N,\lambda}}.
\end{aligned}
\end{equation}
Further on, we look at the homology classes introduced in definition \ref{E:3}:
$$\mathscr F_n^{\xi_N}=s^{(1-N)(n-1)} \mathscr E_n^N|_{{\gamma}}.$$ 
This definition together with formula \eqref{E:5} and the property concerning the relation between specialisations from diagram \ref{fig:7} $$\psi_{\xi_N,\lambda}=\eta_{\xi_N,\lambda} \circ \gamma $$ lead to the following formula: 
\begin{equation}
\Phi_{N}(L,\lambda)={\xi_N}^{(N-1)\lambda w(\beta_n)} \cdot {\xi_N}^{\lambda (1-N)(n-1)} \ <(\beta_{n} \cup {\mathbb I}_{n-1} ) \ { \color{red} \mathscr E_n^{N}}, {\color{green} \mathscr G_n^N}> |_{\psi_{\xi_N,\lambda}}. 
\end{equation}
This concludes the intersection model for the family of the coloured Alexander invariants and also Theorem \ref{THEOREM}.
\clearpage
\section{Coloured Alexander invariants from $\Z\oplus \Z_N$-converings} \label{S:8}
This section concerns the sequence of coloured Alexander invariants. We aim to show that they come directly from a certain covering of the configuration space, without further specialisations. 

\

\noindent{\bf Construction of the covering space} \ We start with the local system from definition \ref{D:12} and equation \eqref{eq:23} that leads to the covering $\tilde{C}_{n,m}$: 
\begin{equation*}
\begin{aligned}
\phi: \pi_1(C_{n,m}) \rightarrow  \ & \Z \oplus \Z.\\
&\langle x \rangle \ \langle d\rangle
\end{aligned}
\end{equation*}
On the other hand, we have seen in Theorem \ref{THEOREM} that the topological model for the $N^{th}$ coloured Alexander polynomial corresponds to an intersection pairing specialised through $\psi_{\xi_N,\lambda}$. This motivates the following definition, where we consider a local system whose second component has finite order.
\begin{definition}

a) (New local system) Let $p_N:\Z \rightarrow \Z_N$ be the projection onto the finite group of order $N$. Let us consider the previous local system projected using this map, as below:
\begin{equation}
\begin{aligned}
&\phi_N: \pi_1(C_{n,m}) \rightarrow  \Z\oplus \Z \rightarrow \Z\oplus \Z_N\\
& \hspace{27mm}\langle x \rangle \ \langle d\rangle \ \ \  \langle s \rangle \langle d_N \rangle\\
&\phi_N=(2\cdot Id \oplus p_N) \circ  \  \phi. 
\end{aligned}
\end{equation}
b) (Covering) Let $C^{N}_{n,m}$ be the covering of $C_{n,m}$ which corresponds to the local system $\phi_N$.
Let us fix ${\bf d_N} \in C^{N}_{n,m}$ a lift of the base point $\bf d$.  
\end{definition}
Looking at the group ring corresponding to this new local system, we have the following:
$$\Z[\Z \oplus \Z_N]\simeq\Z[s^{\pm 1}, \xi_N^{\pm 2}]$$
where $\xi_N=e^{\frac{2\pi i}{2N}}$ as in the previous sections.
\begin{definition}(Homology of the covering space)

We consider the two types of homology of this covering space, defined in an analog manner as the ones from notation \ref{T2}.  More specifically, we define the following:
\begin{equation}
\begin{aligned}
H^N_{n,m}:=&H^{\text{lf},\infty,-}_m(C^N_{n,m}; \Z)\\ H^{\partial,N}_{n,m}:=&H^{lf, \Delta}_{m}(C^N_{n,m}, \partial^{-}; \Z).
\end{aligned}
\end{equation}
\end{definition}
Then, there is a corresponding intersection pairing: 
\begin{equation}\label{P:5}
< \ , \ >_N:H^N_{n,m} \otimes H^{\partial,N}_{n,m} \rightarrow \Z[s^{\pm 1}, \xi_N^{\pm 2}]
\end{equation}
whose method of computation is the one presented in equation \eqref{eq:1}, applied for the local system $\phi_N$.
\begin{definition}(Specialisations of truncated coefficients) Let us consider the following specialisations:
\begin{equation}
\begin{aligned}
& \begin{cases}
\gamma_N: \Z[x^{\pm 1}, d^{\pm 1}]\rightarrow \Z[s^{\pm1},{\xi_N}^{\pm2}]\\
\gamma_N(x)=s^2; \ \ \ \ \ \ 
\gamma_N(d)={\xi_N}^{-2}.
\end{cases}\\
& \begin{cases}
\eta^N_{\xi_N,\lambda}: \Z[s^{\pm1},{\xi_N}^{\pm2}]\rightarrow \C \\
\eta^N_{\xi_N,\lambda}(s)=\xi_N^{\lambda}.
\end{cases}\\
& \begin{cases}
\eta_N: \Z[s^{\pm 1}, q^{\pm 2}]\rightarrow \Z[s^{\pm1},{\xi_N}^{\pm2}]\\
\gamma_N(q)=\xi_N.
\end{cases}
\end{aligned}
\end{equation}
\end{definition}
\begin{figure}[H]
\begin{center}
\begin{tikzpicture}
[x=1.1mm,y=1.1mm]

% Nodes of the diagram
\node (b1)  [color=blue] at (-30,30)    {${\mathbb Z[x^{\pm1},d^{\pm1}]}$};
\node (b2) [color=orange] at (0,30)   {${\mathbb Z[s^{\pm1},{\xi_N}^{\pm2}]}$};
\node (b3) [color=red]   at (0,50)   {$\mathbb Z[s^{\pm1},q^{\pm2}]$};
%\node (b4') [color=brickred]  at (48,30)   {$\hookrightarrow$};
%\node (b4) [color=brickred]  at (55,30)   {$\Q(q,s)$};
\node (b5) [color=green]  at (0,2)   {$\boldsymbol{\mathbb C}$};
%\node (b5') [color=green]  at (0,-5)   {$\Z[q^{\pm 1}]$};

\node (b6) [color=blue]  at (-22,27)   {$\boldsymbol {x}$};
\node (b6') [color=blue]  at (-22,24)   {$\boldsymbol {d}$};
\node (b7) [color=orange]  at (-7,27)   {$\boldsymbol {s^{2}}$};
\node (b7') [color=orange]  at (-7,24)   {$\boldsymbol {\xi_N^{-2}}$};
\node (b8') [color=orange]  at (15,20)   {$\boldsymbol {s}$};
\node (b9') [color=green]  at (15,5)   {$\boldsymbol {\xi_N^{\lambda}}$};
\node (b8'') [color=red]  at (15,48)   {$\boldsymbol {q}$};
\node (b9'') [color=green]  at (15,32)   {$\boldsymbol {\xi_N}$};

\draw[->,color=blue, dashed]   (b6)      to node[right,font=\small, yshift=3mm]{}                           (b7);
\draw[->,color=blue,dashed]   (b6')      to node[right,font=\small, yshift=3mm]{}                           (b7');
\draw[->,color=blue]   (b1)      to node[right,font=\small, yshift=3mm]{${{\gamma}_N}$}                           (b2);

\draw[->,color=red]             (b3)      to node[right,font=\small, yshift=3 mm]{$\eta_N$ }   (b2);
\draw[->, color=blue]             (b1)     to node[right,yshift=-3mm,xshift=-5mm,font=\small] [xshift=-2mm,yshift=7mm]{$ {\gamma}$}   (b3);
\draw[->,color=orange,thick]   (b2)      to node[right,font=\small]{${\eta^N_{\xi_N,\lambda}}$}                        (b5);
\draw[->,color=green]   (b1)      to node[left,font=\small]{${\psi_{\xi_N,\lambda}}$}                        (b5);
%\draw[->,color=orange, dashed]   (b8)      to node[right,font=\small, yshift=3mm]{}                           (b9);
\draw[->,color=orange,dashed]   (b8')      to node[left,font=\small, xshift=50mm,yshift=3mm]{}                           (b9');
\draw[->,color=red,dashed]   (b8'')      to node[left,font=\small, xshift=50mm,yshift=3mm]{}                           (b9'');
\draw[->,color=red,dashed]   (20,50)      to[in=60,out=-60] node[left,font=\small, xshift=10mm,yshift=3mm]{$\eta_{\xi_N,\lambda}$}                           (20,2);

\end{tikzpicture}
\end{center}
\caption{Specialisations}
\label{fig:8}  
\end{figure}
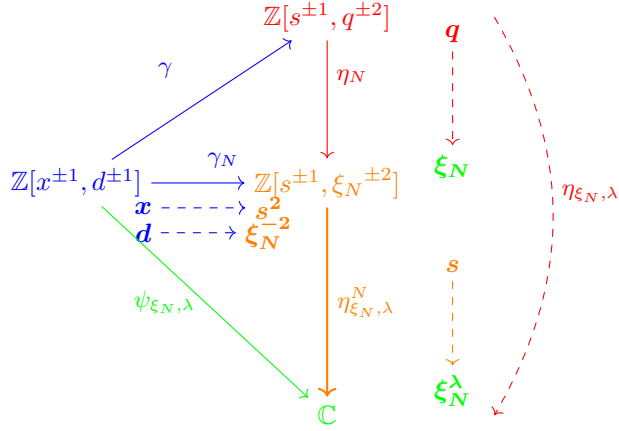

\begin{lemma}(Relations between intersection forms)\label{L:3}
Let $M_1,M_2\subseteq C_{n,m}$ two immersed submanifolds of dimension  $m$ whose lifts in $C_{n,m}^N$ satisfy the requirements from definition \ref{D:55}. Also, let us consider two paths $\gamma_{M_1},\gamma_{M_2}:[0,1]\rightarrow C_{n,m}$ such that:
\begin{equation}
\begin{cases}
\gamma_{M_1}(0)={\bf d}; \gamma_{M_1}(1) \in M_1\\ 
\gamma_{M_2}(0)={\bf d}; \gamma_{M_2}(1) \in M_2.
\end{cases}
\end{equation}

a) Let $\tilde{\gamma}_{M_1},\tilde{\gamma}_{M_2}:[0,1]\rightarrow \tilde{C}_{n,m}$ the corresponding lifts in $\tilde{C}_{n,m}$ through $\bf \tilde{d}$. Moreover, we denote by $$\tilde{M}_1,\tilde{M}_2\subseteq \tilde{C}_{n,m}$$
the lifts of $M_1,M_2$ through the points $\tilde{\gamma}_{M_1}(1)$ and $\tilde{\gamma}_{M_2}(1)$ respectively.

b) Similarly, let $\gamma^N_{M_1},\gamma^N_{M_2}:[0,1]\rightarrow C^N_{n,m}$ be the corresponding lifts through $\bf d_N$. Further on, we consider 
$$M^N_1,M^N_2\subseteq C^N_{n,m}$$
to be the lifts of $M_1,M_2$ through the points $\gamma^N_{M_1}(1)$ and $\gamma^N_{M_2}(1)$ respectively.
Then, we have the following property:
\begin{equation}
<[M^N_1],[M^N_2]>_N~=~<[\tilde{M}_1],[\tilde{M}_2]>_{\gamma_N}.
\end{equation}
\end{lemma}
\begin{proof}
We remind that the intersection pairing in the covering space is encoded by the geometric intersections in the base configuration space together with the corresponding local system. Its precise formula is presented in equation \eqref{eq:1}. This leads to the following descriptions:
\begin{equation}
\begin{aligned}
 <[M_1^N],[M_2^N]>_N & =\sum_{x \in M_1\cap M_2} \alpha_x \cdot \phi_N(l_x) \in \Z[s^{\pm1}, {\xi_N}^{\pm2}].\\
 <[\tilde{M}_1],[\tilde{M}_2]>  \hspace{3mm} & =\sum_{x \in M_1 \cap M_2} \alpha_x \cdot \phi(l_x) \in \Z[x^{\pm1}, d^{\pm1}].
\end{aligned}
\end{equation}
This shows that the two expressions from the statement are given by the same intersections in the base configuration space, and the only difference occurs at the evaluation of the local system. Further on, we notice that $$\gamma_N \circ \phi=\phi_N,$$ if we see the image of $\phi$ and $\phi_N$ in the corresponding group rings, namely $\Z[x^{\pm1},d^{\pm1}] $ and $\Z[s^{\pm1}, {\xi_N}^{\pm2}]$ respectively. 
This remark leads to the equality from the statement and concludes the proof.
\end{proof}

\begin{definition}(Homology classes) Let us define the homology classes
$$\bar{\mathscr E}_n^{\xi_N} \in H^N_{2n-1,(n-1)(N-1)} \text{  and  } \ \  \bar{\mathscr G}_{n}^N \in H^{\partial,N}_{2n-1,(n-1)(N-1)}$$ given by the lifts of the same geometric submanifolds as the ones from definition \ref{fhc} and \ref{fhc}. We follow the same lifting procedure as the one presented in notation \ref{N:2} and definition \ref{D:10} for the paths from figure \ref{fig5}, this time using the point $\bf d_N$ and the covering $C^N_{2n-1,(n-1)(N-1)}$.
\end{definition}
In the following, we show that these homology classes lead to the coloured Alexander polynomials, as presented in Corollary \ref{C:Alex}.
\begin{proof}
Following the model from the previous sections, we have:
\begin{equation}
\begin{aligned}
\Phi_{N}(L,\lambda)=^{Th \ref{THEOREM}}{\xi_N}^{(N-1)\lambda w(\beta_n)} \cdot \ & {\xi_N}^{\lambda (1-N)(n-1)} \cdot \\
 &\cdot <(\beta_{n} \cup {\mathbb I}_{n-1} ) \ { \color{red} \mathscr E_n^{N}}, {\color{green} \mathscr G_n^N}> |_{\psi_{\xi_N,\lambda}}=\\
=^{Lemma \ref{L:3}}{\xi_N}^{(N-1)\lambda w(\beta_n)} \cdot \ & {\xi_N}^{\lambda (1-N)(n-1)} \cdot \\
& \cdot <(\beta_{n} \cup {\mathbb I}_{n-1} ) \ { \color{red} \bar{\mathscr E}_n^{\xi_N}}, {\color{green} \bar{\mathscr G}_n^N}> _{N}|_{s=\xi_N^\lambda}.
\end{aligned}
\end{equation}
The last equality concludes the proof of this model.
\end{proof}
\
In the literature, the coloured Alexander invariants are often seen as Laurent polynomials in the variable $\xi_N^{\lambda}$. This corresponds exactly to the variable $s$ from the local system $\phi_N$, which appears in the pairing from equation \eqref{P:5}. Following the same argument as above, we conclude a topological model using this variable.
\begin{corollary}(Coloured Alexander invariants via the pairing $< ,  >_N$)
\begin{equation}
\Phi_{N}(L,s)={s}^{(N-1) w(\beta_n)} \cdot {s}^{ (1-N)(n-1)} \cdot <(\beta_{n} \cup {\mathbb I}_{n-1} ) \ { \color{red} \bar{\mathscr E}_n^{\xi_N}}, {\color{green} \bar{\mathscr G}_n^N}> _{N}.
\end{equation}
\end{corollary}

\
\
\url{https://www.maths.ox.ac.uk/people/cristina.palmer-anghel}  

\end{document}